\documentclass[11pt,a4paper,twoside]{article}
\usepackage{bm, amsmath, amssymb, amsthm} %, amsthm

 \newtheorem{Theorem}{Theorem}%[section]
 \newtheorem{Corollary}{Corollary}
 \newtheorem{Lemma}{Lemma}
 \newtheorem{Proposition}{Proposition}
 \newtheorem{Definition}{Definition}
 \newtheorem{Remark}{Remark}
 \newtheorem{Example}{Example}
\newcommand\I{\mathfrak{T}}

\newcommand\tT{{\tilde T^\sharp}}
\newcommand\tH{{\tilde H}}

\def\dt{\partial_t}
\def\T{T^\sharp}
\newcommand\mD{{\cal D}}

\def\<{\langle}
\def\>{\rangle}
\def\RR{\mathbb{R}}
\def\eps{\varepsilon}

\newcommand\tr{\operatorname{Tr}}
\newcommand\Div{\operatorname{div}}
\newcommand\id{\operatorname{id}}

\def\Ric{\mathcal Ric}
\def\vol{\operatorname{vol}}

\def\eq{\hspace*{-2mm}&=&\hspace*{-2mm}}
\def\plus{\hspace*{-2mm}&+&\hspace*{-2mm}}

\author{Vladimir Rovenski\footnote{Department of Mathematics, University of Haifa, Mount Carmel, 3498838 Haifa,  Israel.
 \newline E-mail: {\tt vrovenski@univ.haifa.ac.il} }
 \ \ and \
 Tomasz Zawadzki
 \footnote{%Uniwersytet \L\'{o}dzki, ul. Banacha 22, 90-238 \L\'{o}d\'{z}, Poland.
 Faculty of Mathematics and Computer Science, University of Lodz, Banacha 22, 90-238 Lodz, Poland.
 \newline E-mail: {\tt tomasz.zawadzki@wmii.uni.lodz.pl}}
 }

\title{Variational problem on a metric-affine almost product manifold}
\begin{document}

\date{}

\maketitle

\begin{abstract}
We study a variational problem on a smooth manifold with a decomposition of the tangent bundle into $k>2$ subbundles (distributions), namely, we consider the integrated sum of their mixed scalar curvatures as a functional of adapted pseudo-Riemannian metric (keeping the pairwise orthogonality of the distributions) and contorsion tensor, defining a linear connection. This functional allows us to generalize the class of Einstein metrics in the following sense:
if all of the  distributions are one-dimensional, then it coincides with the geometrical part of the Einstein-Hilbert action restricted to adapted metrics.
We prove that metrics in pairs metric-contorsion critical for our functional make all of the distributions totally umbilical.
We obtain examples and obstructions to existence of those critical pairs in some special cases: twisted products with statistical connections; semi-symmetric connections and 3-Sasaki manifolds with metric-compatible connections.


\end{abstract}

\tableofcontents

\section{Introduction}

\textbf{State of the art}.
% On~a manifold equipped with an additional (e.g., contact or almost product) structure, see \cite{bf,g1967},
%one can consider an analogue of Einstein-Hilbert action adjusted to that structure.
%The geometrical part of the Einstein-Hilbert action is the \textit{total scalar curvature}.
%$J : (g,\I)\mapsto\int_{M} \overline{\rm S}\ {\rm d}\vol_g $.
%Variational problems for metric-affine spaces are less studied in comparison with Riemannian case.
This paper links together two topics of differential geometry:
variational problems for metric and linear connection and almost product manifolds with $k\ge2$ factors.

Many canonical geometrical objects
%(Riemannian metrics, connections) appear as
are critical points of nonlinear variational problems, see~\cite{catino}. A~particularly famous of them is
the {integrated scalar curvature} (the geometrical part of Einstein-Hilbert action) with variable Riemannian metric and linear connection,
% $\int_{M} \overline{\rm S}\ {\rm d}\vol_g$;
its Euler-Lagrange equations are the {Einstein equation} and the spin-connection equation,
%underlying the Einstein-Cartan theory,
e.g.,~\cite{ap}.
%Another~Euler-Lagrange equation for this action, when a linear connection changes,
%is an algebraic constraint relating the torsion tensor with the spin tensor, e.g.,~\cite[Chapter~17]{ap}.
Both equations form the basis of the Einstein-Cartan theory of gravity within the framework of metric-affine geometry,
%%%%%% MOVED
%The metric-affine geometry,
%founded by E.\,Car\-tan,
that considers a pseudo-Riemannian metric $g$ and a linear connection $\bar\nabla$
(instead of the Levi-Civita connection $\nabla$) on a manifold as independent variables, e.g.,~\cite{bf,mikes}.
The~following classes of metric-affine manifolds are popular:
%have various applications:

$\bullet$~{Statistical manifolds}, where the (0,3)-tensor $\bar\nabla g$ is symmetric and $\bar\nabla$ is torsion-free,
%(these conditions are equivalent to $\I^\wedge=\I$ and $\I^* = \I$ or symmetry of the cubic form $A(X,Y,Z):=\<\I_XY,\,Z\>$),
%constitute an important class of metric-affine manifolds
are important for probability and statistics as well as information geometry, e.g., \cite{Amari2016}.
%The theory of affine hypersurfaces in $\RR^{n+1}$ is a natural source of such manifolds.
%The above class of connections admits a natural definition of the sectional curvature: in case of metric connections by the same formula
%as for the Levi-Civita connection, and for statistical connections by the analogue introduced in \cite{op2016}.

$\bullet$~{Riemann-Cartan manifolds}, where $\bar\nabla$ is {metric-compatible}, i.e., $\bar\nabla g=0$,
%the $\bar\nabla$-parallel transport preserves the metric,
are important for theoretical physics, e.g., \cite{ap}.
%This is equivalent to $\I^*=-\I$, or the equality $A(X,Y,Z)=A(X,Z,Y)$, and $\bar \nabla$ is then called metric compatible (or, metric connection).
%e.g.,~\cite{cb19}, where the torsion tensor is involved in the {spin connection equation} of Einstein-Cartan theory.
Semi-symmetric connection, introduced by K.\,Yano \cite{Yano}, is a special case of a metric-compatible connection parameterized by a vector field.
%%%%%

\smallskip

In \cite{r-EH-k}, the first author studied the variation problem with mixed scalar curvature
\begin{equation}\label{Eq-Smix-g0}
 J_\mD : g\mapsto\int_{M} {\rm S}_{\,\mD_1,\ldots,\mD_k}\,{\rm d}\vol_g
\end{equation}
on a manifold with $k\ge2$ distributions (see \cite{bdrs,rz-2,rz-connections,rz-3} for $k=2$).
%Namely, the mixed scalar curvature instead of the scalar curvature
This analog of scalar curvature is the averaged sectional curvature of all planes spanned by two unit vectors from different distributions,
and is one of the simplest curvature invariants of an almost product manifold,
%Its study led to many results, such as the existence of foliations with prescribed geometry, integral formulas and splitting theorems,
see \cite{r-affine,RS-1,Rov-Wa-2021,Walczak}.
%%%%%%%%
Recall that a connected
%smooth
manifold $M$ with a decomposition of the tangent~bundle into $k\ge2$ subbundles (distributions),
\begin{equation}\label{Eq-mD-k-TM}
 TM=\mD_1 + \ldots + \mD_k
\end{equation}
 is called an \textit{almost product manifold} \cite{RS-1}
%(also, an \textit{almost product manifold}
(see \cite{bf} for $k=2$).
%An~almost product structure \eqref{Eq-mD-k-TM} %with $k=3$
It appears in such topics as
%\noindent\
%1)
multiply twisted or warped product manifolds, see \cite{Dimitru,MRS-99};
%, wang}. %[Section~3.6]{chen-2017}.
%\noindent\
%2)
the theory of nets composed of foliations (that is, integrable distributions),~see~\cite{Krynski,RS-99};
%, i.e., composed of 3 generic foliations, see~\cite{tol}.
%\noindent\
%3)
%almost
para-$f$-manifolds, see \cite{tar};
%\noindent\
%4)
lightlike manifolds, i.e., with degenerate metric of constant rank and index, see~\cite{dug};
%\noindent\
%5)
hypersurfaces in space forms with $k$ distinct principal curvatures, see~\cite{cecil-ryan}.
%{ots-86}.
%\noindent\
%5) tubes over standard embeddings of a projective plane $FP^2$, for $F = \RR; \mathbb{C};\mathbb{H}$ or $\mathbb{O}$ (Cayley numbers), in $S^4; S^7; S^{13}$, %and $S^{25}$, respectively, see~\cite{cecil-ryan}.
%\end{Remark}
An almost product manifold
%\eqref{Eq-mD-k-TM}
admits a natural class of metrics that will be of our interest in this paper.
A pseudo-Riemannian metric $g$ on $(M,\mD_1,\ldots,\mD_k)$ is \textit{adapted}
(or, compatible, see \cite{RN-21}) if
all distributions are non-degenerate and pairwise orthogonal.
% with respect to $g$.
Any adapted metric is uniquely decomposed as $g=g_1\oplus\ldots\oplus g_k$,
where $g_\mu$ is a bundle metric on $\mD_\mu$, %and
then we write $T M = \bigoplus_{\mu} \mD_\mu$.
A~special family of adapted metrics are multiconformally equivalent (to $g$) metrics
%\[
 $\tilde g=u_1^2\,g_1\oplus\ldots\oplus u_k^2\,g_k$,
%\]
where $u_\mu:M\to\RR$ are smooth functions without zeros, see~\cite{RN-21}.
%\begin{Remark}\rm

%On a manifold with a decomposition \eqref{Eq-mD-k-TM}, an analogue of the Einstein-Hilbert action,
%where the scalar curvature is replaced by the mixed scalar curvature ${\rm S}_{\,\mD_1,\ldots,\mD_k}$ of the Levi-Civita connection $\nabla$,
%was studied in \cite{r-EH-k} (see also \cite{bdrs,rz-2,rz-connections,rz-3} for $k=2$):
%\begin{equation}\label{Eq-Smix-g0}
% J_\mD : g\mapsto\int_{M} {\rm S}_{\,\mD_1,\ldots,\mD_k}\,{\rm d}\vol_g \,.
%\end{equation}

\smallskip

\textbf{Objectives and results}.
We
%consider an almost product manifold $(M,\mD_1,\ldots,\mD_k)$ with an adapted metric $g$ and a linear connection $\bar\nabla$ and
study critical points of the {integrated mixed scalar curvature} on $(M,\mD_1,\ldots,\mD_k)$, depending on adapted metric $g$
and {contorsion tensor} $\I=\bar\nabla-\nabla$,
\begin{equation}\label{Eq-Smix-g}
 \bar J_\mD : (g,\I)\mapsto\int_{M} \overline{\rm S}_{\,\mD_1,\ldots,\mD_k}\,{\rm d}\vol_g \,.
\end{equation}
The mixed scalar curvature $\overline{\rm S}_{\,\mD_1,\ldots,\mD_k}$ is up to factor 2, see \eqref{E-Dk-Smix}, the sum of the mixed scalar curvatures
$\overline{\rm S}_{\,\mD_\mu,\mD^\bot_\mu}$, examined in \cite{rz-3}, and
if all distributions are 1-dimensional, the mixed scalar curvature reduces to the scalar curvature divided by two. %,
%which additionally relates our variation problem \eqref{Eq-Smix-g} to the Einstein-Hilbert action.
If $M$ is a non-compact manifold, we integrate in \eqref{Eq-Smix-g} %instead
over an arbitrarily large, relatively compact domain $\Omega\subset M$, which contains supports of variations of $g$ and $\I$.
We find the Euler-Lagrange equation for \eqref{Eq-Smix-g} with fixed $\I$ for adapted variations of metric preserving the volume of the manifold. %,
%i.e., $\delta_g \bar J_\mD = \lambda\,g$,
%and present it in
It can be presented in the form of the Einstein~equation:
%restricted to each distribution:
\begin{equation}\label{E-geom}
%\big(  \overline\Ric_{\,\mD} - (1/2)\,\overline{\cal S}_{\mD}\cdot g + \lambda\,g \big)_{ | \mD_\mu \times \mD_\mu} = 0\,, \quad \mu = 1, \ldots , k,
 \overline\Ric_{\,\mD} - (1/2)\,\overline{\cal S}_{\mD}\cdot g + \lambda\,g = 0, %\qquad
\end{equation}
where the Ricci tensor of $\bar\nabla$ and the scalar curvature are replaced by the Ricci type tensor %$\overline\Ric_{\,\mD}$
$\overline\Ric_{\,\mD} = \bigoplus\nolimits_{\mu=1}^k \overline\Ric_{\,\mD\,|\,\mD_\mu \times \mD_\mu}$
(introduced in \cite{bdrs} for $k=2$, $\dim M=4$, $\dim \mD_1=1$
%, general variations of metric
and $\bar\nabla=\nabla$)
and its~trace $\overline{\cal S}_{\mD}$.
Although $\overline\Ric_{\,\mD}$ has complicated form even for $k=2$, see \cite{rz-3},
%Although $\overline\Ric_{\,\mD}$ has a
we write it explicitly in special cases: for statistical and semi-symmetric connections, %as well as for
and twisted~products; if all distributions are one-dimensional, it reduces to %the restriction of
the Ricci tensor of ${\bar \nabla}$. %: ${\overline \rm Ric}_{\mD_\mu \times \mD_\mu}$.

We find the Euler-Lagrange equation for \eqref{Eq-Smix-g} with a fixed adapted metric for variations of $\I$,
%(e.g., for statistical and metric-compatible connections), i.e., $\delta_\I\bar J_\mD = 0$,
%which can be regarded as an analogue of the Cartan spin-connection equation.
%
which can be decomposed into independent equations -- some of which do not contain contorsion tensor, but significantly restrict metrics admitting critical contorsions.
Due to such restrictions, a natural setting to consider are twisted and warped products of manifolds, on which we characterize all critical pairs with statistical connections, in terms of contorsion tensor, mean curvatures of distributions and Ricci tensor of the metric.
We also prove the absence of non-trivial critical points of the action on e.g., %warped
%or twisted products of manifolds of certain dimensions with statistical connections,
harmonic distributions with semi-symmetric connections,
 or
%for different reasons,
complete 3-Sasaki manifolds with particular metric connections.
On the other hand,
considering only connections from certain families allows to reduce the variational problem to purely pseudo-Riemannian one.
In particular, we show that some critical pairs $(g,\I)$ of the action \eqref{Eq-Smix-g} %with respect
restricted to adapted %variations of
metrics and %variations of \,$\I$ corresponding to
statistical connections can be obtained from critical adapted metrics of this action
%\eqref{Eq-Smix-g0}.
with the fixed Levi-Civita connection.
Similarly, considering critical points among semi-symmetric connections only slightly modifies %$\delta_g \bar J_\mD$
the tensor $\overline\Ric_{\,\mD}$.
% in the Euler-Lagrange equation for the metric.

The action \eqref{Eq-Smix-g}
%dependent on distributions
may find applications in the theory of nets and theoretical physics, because it
%the action we consider
is strongly related to the Einstein-Hilbert action.
Indeed,
%using the action
\eqref{Eq-Smix-g}
and the relation
$\overline{\rm S}=2\,\overline{\rm S}_{\,\mD_1,\ldots,\mD_k}+\sum\nolimits_{\,\mu=1}^{\,k} \overline{\rm S}({\mD_\mu})$
of $\overline{\rm S}_{\,\mD_1,\ldots,\mD_k}$ with the scalar curvature $\overline{\rm S}$
%= \tr_g\overline{\rm Ric}$
(the trace of the Ricci tensor for~$\bar\nabla$)
and the scalar curvature $\overline{\rm S}({\mD_\mu})$ of~$\mD_\mu$,
%\[
%\label{E-geom-2}
% \overline{\rm S}=2\,\overline{\rm S}_{\,\mD_1,\ldots,\mD_k}+\sum\nolimits_{\,\mu=1}^{\,k} \overline{\rm S}({\mD_\mu})\,,
%\]
%we can %extend
%conclude that
allow us to generalize
%to
the class of Einstein metrics, i.e.,
in the case of one-dimensional distributions any critical pair $(g, \I)$ for the Einstein-Hilbert action %in vacuum
is also critical for~\eqref{Eq-Smix-g}.
However, since we consider variations of metric adapted to an almost-product structure, %even in this case
we obtain also critical pairs $(g, \I)$ with non-Einstein metrics.
In particular, on manifolds $(M,g)$ of constant scalar curvature and certain dimension,
we find decompositions \eqref{Eq-mD-k-TM} that make $g$ critical for the action \eqref{Eq-Smix-g}.
%In addition,
%On the other hand, considering adapted variations, we find critical pairs $(g, \I)$ with non-Einstein metrics.
%using adapted variations of metric, we find critical pairs $(g, \I)$ with non-Einstein metrics.
%
%Also,
Finally, \eqref{Eq-Smix-g} can be combined with the Einstein-Hilbert action in vacuum as
%\[
 $\bar J_{\mD,\epsilon}: (g,\I) \mapsto \int_{M} (\overline{\rm S}+\epsilon\,\overline{\rm S}_{\,\mD_1,\ldots,\mD_k} )\,{\rm d}\vol_g\
 (\epsilon\in\RR)$,
%\]
and considered as a perturbation of the Einstein-Cartan theory.
%Here, the Einstein-Hilbert action corresponds to the case $\epsilon=0$.

\smallskip

\textbf{Structure of the article}.
The article consists of an Introduction and six sections.
Section~\ref{sec:prel} contains necessary results from \cite{rz-2,rz-connections},
among them the notion of the mixed scalar curvature is~central.
In Section~\ref{sec:adapted-metric}, we study adapted variations of metric (with fixed contorsion tensor)
and find the Euler-Lagrange equation for the action~\eqref{Eq-Smix-g}.
In~Section~\ref{sec:contorsion}, we study variations of $\I$
and, using results for $k=2$ from \cite{rz-3}, find the Euler-Lagrange equation for \eqref{Eq-Smix-g} with fixed adapted metric.
In subsequent sections, we examine solutions of the Euler-Lagrange equations in some special cases.
%%% NEW
In Section~\ref{sectionEinstein}, for one-dimensional distributions we find critical pairs $(g, \I)$ with non-Einstein metrics
of constant scalar curvature for $k=3$ and for $k>3$ using Hadamard matrices.
In~Section~\ref{sec:cont-stat}, we apply the results of Sections~\ref{sec:adapted-metric} and \ref{sec:contorsion} to
statistical connections;
in particular, in Section~\ref{sec:contorsion-statistical} we consider twisted products, and
in Section~\ref{sec:metric-stat}, we study the action \eqref{Eq-Smix-g}
for adapted variations of metric and variations of contorsion tensor corresponding to statistical connections.
In Section~\ref{sec-metricconnections}, we apply the results of Sections~\ref{sec:adapted-metric} and \ref{sec:contorsion} to metric connections; in particular, in Section~\ref{sec:contorsion_semi_symmetric}, we consider semi-symmetric connections, and in Section~\ref{sec:5-Sasaki}, we examine
the Euler-Lagrange equations for \eqref{Eq-Smix-g} on a 3-Sasaki manifold.

\section{Preliminaries}
\label{sec:prel}

Here, we recall the properties of the mixed scalar curvature of a metric-affine almost product manifold $(M,g,\bar\nabla;\mD_1,\ldots,\mD_k)$,
see \cite{RS-1}.
We will use bar in the notation of objects related to~$\bar\nabla$. Recall that $\I=\bar\nabla-\nabla$, where $\nabla$ is the Levi-Civita connection of $g$, is the \emph{contorsion tensor}.
%We denote all objects related to $\bar\nabla$ with bar.

A {pseudo-Riemannian metric} $g=\<\cdot\,,\cdot\>$ of index $q$ on $M$ is an element $g\in{\rm Sym}^2(M)$
of the space of symmetric $(0,2)$-tensors
such that each $g_x\ (x\in M)$ is a {non-degenerate bilinear form of index} $q$ on the tangent space $T_xM$.
For~$q=0$ (i.e., $g_x$ is positive definite) $g$ is a Riemannian metric and for $q=1$ it is a Lorentz metric.
 A distribution $\mD_\mu$ on $(M,g)$ is \textit{non-degenerate}, if $g_x$ is non-degenerate on $\mD_\mu(x)\subset T_x M$ for all $x\in M$;
 in this case, the orthogonal complement $\mD_\mu^\bot$ is also non-degenerate, e.g., \cite{bf}.
%%%%%%%%%%%%%%%%
Given an adapted metric $g$ on $(M;\mD_1,\ldots\mD_k)$, there is a~local orthonormal frame
$\{E_{\mu,a}\}$ on $M$, where $1\le a \le n_\mu=\dim \mD_\mu$ such that
 $\{E_{\mu,1},\ldots, E_{\mu,n_\mu}\}\subset\mD_\mu$ for $1\le \mu\le k$.
All quantities defined below using such frame do not depend on the choice of this~frame.
Similarly, $\{ {\cal E}_{\mu, i}\}$, where $i = 1, \ldots, n^\bot_\mu$, is an orthonormal frame of $\mD_\mu^\perp$ with $n^\bot_\mu=\dim\mD_\mu^\perp$.
Thus, the ranges of indices, e.g., $a$ or $i$, are determined by the index of the distribution, $\mD_\mu$ or $\mD^\bot_\mu$, respectively.
Unless explicitly stated otherwise, sums in all formulas will be taken over repeated indices, and always over full ranges of indices.

%\begin{Remark}\rm
For the curvature tensor $\bar R_{X,Y}=[\bar\nabla_Y,\bar\nabla_X]+\bar\nabla_{[X,Y]}$ of $\bar\nabla$, we~get
%\[
%\label{E-RC-2}
 $\bar R_{X,Y} -R_{X,Y} = (\nabla_Y\,\I)_X -(\nabla_X\,\I)_Y +[\I_Y,\,\I_X]$,
%\]
see \cite{r-affine},
where $R_{X,Y}=[\nabla_Y,\nabla_X]+\nabla_{[X,Y]}$ is the curvature tensor of~$\nabla$.
The {mixed scalar curvature} of a pair of distributions $(\mD,\mD^\bot)$ on a
%metric-affine
manifold $(M,g;\bar\nabla)$ is given~by
\[
%\label{eq-wal2}
 \bar{\rm S}_{\,\mD,\mD^\bot} = \frac12\sum\nolimits_{\,a,b} \eps_a\,\eps_b(\<\bar R_{\,{E}_a, {\cal E}_b}{E}_a, {\cal E}_b\>+\<\bar R_{\,{\cal E}_b, {E}_{a}}{\cal E}_b, {E}_{a}\>)\,,
\]
where we use a local orthonormal frame on $M$ such that ${E}_{a}\in\mD$ for $a\le\dim\mD$, ${\cal E}_{b}\in\mD^\perp$ for %$\dim\mD<b\le n=\dim M$
$b\le\dim\mD^\perp$
and %$\eps_\nu=\<{E}_{\nu},{E}_{\nu}\>\in\{-1,1\}$.
$\eps_a=\<{E}_{a},{E}_{a}\>\in\{-1,1\}$, $\eps_b=\< {\cal E}_b , {\cal E}_b \>\in\{-1,1\}$.
If~$\mD$ is spanned by a unit vector field $N$,
then $\bar{\rm S}_{\,\mD,\mD^\bot}=\eps_N\overline{\rm Ric}_{N,N}$, where $\overline{\rm Ric}_{N,N}$ is the Ricci curvature of $\bar\nabla$ in the $N$-direction.
%\end{Remark}
This concept can be generalized to $k>2$ distributions.
%A plane in $TM$ spanned by two vectors belonging to different distributions $\mD_\mu$ and $\mD_\nu$ will be called~\textit{mixed}.
%and its sectional curvature is called mixed.
%The mixed scalar curvature of $(M,g;\mD_1,\ldots,\mD_k)$ is
%defined as
%an averaged sectional curvature of all mixed planes.

\begin{Definition}\rm
Given $(M;\mD_1,\ldots,\mD_k)$ with an adapted metric $g$ and a linear connection $\bar\nabla$,
the following function on $M$ is called the \textit{mixed scalar curvature} with respect to~$\bar\nabla$, see \cite{RS-1}:
\begin{equation}\label{E-Smix-k}
 \overline{\rm S}_{\,\mD_1,\ldots,\mD_k}=\frac12\sum\nolimits_{\,\nu<\mu}\sum\nolimits_{\,a,b}
 \eps_a\,\eps_b\big(\<\bar R_{{E}_{\nu,a},{E}_{\mu,b}}\,{E}_{\nu,a},\,{E}_{\mu,b}\> + \<\bar R_{{E}_{\mu,b},{E}_{\nu,a}}\,{E}_{\mu,b},\,{E}_{\nu,a}\>\big).
\end{equation}
%where $\eps_a=\<{E}_{\nu,a},{E}_{\nu,a}\>\in\{-1,1\}$.
If $\I=0$, then the above function is called the \textit{mixed scalar curvature} (with respect to $\nabla$):
%(with respect to the Levi-Civita connection $\nabla$),
%see \cite{r-IF-k}:
\[
 {\rm S}_{\,\mD_1,\ldots,\mD_k}=\sum\nolimits_{\,\nu<\mu}\sum\nolimits_{\,a,b}\eps_a\,\eps_b\<R_{{E}_{\nu,a},{E}_{\mu,b}}\,{E}_{\nu,a},\,{E}_{\mu,b}\>.
\]
\end{Definition}

The~symmetric second fundamental form $h_\mu:\mD_\mu\times \mD_\mu\to \mD_\mu^\bot$
and the skew-symmetric integrability tensor $T_\mu:\mD_\mu\times \mD_\mu\to \mD_\mu^\bot$ (of the distribution $\mD_\mu$) are defined by
\[
 h_\mu(X,Y) = \frac12\,P_{\,\mu}^\bot(\nabla_XY+\nabla_YX),\quad
 T_\mu(X,Y) = \frac12\,P_{\,\mu}^\bot(\nabla_XY-\nabla_YX) = \frac12\,P_{\,\mu}^\bot\,[X,Y],
\]
where $P_\mu:TM\to\mD_\mu$ and $P_{\,\mu}^\bot:TM\to \mD_{\,\mu}^\bot$ are orthoprojectors.
The mean curvature vector field of $\mD_\mu$ is given by the trace of second fundamental form: $H_\mu=\tr_g\, h_\mu=\sum_{\,a} h_\mu (E_{\mu, a},E_{\mu, a})$.
Similarly, $\tilde h_{\,\mu},\,\tilde H_{\,\mu}=\tr_g \tilde h_{\,\mu}$ and $\tilde T_{\,\mu}$ are defined for
%the~second fundamental forms, mean curvature vector fields and the integrability tensors of
$\mD_{\,\mu}^\bot$, e.g., $\tilde h_{\,\mu} : \mD_\mu^\bot \times \mD_\mu^\bot \to \mD_\mu $ is given by $\tilde h_{\,\mu} (X,Y) = \frac12\,P_{\,\mu} (\nabla_XY+\nabla_YX)$.
We have $H_\mu=\sum\nolimits_{\,\nu\ne\mu} P_\nu H_\mu$ and ${\tilde H}_\mu=\sum\nolimits_{\,\nu\ne\mu} P_\mu H_\nu$. %etc.
Set
\begin{equation} \label{defcalH}
 {\cal H} = \sum\nolimits_{\mu} H_\mu = \sum\nolimits_{\mu} {\tilde H}_\mu.
\end{equation}
To see that the above definition is valid, we use $P_\mu H_\mu=0$ to obtain
\begin{equation} \label{calH}
\sum\nolimits_{\,\mu} {\tilde H}_\mu = \sum\nolimits_{\,\mu} P_\mu \sum\nolimits_{\,\nu} H_\nu
= \sum\nolimits_{\,\nu} \sum\nolimits_{\,\mu} P_\mu H_\nu = \sum\nolimits_{\,\nu} H_\nu\,.
\end{equation}
A~distribution $\mD_\mu$ is called integrable if $T_\mu=0$, and $\mD_\mu$ is called {totally umbilical}, {harmonic}, or {totally geodesic},
if ${h}_\mu=({H}_\mu/n_\mu)\,g,\ {H}_\mu =0$, or ${h}_\mu=0$, respectively, e.g.,~\cite{bf}.
Totally umbilical and totally geodesic integrable distributions naturally appear on twisted products.

The squares of norms of tensors on $(M,g;\mD_1,\ldots\mD_k)$ are determined using
\begin{eqnarray*}
 &&\<h_\mu,h_\mu\>=\sum\nolimits_{\,a,b} \eps_a\eps_b\,\<h_\mu({E}_{\mu,a},{E}_{\mu,b}),h_\mu({E}_{\mu,a},{E}_{\mu,b})\>, \\
 &&\<T_\mu,T_\mu\>=\sum\nolimits_{\,a,b} \eps_a\eps_b\,\<T_\mu({E}_{\mu,a},{E}_{\mu,b}),T_\mu({E}_{\mu,a},{E}_{\mu,b})\>,\quad {\rm etc}.
\end{eqnarray*}
Similarly, for two $(0,s)$ or $(1,s)$ tensors $F_1,F_2$, we will denote by $\< F_1,F_2 \>$ their inner product defined by $g$.
Let $h_{\mu\nu},\,H_{\mu\nu},\,T_{\mu\nu}$ be the second fundamental forms, the mean curvature vector fields and the integrability tensors
%and orthoprojectors $P_{\mu\nu}:TM\to $
related to the distributions ${\cal D}_\nu\oplus{\cal D}_\mu$ for $\mu\ne\nu$.

%%%%%%%%%%%%%
\begin{Definition}\rm
A pair $({\cal D}_\mu,{\cal D}_\nu)$ with $\mu\ne\nu$ of distributions on
$(M,g;{\cal D}_1,\ldots,{\cal D}_k)$ is called

a) \textit{mixed totally geodesic}, if $h_{\mu\nu}(X,Y)=0$ for all $X\in{\cal D}_\mu$ and $Y\in{\cal D}_\nu$.

b) \textit{mixed integrable}, if $T_{\mu\nu}(X,Y)=0$ for all $X\in{\cal D}_\mu$ and $Y\in{\cal D}_\nu$.
\end{Definition}

Let~$\mathfrak{X}_M$
%(resp., $\mathfrak{X}_{{\mD}}$)
be the module over $C^\infty(M)$ of all vector fields on $M$.
% (resp. all vector fields with values in ${\mD}$).
The ``musical" isomorphisms $\sharp$ and $\flat$ will be used for rank one and symmetric rank 2 tensors.
For~example, if $\omega\in\Lambda^1(M)$ is a 1-form and $X,Y\in\mathfrak{X}_M$ then $\omega(Y)=\<\omega^\sharp,Y\>$ and $X^\flat(Y) =\<X,Y\>$.
For arbitrary (0,2)-tensors $B$ and $C$ we also have $\<B, C\> =\tr_g(B^\sharp C^\sharp)=\<B^\sharp, C^\sharp\>$.

The shape operator $(A_\mu)_Z$ of $\mD_\mu$ with respect to $Z\in\mD_\mu^\bot$ (dual to the second fundamental form $h_\mu$) and the operator $(T_\mu^\sharp)_{Z}$ (dual to the integrability tensor $T_\mu$) are given~by
\[
 \<(A_\mu)_Z(X),Y\>= \< h_\mu(X,Y),Z\>,\quad \<(T_\mu^\sharp)_Z(X),Y\>=\<T_\mu(X,Y),Z\>, \quad X,Y \in \mD_\mu .
\]
Similarly, linear operators $(\tilde A_\mu)_Z$ and $(\tilde T_\mu^{\sharp})_Z$ on $\mD_\mu^\bot$ with $Z\in\mD_\mu$ are defined. To make formulas easier to read, we will sometimes write $A_{\mu, Z}$ instead of $(A_\mu)_Z$ and ${\tilde A}_{\mu, a}$ instead of  $({\tilde A}_\mu)_{E_{\mu,a}}$, etc.
The divergence of a $(1,s)$-tensor field $S$ on $(M,g)$ is a $(0,s)$-tensor field
%\[
 ${\rm div}\,S =\operatorname{trace}(Y\to\nabla_{Y} S)$,
%\]
%that is
\[
%\label{eq:div2}
 (\Div S)(X_1,\ldots, X_s) = \sum\nolimits_{\,i}\<(\nabla_{ E_i }\,S)(X_1,\ldots, X_s), E_i\>\,,
\]
where $(E_1,\ldots, E_n)$ is a local orthonormal frame on $TM$. For $s=0$, this is the {divergence}
%\[
 ${\rm div}\,X =\tr\nabla X$
%\]
of a vector field $X\in\mathfrak{X}_M$.
The identity tranformation on $TM$ will be denoted by $\id$.

For the contorsion tensor, we define auxiliary (1,2)-tensors $\I^*$ and $\I^\wedge$ by
\[
\<\I^*_X Y,Z\> = \<\I_X Z, Y\>,\quad \I^\wedge_X Y = \I_Y X,
\quad X,Y,Z\in\mathfrak{X}_M\, ,
\]
similarly $\< \I^{*\wedge}_X Y,Z \> = \< \I^*_Y X, Z\> = \< \I_Y Z, X \>$.
%The next definition is given to simplify the presentation of results.
%A linear connection $\bar\nabla=\nabla+\I$ on $(M,g;\mD_1,\ldots,\mD_k)$
%will be called here \textit{adapted} if $\I$ is decomposed into the sum of $\mD_\mu$-components,
%\begin{eqnarray}\label{E-adapted-I}
%\nonumber
% && \I_X\,Y=0\quad (X\in\mD_\mu,\ \ Y\in\mD_\nu,\ \ \mu\ne\nu),\\
% && \I_X\,Y \in\mD_\mu\quad (X,Y\in\mD_\mu).
%\end{eqnarray}
Set $\I_{\mu, a} = \I_{E_{\mu,a}}$.
% for $E_{\mu,a}$.
The partial traces of
%a contorsion tensor
$\,\I$ are defined by
%\[
%\label{E-defTT}
 $\tr^\bot_\mu \I = \sum\nolimits_i \I_{{\cal E}_{\mu, i}}{\cal E}_{\mu, i},\
 \tr^\top_\mu \I = \sum\nolimits_a \I_{\mu, a} E_{\mu, a}$.
%\]
Note that $\tr^\bot_\mu \I=\sum\nolimits_{\,\nu\ne\mu} \tr^\top_\nu \I$.

Set $V_\mu=(\mD_\mu\times\mD_\mu^\bot)\cup(\mD_\mu^\bot\times\mD_\mu)$.
For $(M,g,\bar\nabla;\mD_\mu,\mD_\mu^\bot)$ we have the following equalities:
\begin{equation}\label{E-div-barQ}
 \Div X_\mu = \bar{\rm S}_{\,\mD_\mu,\mD_\mu^\bot} -Q(\mD_\mu,g) -\bar Q(\mD_\mu,g,\I)\,,
\end{equation}
see \cite{r-affine}, where $X_\mu =\frac12\,\big(P_\mu\tr_{\,\mu}^\bot(\I -\I^*) +P_\mu^\bot\tr_{\,\mu}^\top(\I -\I^*)\big) +H_\mu+\tilde H_\mu$ and
\begin{eqnarray}\label{E-func-Q}
 Q(\mD_\mu,g)\eq\<\tilde H_\mu,\tilde H_\mu\>+\<H_\mu,H_\mu\>-\<h_\mu,h_\mu\>-\<\tilde h_\mu,\tilde h_\mu\>+\<T_\mu,T_\mu\>+\<\tilde T_\mu,\tilde T_\mu\>\,,\\
\label{E-barQ}
\nonumber
 2\,\bar Q(\mD_\mu,g,\I) \eq \<\tr_{\,\mu}^\top\I,\, \tr_{\,\mu}^\bot\I^*\> +\<\tr_{\,\mu}^\bot\I,\, \tr_{\,\mu}^\top\I^*\>\\
\nonumber
 \plus \<\tr_{\,\mu}^\top(\I-\I^*) -\tr_{\,\mu}^\bot(\I -\I^*), H_\mu -\tilde H_\mu\> \\
 \plus \<\I -\I^* +\I^\wedge - \I^{*\wedge},\ \tilde A_\mu -\tilde T_\mu^{\sharp} + A_\mu -T_\mu^{\sharp}\> -\<\I^*,\,\I^\wedge\>_{\,|\,V_\mu}\,.
\end{eqnarray}
%where the auxiliary functions $Q(\mD,g)$ and $\bar Q(\mD,g,\I)$ are given by \eqref{E-func-Q} and \eqref{E-barQ},.
In a local adapted frame, two terms in the last line of \eqref{E-barQ} have the following form:
%are represented by
\begin{eqnarray*}
%\label{E-Q1Q2-gen0}
 && \<\I -\I^* +\I^\wedge - \I^{*\wedge},\ \tilde A_\mu -\tilde T_\mu^{\sharp} + A_\mu - T_\mu^{\sharp}\>
 = \sum\nolimits_{\,a,b}\big(\<(\I_{{\cal E}_{\mu,b}} -\I^*_{{\cal E}_{\mu,b}}) {E}_{\mu,a}\\
 && +\,(\I_{{\mu,a}} -\I^*_{{\mu,a}}) {\cal E}_{\mu,b},\ ((\tilde A_\mu)_{{E}_{\mu,a}} -(\tilde T_\mu^{\sharp})_{{E}_{\mu,a}}) {\cal E}_{\mu,b} +((A_\mu)_{{\cal E}_{\mu,b}}-(T_\mu^{\sharp})_{{\cal E}_{\mu,b}}) {E}_{\mu,a}\>\big),\\
 && \<\I^*,\ \I^\wedge\>_{\,|\,V_\mu}
 = \sum\nolimits_{\,a,b}\big(\<\I_{{\mu,a}} {\cal E}_{\mu,b},\,\I^*_{{\cal E}_{\mu,b}} {E}_{\mu,a}\> + \<\I^*_{{\mu,a}} {\cal E}_{\mu,b},\,\I_{{\cal E}_{\mu,b}} {E}_{\mu,a}\>\,\big).
%%%%%%%%%%%%%%%%%%%%%%%%%%%%%%%
\end{eqnarray*}
The following result, see  \cite[Proposition~2]{RS-1}, generalizes \eqref{E-div-barQ} for $k>2$.
% and any linear connection $\bar\nabla$.

%\label{L-QQ-first}
\begin{Proposition}\label{P-decomp2}
 For~an almost product manifold $(M,g;\mD_1,\ldots,\mD_k)$
 equipped with a linear connection $\bar\nabla=\nabla+\I$ we~have
\begin{eqnarray}\label{E-Q1Q2-gen}
 \Div\,Y = 2\,\bar{\rm S}_{\,\mD_1,\ldots,\mD_k} - \sum\nolimits_{\,\mu}\big( Q(\mD_\mu,g) +\bar Q(\mD_\mu,g,\I) \big),
\end{eqnarray}
where tensors $Q(\mD_\mu,g)$ and $\bar Q(\mD_\mu,g,\I)$ are given by \eqref{E-func-Q} and \eqref{E-barQ} with $\mD=\mD_\mu$, and
\begin{equation}\label{E-prop-X}
 Y=\sum\nolimits_{\,\mu}\big(\frac12\,P_\mu\tr_{\,\mu}^\bot(\I -\I^*) +\frac12\,P_\mu^\bot\tr_{\,\mu}^\top(\I -\I^*) +H_{\,\mu}+\tilde H_{\,\mu}\big).
\end{equation}
\end{Proposition}

\begin{proof}
%This is given in \cite[Proposition~2]{RS-1}.
%For the convenience of the reader, we present a proof.
For a pair of complementary distributions $(\mD_\mu,\mD_\mu^\bot)$ on $(M,g)$ we have
%\[
%\label{eq-wal2}
 $\overline{\rm S}_{\,\mD_\mu,\mD_\mu^\bot} = \sum\nolimits_{\,a,\,b}
 \eps_a \eps_b\,\<\overline R_{\,{E}_{\mu,a}, {\cal E}_{\mu,b}}{E}_{\mu,a}, {\cal E}_{\mu,b}\>$.
%\]
Thus,
%\eqref{E-Dk-Smix} directly follows
from the equality $\overline{\rm S}_{\,\mD_\mu,\mD^\bot_\mu}=\sum\nolimits_{\,\nu\ne\mu}\overline{\rm S}_{\,\mD_\mu,\mD_\nu}$
and definition \eqref{E-Smix-k} we obtain the following decomposition formula of the mixed scalar curvature,
see \cite{RS-1}: %(and \cite{r-IF-k} for $\I=0$):
% and \eqref{E-Rictop2}.
\begin{equation}\label{E-Dk-Smix}
 2\,\overline{\rm S}_{\,\mD_1,\ldots,\mD_k} = \sum\nolimits_{\,\mu}\overline{\rm S}_{\,\mD_\mu,\mD^\bot_\mu}\,.
\end{equation}
Summing $k$ copies of \eqref{E-div-barQ} with $\mD_\mu\ (\mu=1,\ldots,k)$ and using \eqref{E-Dk-Smix} yields \eqref{E-Q1Q2-gen}.
\end{proof}

\begin{Remark}
%\label{C-stat}
\rm
For a statistical connection $\bar\nabla$ on $(M,g)$ we have $\I^\wedge=\I$ and $\I^* = \I$;
in this case, \eqref{E-barQ} has a shorter form
%\label{E-QI}
 $2\,\bar Q(\mD_\mu,g,\I)= 2\,\<\tr_{\,\mu}^\top\I,\ \tr_{\,\mu}^\bot\I\> -\<\I,\,\I\>_{\,|\,V_\mu}$\,,
%For a statistical connection $\bar\nabla$ on $(M,g)$, by \eqref{E-QI},  equality
and \eqref{E-Q1Q2-gen} reduces~to
%\eqref{eqvarstat}.
\begin{equation}\label{eqvarstat}
 2\,\overline{\rm S}_{\,\mD_1,\ldots,\mD_k} = 2\,{\rm S}_{\,\mD_1,\ldots,\mD_k}
 -\sum\nolimits_{\,\mu}\big(\<\tr_{\,\mu}^\bot\I,\,\tr_{\,\mu}^\top\I\> -\frac12\,\<\I,\,\I\>_{\,|V_\mu}\big) .
%%%%%%%%%%%%%%%%%%%%%%%%%%%%%%%
\end{equation}
\end{Remark}

\section{Adapted variations of metric}
\label{sec:adapted-metric}

Here, we define adapted variations $g_t\ (|t|<\eps)$ of a pseudo-Riemannian metric $g = g_0$
and similarly to the case of \eqref{Eq-Smix-g0}, see \cite{r-EH-k},
find a general form of the Euler-Lagrange equation for the action \eqref{Eq-Smix-g} with respect to those variations. %, with respect to adapted variations of the metric.
%We write it explicitly for statistical and semi-symmetric connections.
% and give examples of critical multiply twisted product metrics.
%\subsection{The Euler-Lagrange equation}
%\label{sec:contorsion_statistical}
Let~infinitesimal variations
%\[
 ${B}_t\equiv\partial g_t/\partial t$
%\]
be supported in a relatively compact domain $\Omega\subset M$, i.e., on $M\setminus\Omega$ we have $g_t=g$ and ${B}_t=0$ for all $t$.
%A~variation $g_t$ is called \textit{volume-preserving} if ${\rm Vol}(\Omega,g_t) = {\rm Vol}(\Omega,g)$ for all~$t$.
 We~adopt notations $\partial_t \equiv\partial/\partial t,\ {B}\equiv{\dt g_t}_{\,|\,t=0}$,
and write ${B}$ instead of ${B}_t$ to make formulas easier to~read.
%First, recall some
The~volume form ${\rm d}\vol_g$ of metric $g$ varies as follows, e.g., \cite{rz-2},
\begin{equation}\label{E-dotvolg}
 \partial_t\,\big({\rm d}\vol_{g}\!\big) = \frac12\,(\tr_{g}\,{B})\,{\rm d}\vol_{g}
 =\frac12\,\<{B},\,g\>\,{\rm d}\vol_{g}\,.
\end{equation}
%\begin{Remark}\label{R-volume}\rm
For any variations $g_t$ of metric, the Euler-Lagrange equation
means vanishing of the partial gradient $\delta_g\bar J_\mD(g,\I)$ of the functional \eqref{Eq-Smix-g} with fixed tensor $\I$, where
%\[
 ${\rm\frac{d}{dt}}\,\bar J_\mD(g_t,\I)_{|\,t=0} =\int_{\Omega}\<\delta_g \bar J_\mD, {B}\>\,{\rm d}\vol_g$
%\]
and ${B} =\dt g_{t\,|\,t=0}$. %and similarly for \eqref{Eq-Smix-g}.
Solutions $g$ of $\delta_g \bar J_\mD=0$ are called critical metrics.
For variations preserving the volume of $\Omega$, i.e., ${\rm Vol}(\Omega,g_t) = {\rm Vol}(\Omega,g)$ for all~$t$,
using \eqref{E-dotvolg}, we get
\[
 0 = \dt\int_{M}{\rm d}\vol_g = \int_{M} \dt\,({\rm d}\vol_g) = \int_{M}\frac{1}{2}\,(\tr_g{B})\,{\rm d}\vol_g =\frac{1}{2} \int_{\Omega}\<g,\,{B}\>\,{\rm d}\vol_g.
\]
Hence, $g$ is critical for variations of metric preserving the volume of $\Omega$ if and only if
the condition
%\[
 $\int_{\Omega}\<\delta_g \bar J_\mD,\,{B}\>\,{\rm d}\vol_g =0$
%\]
holds for all tensors ${B}$ satisfying
%\[
 $\int_{\,\Omega}\<g,\,{B}\>\,{\rm d}\vol_g=0$.
%\]
Thus, the Euler-Lagrange equation for variations preserving the volume of $\Omega$ is
%\begin{equation}\label{ELgeneralvolpreserving}
\begin{equation}\label{E-delta-g-J}
 \delta_g \bar J_\mD = \lambda\,g
\end{equation}
for some $\lambda\in\RR$.
Following \cite{rz-3}, where $k=2$, we define auxiliary Casorati type operators
${\cal T}_\mu:\mD_\mu\to\mD_\mu$ and self-adjoint $(1,1)$-tensors ${\cal K}_\mu$ (using the Lie bracket) by
\[
 {\cal T}_\mu=\sum\nolimits_{\,a}\eps_a(\T_{\mu,a})^2,\quad
 {\cal K}_\mu=\sum\nolimits_{\,a}\eps_{\,a}\,[\T_{\mu,a},\, A_{\mu,a}]\,.
% {\cal T}_\mu=\sum\nolimits_{\,a}\eps_a(T_\mu^\sharp)_{{\cal E}_{\mu,a}}^2,\quad
% {\cal K}_\mu=\sum\nolimits_{\,a}\eps_{\,a}\,[(\T_\mu)_{{\cal E}_{\mu,a}},\, (A_\mu)_{{\cal E}_{\mu,a}}].
\]
For any $(1,2)$-tensors $P,P'$ and a $(0,2)$-tensor $S$ define the $(0,2)$-tensor $\Upsilon_{P,P'}$~by
\[
 \<\Upsilon_{P,P'}, S\> = \sum\nolimits_{\,\lambda, \mu} \eps_\lambda\, \eps_\mu\,
 \big[S(P(e_{\lambda}, e_{\mu}), P'( e_{\lambda}, e_{\mu})) + S(P'(e_{\lambda}, e_{\mu}), P( e_{\lambda}, e_{\mu}))\big],
\]
where we use the inner product of tensors induced by $g$,
$\{e_{\lambda}\}$ is a local orthonormal basis of $TM$ and $\eps_\lambda = \<e_{\lambda}, e_{\lambda}\>\in\{-1,1\}$.
%\begin{Remark}\rm
If %$g>0$
$g$ is a Riemannian metric,
then $\Upsilon_{h_\mu,h_\mu}=0$ if and only if $h_\mu=0$.
Thus, $\Upsilon_{h_\mu,h_\mu}$ measures ``non-total geodesy" of~$\mD_\mu$;
similarly, $\Upsilon_{T_\mu,T_\mu}$ measures ``non-integrability" of~$\mD_\mu$.
%\end{Remark}

\begin{Definition}[see \cite{r-EH-k}]\rm
A family of adapted metrics $g_t\ (|t|<\eps)$ on $(M;\mD_1,\ldots\mD_k)$ such that $g_0 =g$
and $B_t$ are compactly supported, is called an \textit{adapted variation} of $g$.
In~this case, distributions $\mD_\mu$ and $\mD_\nu$ are $g_t$-orthogonal for all $\mu\ne\nu$ and all~$t$.
An adapted variation $g_t$ is called a $\mD_\mu$-\textit{variation} (for some fixed $\mu\in\{1,\ldots, k\}$) if
the metric changes along $\mD_\mu$ only,~i.e.,
%\[
 $g_t(X,Y)=g_0(X,Y),\quad X,Y\in\mD_\mu^\bot,\quad |t|<\eps$.
%\]
\end{Definition}

An adapted variation $g_t$ is a sum $g_t=g_1(t)\oplus\ldots\oplus g_k(t)$ of $\mD_\mu$\,-variations $g_\mu(t) =g_t|_{\,\mD_\mu}$.
In~this case, the tensor ${B}_t=\dt\, g_t$ is a sum ${B}_t=\sum\nolimits_{\,\mu}{B}_\mu(t)$,
where ${B}_\mu(t) =\dt g_\mu(t) ={B}_t|_{\,\mD_\mu}$.
A~special case of adapted variations is a multiconformal variation of metric, see~\cite{RN-21}.

In~view of Proposition~\ref{P-decomp2}, we need the variation of $\sum\nolimits_{\,\nu}\big(\bar Q(\mD_\nu,g,\I) + Q(\mD_\nu,g)\big)$.

\begin{Lemma}[see \cite{r-EH-k}]\label{propvar1}
Let $g_t$
%with $\dt g={B}_\mu$
be a $\mD_\mu$-variation of an adapted metric $g$ on $(M;\mD_1,\ldots\mD_k)$,
%for $k>2$,
 then
\begin{eqnarray*}
\nonumber
%\label{E-tildeh-gen}
 && \dt\<\tilde{h}_\mu, \tilde{h}_\mu\> = -\< (1/2)\Upsilon_{\tilde h_\mu, \tilde h_\mu},\ {B}_\mu\>\,,\quad
% \\
%\label{E-h-gen}
% &&
 \dt \<h_\mu,\ h_\mu\> = \<\,\Div{h}_\mu + {\cal K}_\mu^\flat,\ {B}_\mu\> - \Div\<h_\mu, {B}_\mu\>\,,\\
%\label{E-tildeH-gen}
 && \dt \<\tilde {H}_\mu, \tilde {H}_\mu\> = -\<\,\tilde{H}_\mu^\flat\otimes\tilde{H}_\mu^\flat,\ {B}_\mu\>\,,\quad
%\label{E-H-gen}
 \dt \<H_\mu, H_\mu\> = \<\,(\Div H_\mu)\,g_\mu,\ {B}_\mu\> -\Div((\tr {B}_\mu^\sharp) H_\mu)\,, \\
%\label{E-tildeT-gen}
 && \dt\<\tilde{T}_\mu, \tilde{T}_\mu\> = \<\,(1/2)\Upsilon_{\tilde T_\mu, \tilde T_\mu},\ {B}_\mu\>\,,\quad
%\label{E-T-gen}
 \dt\<T_\mu,\ T_\mu\> = \< 2\,{\cal T}_\mu^\flat,\ {B}_\mu\>\,,
\end{eqnarray*}
and for $\nu\ne\mu$ we have dual equations
\begin{eqnarray*}
%\label{E-tildeh-gen}
 && \dt\<{h}_\nu, {h}_\nu\> = \<-(1/2)\Upsilon_{h_\nu,h_\nu},\ {B}_\mu\>\,,\quad
%\label{E-h-gen}
 \dt \<\tilde h_\nu,\ \tilde h_\nu\> = \<\,\Div\tilde{h}_\nu + \tilde{\cal K}_\nu^\flat,\ {B}_\mu\> - \Div\<\tilde h_\nu, {B}_\mu\>\,,\\
%\label{E-tildeH-gen}
 && \dt \<{H}_\nu, {H}_\nu\> = -\<\,{H}_\nu^\flat\otimes{H}_\nu^\flat,\ {B}_\mu\>\,,\quad
%\label{E-H-gen}
 \dt \<\tilde H_\nu, \tilde H_\nu\> = \<\,(\Div\tilde H_\nu)\,g_\nu^\bot,\ {B}_\mu\> -\Div((\tr {B}_\mu^\sharp) \tilde H_\nu)\,, \\
%\label{E-tildeT-gen}
 && \dt\<{T}_\nu, {T}_\nu\> = \<\,(1/2)\Upsilon_{T_\nu,T_\nu},\ {B}_\mu\>\,,\quad
%\label{E-T-gen}
 \dt\<\tilde T_\nu,\ \tilde T_\nu\> = \< 2\,\tilde{\cal T}_\nu^\flat,\ {B}_\mu\>\, .
\end{eqnarray*}
\end{Lemma}

Variational formulas of terms of $\bar Q$ in \eqref{E-barQ} obtained in the following lemma are a special case (i.e., for adapted variations of metric) of equations from \cite[Lemma 3]{rz-3} for $(\mD_\mu,\mD_\mu^\bot)$.
Detailed proof of such variational formulas %Lemma~\ref{P-dT-3} %for particular choice of $\I$ -not only, also on twisted product
in particular setting
will be given further below in Lemma~\ref{propdeltaQforstatistical}.

\begin{Lemma}\label{P-dT-3}
Let $g_t$
%with $\dt g={B}_\mu$,
be a $\mD_\mu$-variation of an adapted metric $g$ on $(M;\mD_1,\ldots\mD_k)$
%for $k>2$
with fixed statistical connection $\bar\nabla=\nabla+\I$. Then
\begin{eqnarray*}
%\label{Eq-stat-var}
\nonumber
 && \dt \<\I^*,\ \I^\wedge\>_{\,|\,V_\mu} = -\sum\nolimits_{\,a,b} {B}_\mu({E}_{\mu,a}, {E}_{\mu,b}) \<\I_{{\mu,a}}, \I_{{\mu,b}}\>_{\,|\,\mD_\mu^\bot}\,,\\
\nonumber
 && \dt \<\tr_{\,\mu}^\bot\I,\ \tr_{\,\mu}^\top\I^*\> =0\,,\\
\nonumber
 && \dt \<\tr_{\,\mu}^\top\I^*,\ \tr_{\,\mu}^\bot\I\> = -\sum\nolimits_{\,a,b} {B}_\mu({E}_{\mu,a}, {E}_{\mu,b})\<\I_{{\mu,a}}{E}_{\mu,b}\,,\ \tr_{\,\mu}^\bot\I\>,\\
\nonumber
 && \dt \<\Theta, \tilde A_\mu\> = -2 \sum\nolimits_{\,a,b} {B}_\mu({E}_{\mu,a}, {E}_{\mu,b}) \<(\tilde A_\mu)_{{E}_{\mu,a}},\, \I_{{\mu,b}}\>\,, \\
\nonumber
 && \dt \<\Theta, \tilde T_\mu^{\sharp}\> =-2\sum\nolimits_{\,a,b}{B}_\mu({E}_{\mu,a}, {E}_{\mu,b})\<(\tilde T^{\sharp}_\mu)_{{E}_{\mu,a}},\, \I_{{\mu,b}}\>\,,\\
\nonumber
 && \dt \<\Theta, T_\mu^{\sharp}\>
 =-6\sum\nolimits_{\,a,b} {B}_\mu({E}_{\mu,a}, {E}_{\mu,b})\<T_\mu({E}_{\mu,b},\,\cdot),\, \I_{{\mu,a}}\>\,,\\
\nonumber
 && \dt \<\Theta, A_\mu\>
 = 2\sum\nolimits_{\,a,b} {B}_\mu({E}_{\mu,a}, {E}_{\mu,b})\< h_\mu({E}_{\mu,b},\,\cdot),\, \I_{{\mu,a}}\>\,,\\
\nonumber
 && \dt \<\tr_{\,\mu}^\top(\I^*-\I),\, H_\mu - \tilde H_\mu\> =\sum\nolimits_{\,a,b} {B}_\mu({E}_{\mu,a}, {E}_{\mu,b})
  \big(\<\I_{{\mu,b}}\,{E}_{\mu,a}, H_\mu -\tilde H_\mu\> \\
\nonumber
 &&\quad +\,\<\tr_{\,\mu}^\top\I, {E}_{\mu,a}\>\<\tilde H_\mu, {E}_{\mu,b}\> \big), \\
 && \dt\<\,\tr_{\,\mu}^\bot(\I^*-\I),\, H_\mu - \tilde H_\mu\> =\sum\nolimits_{\,a,b} {B}_\mu({E}_{\mu,a}, {E}_{\mu,b})\<\tr_{\,\mu}^\bot\I, {E}_{\mu,b}\> \<\tilde H_\mu, {E}_{\mu,a}\>\,,
\end{eqnarray*}
where $\Theta=\I -\I^* +\I^\wedge - \I^{*\wedge}$, and for $\nu\ne\mu$ we get dual equations.
\end{Lemma}

%\begin{proof}
%Equations \eqref{Eq-stat-var} are a special case (i.e., for adapted variations of metric) of equations from
%\cite[Lemma 3]{rz-3} for $(\mD_\mu,\mD_\mu^\bot)$,
%and equations \eqref{Eq-stat-var-dual} are dual to \eqref{Eq-stat-var}.
%\end{proof}

\begin{Proposition}\label{C-vr-terms}
For any $\mD_\mu$-variation $g_t$ of an adapted metric $g$ on $(M;\mD_1,\ldots,\mD_k)$ with fixed linear connection $\bar\nabla=\nabla+\I$, we have
\begin{eqnarray}\label{E-dt-Q}
 \dt\sum\nolimits_{\,\nu} Q(\mD_\nu,g_t) \eq \<\delta Q_\mu,\, {B}_\mu\> - \Div X_\mu\,, \\
%\end{eqnarray}\begin{eqnarray}
\label{E-dt-barQ}
 \dt\sum\nolimits_{\,\nu}\bar Q(\mD_\nu,g_t,\I) \eq\<{\delta_g\bar Q}_\mu,\,{B}_\mu\>\,,
\end{eqnarray}
where $(0,2)$-tensors ${\delta Q}_\mu$ on $\mD_\mu\times\mD_\mu$ and vector fields $X_\mu$ on $M$ are given~by
\begin{eqnarray*}
 &&\hskip-5mm 2 X_\mu=\<h_\mu, {B}_\mu\> -(\tr {B}_\mu^\sharp) H_\mu
 +\sum\nolimits_{\,\nu\ne\mu}\big(\<\tilde h_\nu, {B}_\mu\> -(\tr {B}_\mu^\sharp) \tilde H_\nu\big),\\
%%%%%%%%%%%%%
 &&\hskip-5mm {\delta Q}_\mu = -\Div{h}_\mu -{\cal K}_\mu^\flat -\tilde{H}_\mu^\flat\otimes\tilde{H}_\mu^\flat
 +\frac12\Upsilon_{\tilde h_\mu, \tilde h_\mu} +\frac12\Upsilon_{\tilde T_\mu, \tilde T_\mu} +\,2\,{\cal T}_\mu^\flat + (\Div H_\mu)\,g_\mu \\
 && +\sum\nolimits_{\,\nu\ne\mu}\big(-\Div\tilde{h}_\nu|_{\mD_\mu} - (P_\mu\tilde{\cal K}_\nu)^\flat -(P_\mu{H}_\nu)^\flat\otimes(P_\mu{H}_\nu)^\flat \\
 && +\,\frac12\Upsilon_{P_\mu h_\nu,P_\mu h_\nu} +\frac12\Upsilon_{P_\mu T_\nu,P_\mu T_\nu} +2\,(P_\mu\tilde{\cal T}_\nu)^\flat +(\Div\tilde H_\nu)\,g_\mu\big),
\end{eqnarray*}
and certain $(0,2)$-tensors ${\delta_g\bar Q}_\mu$ on $\mD_\mu\times\mD_\mu$ have long expressions $($see {\rm \cite{rz-3}} for $k=2)$.
%for certain $(0,2)$-tensors ${\delta Q}_\mu$ and ${\delta_g\bar Q}_\mu$ on $\mD_\mu\times\mD_\mu$ and vector fields $X_\mu$ on $M$,
%where tensors $Q(\mD_\mu,g)$ and $\bar Q(\mD_\mu,g,\I)$ are defined in \eqref{E-func-Q} and \eqref{E-barQ}.
%
If $\,\bar\nabla$ is statistical, then
%$(0,2)$-
tensors ${\delta_g\bar Q}_\mu$ on $\mD_\mu\times\mD_\mu$ in \eqref{E-dt-barQ} can be written explicitly by
\begin{eqnarray*}
 &&\hskip-5mm 2\,{\delta_g\bar Q}_\mu({E}_{\mu,a},{E}_{\mu,b}) =
 \<\I_{{\mu,a}}, \I_{{\mu,b}}\>_{\,|\mD_\mu^\bot} -\<\I_{{\mu,a}}{E}_{\mu,b},\,\tr_{\,\mu}^\bot\I\> -2\,\<(\tilde A_\mu)_{{E}_{\mu,a}},\,\I_{{\mu,b}}\> \\
 && +\,2\,\<(\tilde T^{\sharp}_\mu)_{{E}_{\mu,a}},\,\I_{{\mu,b}}\> +6\,\<T_\mu({E}_{\mu,b},\,\cdot),\, \I_{{\mu,a}}\> + 2\,\< h_\mu({E}_{\mu,b},\,\cdot),\, \I_{{\mu,a}}\> \\
 && -\,\<\I_{{\cal E}_{\mu,b}}{E}_{\mu,a}, H_\mu - \tilde H_\mu\> -\<\tr_{\,\mu}^\top\I, {E}_{\mu,a}\>\< \tilde H_\mu, {E}_{\mu,b}\> +\<\tr_{\,\mu}^\bot\I, {E}_{\mu,b}\> \<\tilde H_\mu, {E}_{\mu,a}\> \\
%%%%%%%%%%%%%%%%%%%%%%%%%%%%%%%%%%%
 && +\sum\nolimits_{\,\nu\ne\mu}\big(\<\I_{{\mu,a}}, \I_{{\mu,b}}\>_{\,|\mD_\nu} {-}\<\I_{{\mu,a}}{E}_{\mu,b}, \tr_{\,\nu}^\top\I\> {-}2\<(A_\nu)_{{E}_{\mu,a}}, \I_{{\mu,b}}\> {+}2\< (T^{\sharp}_\nu)_{{E}_{\mu,a}}, \I_{{\mu,b}}\> \\
 && +\,6\,\< \tilde T_\nu({E}_{\mu,b},\,\cdot),\, \I_{{\mu,a}}\> + 2\,\<\tilde h_\nu({E}_{\mu,b},\,\cdot),\, \I_{{\mu,a}}\>
 -\<\I_{{\mu,b}}{E}_{\mu,a}, \tilde H_\nu - H_\nu\>\\
 && -\,\<\tr_{\,\nu}^\bot\I, {E}_{\mu,a}\>\< H_\nu, {E}_{\mu,b}\> +\<\tr_{\,\nu}^\top\I, {E}_{\mu,b}\> \<H_\nu, {E}_{\mu,a}\>\big).
% , \quad {E}_{\mu,a},{E}_{\mu,b}\in\mD_\mu\,.
\end{eqnarray*}
\end{Proposition}

\begin{proof}
Equations \eqref{E-dt-Q} and \eqref{E-dt-barQ}
with $Q(\mD_\mu,g)$ and $\bar Q(\mD_\mu,g,\I)$ given by \eqref{E-func-Q} and \eqref{E-barQ} follow from \eqref{E-Dk-Smix},
and explicit forms of tensors ${\delta Q}_\mu$ on $\mD_\mu\times\mD_\mu$ and vector fields $X_\mu$ follow from Lemma~\ref{propvar1}.
If $\,\bar\nabla$ is statistical, then explicit forms of tensors ${\delta_g\bar Q}_\mu$ on $\mD_\mu\times\mD_\mu$
follow from Lemma~\ref{P-dT-3}.
Note that $X_\mu$, ${\delta Q}_\mu$ and ${\delta_g\bar Q}_\mu$ consist of two parts, the summation part (related to $\mD_\mu^\bot$)
is dual to the part related to~$\mD_\mu$.
\end{proof}

%Using Theorem~\ref{T-main01}, we present the ``mixed Ricci" tensor $\overline\Ric_{\,\mD}$ explicitly for statistical connections.
%The~following proposition generalizes results in \cite{r2018} when~$k=2$.

In the following theorem (based on Proposition~\ref{C-vr-terms}) we generalize results in \cite{rz-3} with $k=2$.
% and present the Ricci type tensor $\overline\Ric_{\,\mD}$ explicitly for statistical connections.

\begin{Theorem}\label{T-main01}
Let $g$ be an adapted metric and $\bar\nabla=\nabla+\I$ a
%fixed
%statistical
linear connection on a manifold $(M;\mD_1,\ldots,\mD_k)$.
Then $g$ is critical for \eqref{Eq-Smix-g} with respect to adapted variations of metric, preserving the volume of $\Omega$, if and only if
the following Euler-Lagrange equations \eqref{E-delta-g-J}
%$\delta_g \bar J_\mD = \lambda\,g$
are satisfied for some $\lambda\in\RR$:
\begin{eqnarray}\label{ElmixDDvp-b}
 && {\delta Q}_\mu +{\delta_g\bar Q}_\mu +\Big(\,\overline{\rm S}_{\,\mD_1,\ldots,\mD_k} -\frac12\,\Div
 \sum\nolimits_{\,\nu}
 \big(\frac12\,P_\nu\tr_{\,\nu}^\bot(\I -\I^*) \nonumber \\
 && +\,\frac12\,P_\nu^\bot\tr_{\,\nu}^\top(\I -\I^*) +H_{\,\nu}+\tilde H_{\,\nu}\big) +\lambda\Big)\,g_\mu=0,
 \quad \mu=1,\ldots, k,
\end{eqnarray}
where ${\delta Q}_\mu$ and ${\delta_g\bar Q}_\mu$ are defined in Proposition~\ref{C-vr-terms}.
\end{Theorem}

\begin{proof}
Let a $\mD_\mu$-variation $g_t$ of $g$ (for some $\mu\ge 1$) be compactly supported in $\Omega\subset M$.
%For $X$ given in \eqref{E-prop-X}, we get $\frac{d}{dt}\int_\Omega (\Div X)\,{\rm d}\vol_g = 0$, see \eqref{E-DivThm-2}.
For
%any variation $g_t$ with ${\rm supp}\,(\dt g)\subset\Omega$ and
a $t$-dependent vector field $Y$ given in \eqref{E-prop-X},
%with ${\rm supp}\,(\dt X)\subset\Omega$,
by \eqref{E-dotvolg} and the Divergence Theorem, we obtain
%\[
%\label{E-DivThm-2}
 $\frac{d}{dt}\int_\Omega (\Div Y)\,{\rm d}\vol_g = \int_\Omega \Div\big(\dt Y+\frac12\,(\tr_g {B}) Y\big)\,{\rm d}\vol_g = 0$.
%\]
Thus, for $Q(\mD_\mu,g_t)$ and $\bar Q(\mD_\mu,g_t,\I)$ given in \eqref{E-func-Q} and \eqref{E-barQ}, using Proposition~\ref{P-decomp2}, we~obtain
\[
%\label{eq62}
 {\rm\frac{d}{dt}}\,\bar J_{\mD}(g_t,\I)
%_{\,|\,t=0}
  = \frac12\,{\rm\frac{d}{dt}}\int_{\Omega}
 \sum\nolimits_{\,\nu=1}^{\,k}\big(\bar Q(\mD_\nu,g_t,\I) + Q(\mD_\nu,g_t)\big)\,{\rm d}\vol_{g_t} . % \,|\,t=0
\]
Therefore,
\begin{eqnarray*}
 && {\rm\frac{d}{dt}}\int_{\Omega} \sum\nolimits_{\,\nu=1}^{\,k}\big(\bar Q(\mD_\nu,g_t,\I)+Q(\mD_\nu,g_t)\big)\,{\rm d}\vol_{g_t}  \\
 && =\int_{\Omega} \<{\delta_g\bar Q}_\mu+{\delta Q}_\mu,\ {B}_\mu\>\,{\rm d}\vol_{g_t}
 +\int_{\Omega} \sum\nolimits_{\,\nu=1}^{\,k}\big(\bar Q(\mD_\nu,g_t,\I)+Q(\mD_\nu,g_t)\big)\,\dt({\rm d}\vol_{g_t})\,.
\end{eqnarray*}
%In the case of a statistical connection $\bar\nabla$,
From \eqref{E-Q1Q2-gen}, \eqref{E-prop-X}, \eqref{E-dotvolg}, \eqref{E-dt-Q} and \eqref{E-dt-barQ}, we obtain
\begin{eqnarray*}
%\label{E-varJh-init2b}
 {\rm\frac{d}{dt}}\,\bar J_{\mD}(g_t,\I)_{|\,t=0} &=& \frac12\int_{\Omega} \big\<{\delta Q}_\mu +{\delta_g\bar Q}_\mu {+}\big(\,\overline{\rm S}_{\,\mD_1,\ldots,\mD_k}
 {-} \frac12\,\Div \sum\nolimits_{\,\nu}\big(\frac12\,P_\nu\tr_{\,\nu}^\bot(\I -\I^*) \\
 &+&\frac12\,P_\nu^\bot\tr_{\,\nu}^\top(\I -\I^*) +H_{\,\nu}+\tilde H_{\,\nu}\big) \big)\,g_\mu,\ {B}_\mu\big\>\,{\rm d}\vol_g.
\end{eqnarray*}
If $g$ is critical for $\bar J_{\mD}$ (with fixed $\I$) for $\mD_\mu$-variations,
% of metric,
then the above integral is zero for any symmetric $(0,2)$-tensor ${B}_\mu$.
This yields the $\mD_\mu$-component of the Euler-Lagrange equation
\begin{equation}\label{ElmixDD-b}
 {\delta Q}_\mu +{\delta_g\bar Q}_\mu +\big(\,\overline{\rm S}_{\,\mD_1,\ldots,\mD_k} -\frac12\,\Div \sum\nolimits_{\,\nu}\big(\frac12\,(P_\nu\tr_{\,\nu}^\bot +P_\nu^\bot\tr_{\,\nu}^\top)(\I -\I^*) +H_{\,\nu}+\tilde H_{\,\nu}\big) \big)\,g_\mu=0\,.
\end{equation}
According to \eqref{E-delta-g-J}, the Euler-Lagrange equation for \eqref{Eq-Smix-g} (with fixed $\I$) for adapted
%$\mD_\mu$-
variations of $g$ preser\-ving the volume of $\Omega$
%according to Remark~\ref{R-volume},
is \eqref{ElmixDDvp-b} instead of \eqref{ElmixDD-b}.
\end{proof}

\begin{Remark}\rm
One can present \eqref{ElmixDDvp-b} in the form of \eqref{E-geom} given by its restrictions on $\mD_\mu$,
%An equivalent form of \eqref{ElmixDDvp-b} is \eqref{E-geom} with $\overline\Ric_{\,\mD}$ given by its restrictions on subbundles $\mD_\mu$,
\begin{equation}\label{E-main-0ij-kk}
 \overline\Ric_{\,\mD\,|\,\mD_\mu\times\mD_\mu} = - {\delta Q}_\mu - {\delta_g\bar Q}_\mu + \rho_\mu\,g_\mu, \quad \mu=1,\ldots,k\,,
\end{equation}
(see \cite{rz-3} for $k=2$), where $\rho_\mu$ are defined in \eqref{E-mu-k} below.
Indeed,
\begin{equation}\label{Eq-2-Ric}
 \overline\Ric_{\,\mD\,|\,\mD_\mu\times\mD_\mu} = \Ric_{\,\mD\,|\,\mD_\mu\times\mD_\mu} - {\delta_g\bar Q}_\mu, \quad \mu=1,\ldots,k\,,
\end{equation}
and it was shown in \cite{r-EH-k} that
\begin{equation}\label{E-main-0ij-k}
 \Ric_{\,\mD\,|\,\mD_\mu\times\mD_\mu} = -{\delta Q}_\mu +\rho_\mu \,g_\mu, \quad \mu=1,\ldots,k\,.
\end{equation}
that corresponds to the Euler-Lagrange equations \eqref{ElmixDDvp-b} for $\,\I=0$:
\begin{equation}\label{ElmixDDvp}
 {\delta Q}_\mu =-\big(\,{\rm S}_{\,\mD_1,\ldots,\mD_k} -\frac12\,\Div\sum\nolimits_{\,\nu=1}^{\,k}(H_\nu + \tilde{H}_\nu) +\lambda\big)\,g_\mu,
 \quad \mu=1,\ldots, k\,,
\end{equation}
and $(\rho_1,\ldots,\rho_k)$ in \eqref{E-main-0ij-k} for $n=\dim M>2$ are given by
\begin{equation}\label{E-mu-k}
 \rho_{\,\nu}=-\frac1{2n-4}\,\big(\sum\nolimits_{\,\mu}\,(a_{\,\nu}-a_{\mu})\,n_{\mu}-2\,a_{\,\nu}\big),
\end{equation}
with coefficients
%\[
 $a_\mu= \tr_g(\sum\nolimits_{\,\nu}{\delta Q}_\nu) -2\,{\delta Q}_\mu$.
%\]
From \eqref{Eq-2-Ric} and \eqref{E-main-0ij-k} the system \eqref{E-main-0ij-kk} follows.
\end{Remark}

\begin{Example}[Case $k=2$]\rm
For $(M,g;\mD,\mD^\bot)$ with a statistical connection,
the tensor $\overline\Ric_{\,\mD}$ in~\eqref{E-geom} is defined
by its restrictions on complementary subbundles $\mD$ and $\mD^\bot$ of $TM$,
\begin{eqnarray*}
%\label{E-main-0ij-22}
 \overline\Ric_{\,\mD\,|\,\mD\times\mD} =
 \Div{h} +{\cal K}^\flat +\tilde{H}^\flat\otimes\tilde{H}^\flat -\frac12\Upsilon_{\tilde h, \tilde h} -\frac12\Upsilon_{\tilde T,\tilde T}
 -2\,{\cal T}^\flat - {\delta_g\bar Q}_1 + (\rho_1 - \Div H)\,g^\top, \\
%%%%%
 \overline\Ric_{\,\mD\,|\,\mD^\bot\times\mD^\bot} =
 \Div\tilde{h} +\tilde{\cal K}^\flat +{H}^\flat\otimes{H}^\flat -\frac12\Upsilon_{h,h} -\frac12\Upsilon_{T,T} -2\,\tilde{\cal T}^\flat
 - {\delta_g\bar Q}_2 + (\rho_2 -\Div\tilde H)\,g^\bot ,
\end{eqnarray*}
see \cite{rz-connections,rz-3}, where
%\[
 $\rho_1=-\frac{n_1-1}{n-2}\,\Div(\tilde H-{H})$,
%\quad
% and
 $\rho_2=\frac{n_2-1}{n-2}\,\Div(\tilde H-{H})$,
%\]
and $n=\dim M>2$. If~$n=2$,
%(and $k=2$),
then $\rho_1=\rho_2=0$. %, see \cite{r2018}.
The (0,2)-tensors ${\delta_g\bar Q}_1: \mD\times\mD\to\RR$ and ${\delta_g\bar Q}_2: \mD^\bot\times\mD^\bot\to\RR$ are given
(using adapted frame, $E_a\in\mD$ and ${\cal E}_i\in\mD^\bot$) by
\begin{eqnarray*}
 &&\hskip-4mm {\delta_g\bar Q}_1(E_a,{E}_b) =
 \<\I_{{a}}, \I_{{b}}\>_{\,|\,\mD^\bot} -\<\I_{{a}}{E}_{b},\,\tr_{\,\mD^\bot}\I\> -2\,\<\tilde A_{{E}_{a}},\,\I_{{b}}\> \\
 && +\,2\,\< \tilde T^{\sharp}_{{E}_{a}},\, \I_{{b}}\> + 6\,\<T({E}_{b},\,\cdot),\, \I_{{a}}\> + 2\,\< h({E}_{b},\,\cdot),\, \I_{{a}}\> \\
 && -\,\<\I_{{b}}{E}_{a}, H - \tilde H\> -\<\tr_{\,\mD}\I, {E}_{a}\>\< \tilde H, {E}_{b}\> +\<\tr_{\,\mD^\bot}\I, {E}_{b}\> \<\tilde H, {E}_{a}\>\,,\\
%%%%%%%%%%%%%%%%%%%%%%%%%%%%%%%%%%%
 &&\hskip-4mm {\delta_g\bar Q}_2({\cal E}_i,{\cal E}_j) =
 \<\I_{{\cal E}_i}, \I_{{\cal E}_j}\>_{\,|\,\mD} -\<\I_{{\cal E}_i}{\cal E}_j,\,\tr_{\,\mD}\I\> -2\,\<A_{{\cal E}_i},\, \I_{{\cal E}_j}\> \\
 && +\,2\,\< T^{\sharp}_{{\cal E}_i},\, \I_{{\cal E}_j}\> + 6\,\< \tilde T({\cal E}_j,\,\cdot),\, \I_{{\cal E}_i}\> + 2\,\< \tilde h({\cal E}_j,\,\cdot),\, \I_{{\cal E}_i}\> \\
 && +\,\<\I_{{\cal E}_j}{\cal E}_i, \tilde H - H\> -\<\tr_{\,\mD^\bot}\I, {\cal E}_i\>\< H, {\cal E}_j\> +\<\tr_{\,\mD}\I, {\cal E}_j\> \<H, {\cal E}_i\>\,.
\end{eqnarray*}
\end{Example}

\section{Variations with respect to contorsion tensor}
\label{sec:contorsion}

Here, we consider action \eqref{Eq-Smix-g} with fixed metric $g$, as a functional of $\I$.
The Euler-Lagrange equation for
%the functional
\eqref{Eq-Smix-g}
%of the contorsion tensor $\I$,
with fixed adapted metric,
means that the partial gradient vanishes,
\begin{equation}\label{E-delta-I-J}
 \delta_\I\bar J_\mD=0\,,
\end{equation}
where
%\[
 ${\rm\frac{d}{dt}}\,\bar J_\mD(g,\I_t)_{|\,t=0} =\int_{\Omega}\<\delta_\I \bar J_\mD, \overset{\centerdot}\I\>\,{\rm d}\vol_g$
%\]
for any variation $\I_t$ with $\overset{\centerdot}\I =\dt \I_{t_{|\,t=0}}$. In what follows, we consider particular components of \eqref{E-delta-I-J}, defined by the distributions.
We adapt notation from \cite{rz-3} to the case of several distributions.
Greek letters $\mu, \rho, \lambda, \nu$ are used for indices of pairwise orthogonal distributions spanning the tangent bundle, $TM = \bigoplus_{\mu} \mD_\mu$;
$\delta_{a,b} =1$ if $a=b$ and 0 otherwise.
%In particular, here $h_\mu$ is the second fundamental form of the distribution $\mD_\mu$,
%but the second fundamental form of $\mD_\mu^\perp$ is denoted by ${\tilde h}_\mu$.
%(this notation makes easier using previously obtained equations, we can change it later).
%For any family $\I_t\ (t\ge0)$ of contorsion tensors, set $\overset{\centerdot}\I=d\I_t/dt$ for $t=0$.

\begin{Proposition}\label{P-03}
A contorsion tensor $\I$ is critical for the action \eqref{Eq-Smix-g} with fixed adapted metric $g$ on $(M;\mD_1,\ldots,\mD_k)$
%\eqref{actionI},
if and only if the following Euler-Lagrange equations %, %hold, see
\eqref{E-delta-I-J} hold:
\begin{subequations}
\begin{eqnarray}
\label{ELI1}
&& \<\tr_\mu^\bot{\mathfrak T}^*-\tH_\mu, {E}_{\mu, c}\>\delta_{a,b} +\<\tr_\mu^\bot{\mathfrak T}+\tH_\mu, {E}_{\mu, b}\>\delta_{a,c}=0\,, \\
&& \label{ELI2}
\<\tr_\mu^\bot{\mathfrak T}^* + H_\mu, E_{\rho, i}\>\delta_{a,b}
-\< (h_\mu - T_\mu) ({E}_{\mu, a}, {E}_{\mu, b} ), E_{\rho, i}\> -\<\I_{{\rho, i}}{E}_{\mu, a}, {E}_{\mu, b}\> = 0\,, \\
&& \label{ELI3}
\<\tr_\mu^\bot{\mathfrak T} - H_\mu, E_{\rho, i}\>\delta_{a,b}+ \< (h_\mu + T_\mu)( {E}_{\mu, b}, {E}_{\mu, a}), {E}_{\rho,i}\>
-\<\I_{{\rho, i}}{E}_{\mu, b}, {E}_{\mu, a}\> = 0\,, \\
&& \label{ELI4}
 \<2\,T_\mu ({E}_{\mu, a}, {E}_{\mu, b} ), E_{\rho,i}\> +\<\I_{{\mu, a}}{E}_{\mu, b} + \I^*_{{\mu, b}}{E}_{\mu, a}, {E}_{\rho, i}\> =0\,, \\
&& \label{ELI5}
 \<2\,{\tilde T}_\mu ({E}_{\rho, j}, {E}_{\xi, l} ), E_{\mu, a }\>
%\nonumber \\
%%%ija->ijl; mu->lambda
%&&
 +\,\<({\tilde h}_\xi + {\tilde T}_\xi )( {E}_{\rho, j}, {E}_{\mu, a} ), E_{\xi,l}\>
 + 2\<\I_{{\xi, l}}{E}_{\mu, a}, {E}_{\rho, j}\>\nonumber \\
%%%
&&
%%%iaj -> ijl
 -\,\< ({\tilde h}_\rho + {\tilde T}_\rho )( {E}_{\xi, l}, {E}_{\mu, a} ), E_{\rho, j} \>
 +2\<\I_{{\rho, j}}{E}_{\xi, l}, {E}_{\mu, a}\> = 0\,,
\end{eqnarray}
\end{subequations}
for all
$\mu,\rho,\xi\in\{1, \ldots k \}$, such that $\rho\ne\mu$, $\xi\notin\{\mu, \rho\}$,
%$i\in\{1, \ldots, n_\mu \}$
% $i,j \in \{ 1, \ldots, n_\rho \}$ and $l\in\{ 1, \ldots, n_\xi \}$.
and for all $a,b,c\in\{1,\ldots,n_\mu \}$, $i,j\in\{1,\ldots,n_\rho \}$ and $l \in\{1,\ldots,n_\xi \}$.
\end{Proposition}

\begin{proof}
 Using \eqref{E-Dk-Smix} and
the formula for two complementary distributions,
%$\mD$ and $\mD^\bot$,
 see \cite[Theorem~2]{rz-3},
%\eqref{dtQ1Q2rz3}
%and \cite[Theorem 2]{rz-3},
we~get for $k>2$ distributions the following:
\begin{eqnarray} \label{dtSD1DkI}
%\label{dtQ1Q2}
&&\quad 2\,{\frac{\rm d}{\rm dt}\int_M \bar{\rm S}_{\mD_1, \ldots \mD_k }(\I_t)\,{\rm d} \vol_g}\,|_{\,t=0}
=\frac12\int_M \sum\nolimits_{\mu } \sum\Big\{\<\overset{\centerdot}\I_{{\mu,a}}{E}_{\mu, b},{E}_{\mu, c}\>\times\nonumber\\
&&\hskip-1mm \times\big(\<\tr_\mu^\bot{\mathfrak T}^*-\tH_\mu, {E}_{\mu, c}\>\delta_{a,b}
%\<{E}_{\mu, a}, {E}_{\mu, b}\>
+\<\tr_\mu^\bot{\mathfrak T}+\tH_\mu, {E}_{\mu, b}\>\delta_{a,c}
%\<{E}_{\mu, a}, {E}_{\mu, c}\>
\big)\nonumber\\
&&\hskip-1mm +\,\<\overset{\centerdot}\I_{{\mu, a}}{E}_{\mu, b}, {\cal E}_{\mu, i}\>\big(\<\tr_\mu^\bot{\mathfrak T}^*
+ H_\mu, {\cal E}_{\mu, i}\>\delta_{a,b}
%\<{E}_{\mu, a}, {E}_{\mu, b}\>
-\<({A}_{\mu, i}{-}{T}^\sharp_{\mu, i}){E}_{\mu, a}, {E}_{\mu, b}\> - \<\I_{\mu, i}{E}_{\mu, a}, {E}_{\mu, b}\> \big)\nonumber\\
&&\hskip-1mm +\,\<\overset{\centerdot}\I_{{\mu, a}}{\cal E}_{\mu, i}, {E}_{\mu, b}\>\big(\<\tr_\mu^\bot{\mathfrak T}{-} H_\mu, {\cal E}_{\mu, i}\>\delta_{a,b}
%\<{E}_{\mu, a}, {E}_{\mu, b}\>
+ \<({A}_{\mu, i}+{T}^\sharp_{\mu, i}){E}_{\mu, b}, {E}_{\mu, a}\> - \<\I_{\mu, i}{E}_{\mu, b}, {E}_{\mu, a}\>\big)\nonumber\\
&&\hskip-1mm +\,\<\overset{\centerdot}\I_{{\mu, a}}{\cal E}_{\mu, i}, {\cal E}_{\mu, j}\> \big(\<(\tilde{A}_{\mu, a}{-} \tilde{T}^\sharp_{\mu, a}){\cal E}_{\mu, i}, {\cal E}_{\mu, j}\>
{-}\<(\tilde{A}_{\mu, a}{+} \tilde{T}^\sharp_{\mu, a}){\cal E}_{\mu, i}, {\cal E}_{\mu, j}\>
{-} \<\I_{\mu, i}{\cal E}_{\mu, j}{+} \I^*_{\mu, j}{\cal E}_{\mu, i}, {E}_{\mu, a}\> \big)\nonumber\\
%\end{eqnarray*}
%\begin{eqnarray*}
%the dual to the above is the following:
&&\hskip-1mm +\,\<\overset{\centerdot}\I_{{\cal E}_{\mu, i}}{\cal E}_{\mu, j}, {\cal E}_{\mu, l}\>\big(\<\tr_\mu^\top{\mathfrak T}^* - H_\mu, {\cal E}_{\mu, l}\>\delta_{i,j}
%\<{\cal E}_{\mu, i}, {\cal E}_{\mu, j}\>
+\<\tr_\mu^\top{\mathfrak T} +H_\mu, {\cal E}_{\mu, j}\>\delta_{i,l}
%\<{\cal E}_{\mu, i}, {\cal E}_{\mu, l}\>
\big)\nonumber\\
&&\hskip-1mm +\,\<\overset{\centerdot}\I_{{\cal E}_{\mu, i}}{\cal E}_{\mu, j}, {E}_{\mu,a}\>\big(\<\tr_\mu^\top{\mathfrak T}^*
+\tH_\mu, {E}_{\mu, a}\>\delta_{i,j}
%\<{\cal E}_{\mu, i}, {\cal E}_{\mu, j}\>
-\<(\tilde{A}_{\mu, a}+\tilde{T}^\sharp_{\mu, a}){\cal E}_{\mu, j}, {\cal E}_{\mu, i}\> -\<\I_{\mu, a}{\cal E}_{\mu, i}, {\cal E}_{\mu, j}\> \big) \nonumber\\
&&\hskip-1mm +\,\<\overset{\centerdot}\I_{{\cal E}_{\mu, i}}{E}_{\mu, a}, {\cal E}_{\mu,j}\>\big(\<\tr_\mu^\top{\mathfrak T}-\tH_\mu, {E}_{\mu, a}\>\delta_{i,j}
%\<{\cal E}_{\mu, i}, {\cal E}_{\mu, j}\>
+\<(\tilde{A}_{\mu, a} +\tilde{T}^\sharp_{\mu, a}){\cal E}_{\mu, j}, {\cal E}_{\mu, i}\> - \<\I_{\mu, a}{\cal E}_{\mu, j}, {\cal E}_{\mu, i}\> \big) \nonumber\\
\nonumber
&& \hskip-1mm +\,\<\overset{\centerdot}\I_{{\cal E}_{\mu, i}}{E}_{\mu, a}, {E}_{\mu, b}\> \big(\<({A}_{\mu, i} - {T}^\sharp_{\mu, i}){E}_{\mu, a}, {E}_{\mu, b}\>
-\<({A}_{\mu, i} + {T}^\sharp_{\mu, i}){E}_{\mu, a}, {E}_{\mu, b}\> \\
&&\hskip-1mm -\,\<\I_{\mu, a}{E}_{\mu, b} + \I^*_{\mu, b}{E}_{\mu, a}, {\cal E}_{\mu, i}\> \big) \Big\}\,{\rm d}\vol_g\,,
%\nonumber
\end{eqnarray}
where we used notation ${A}_{\mu, i} = ({A}_{\mu})_{ {\cal E}_{\mu , i} }$ and $\tilde{A}_{\mu, a} = (\tilde{A}_{\mu})_{E_{\mu, a}}$ etc.
%(the second sum is over repeated second indices)
In \eqref{dtSD1DkI}, terms with the same coefficients of tensor $\overset{\centerdot}\I$ appear in different forms in sum over $\mu$, e.g., term $\<\overset{\centerdot}\I_{{\cal E}_{1, i}}{E}_{1, a}, {E}_{1, b}\>$, where ${\cal E}_{1, i} \in \mD_2$ coincides with some term $\<\overset{\centerdot}\I_{{2, a}}{\cal E}_{2, i}, {\cal E}_{2, j}\>$ and some terms
$\<\overset{\centerdot}\I_{{\cal E}_{\mu, l}}{\cal E}_{\mu, i}, {\cal E}_{\mu, j}\>$ for $\mu \notin \{1,2\}$.
%
%Hence, we shall use disjoint sets of vectors $\{ E_{\mu, a},\ a= 1, \ldots, n_\mu \}$ from the full orthonormal frame
%$\{ E_{\mu, a},\ a= 1, \ldots, n_\mu;\ \mu =1,\ldots, l \}$.
To~relate the indices of various elements $\{E_{\mu, a}\}$ and $\{{\cal E}_{\nu, i}\}$ of the whole frame,
%for all $\nu, \mu \in \{1, \ldots, l\}$ and $a \in \{1, \ldots, n_\mu \}$,
let $\iota(\nu,\mu,a)$ be such that $E_{\mu, a} = {\cal E}_{\nu,\,\iota(\nu,\mu,a)}$.
Since $\overset{\centerdot}\I$ is arbitrary, the equality $\frac{\rm d}{\rm dt}\,\bar J_{\mD}(g,\I_t)\,|_{\,t=0} =0$ is valid for all $\I_t$ if and only if all coefficients of terms with $\overset{\centerdot}\I$ in \eqref{dtSD1DkI} vanish.
For fixed $\mu$ and $a,b,c \in \{1, \ldots, n_\mu \}$, we consider the term of \eqref{dtSD1DkI} with $ \<\overset{\centerdot}\I_{{\mu,a}}{E}_{\mu, b},{E}_{\mu, c}\>$, which comes from one term with $\<\overset{\centerdot}\I_{{\mu,a}}{E}_{\mu, b},{E}_{\mu, c}\>$ and $k-1$ terms with $\<\overset{\centerdot}\I_{{\cal E}_{\nu, \iota(\nu,\mu,a)}}{\cal E}_{\nu, \iota(\nu,\mu,b)}, {\cal E}_{\nu, \iota(\nu,\mu,c)}\>$ for $\mD_\nu$ with $\nu \ne \mu$:
\begin{eqnarray}\label{ELI1term}
&& \<\overset{\centerdot}\I_{E_{\mu,a}}{E}_{\mu, b},{E}_{\mu, c}\>\,
%\nonumber \\ &&\hskip-2mm \times
\big(\<\tr_\mu^\bot{\mathfrak T}^*-\tH_\mu, {E}_{\mu, c}\>\delta_{a,b}
%\<{E}_{\mu, a}, {E}_{\mu, b}\>
+\<\tr_\mu^\bot{\mathfrak T}+\tH_\mu, {E}_{\mu, b}\>\delta_{a,c}
%\<{E}_{\mu, a}, {E}_{\mu, c}\>
\nonumber\\
&& +\sum\nolimits_{\nu\ne\mu} \big(\<\tr_\nu^\top{\mathfrak T}^* - H_\nu, {E}_{\mu, c}\>\delta_{a,b}
%\<E_{\mu,a}, {E}_{\mu, b}\>
+\<\tr_\nu^\top{\mathfrak T} +H_\nu, {E}_{\mu, b}\>\delta_{a,c}
%\<E_{\mu,a}, {E}_{\mu, c}\>
\big)\big).
\end{eqnarray}
On the other hand, we have
\begin{eqnarray}\label{ELI1termaux}
&& \<\tr_\mu^\bot{\mathfrak T}^*-\tH_\mu, {E}_{\mu, c}\>\delta_{a,b}
%\<{E}_{\mu, a}, {E}_{\mu, b}\>
+\<\tr_\mu^\bot{\mathfrak T}+\tH_\mu, {E}_{\mu, b}\>\delta_{a,b}
%\<{E}_{\mu, a}, {E}_{\mu, c}\>
\nonumber \\
&& = \sum\nolimits_{\nu\ne\mu} \big(\<\tr_\nu^\top{\mathfrak T}^* - H_\nu, {E}_{\mu, c}\>\delta_{a,b}
%\<E_{\mu,a}, {E}_{\mu, b}\>
+\<\tr_\nu^\top{\mathfrak T} +H_\nu, {E}_{\mu, b}\>\delta_{a,c}
%\<E_{\mu,a},{E}_{\mu, c}\>
\big).
\end{eqnarray}
If $\I$ is a critical point of \eqref{Eq-Smix-g} with fixed metric,
%\eqref{actionI},
then the coefficient of $\<\overset{\centerdot}\I_{{\mu,a}}{E}_{\mu, b},{E}_{\mu, c}\>$ in \eqref{dtSD1DkI}, given in \eqref{ELI1term}, is zero for every $a,b,c$, so by \eqref{ELI1term} and \eqref{ELI1termaux} the first Euler-Lagrange equation \eqref{ELI1} follows.
%\begin{eqnarray} \label{ELI1t}
%\<\tr_\mu^\bot{\mathfrak T}^*-\tH_\mu, {E}_{\mu, c}\>\,\<{E}_{\mu, a}, {E}_{\mu, b}\> +\<\tr_\mu^\bot{\mathfrak T}+\tH_\mu, {E}_{\mu, b}\>\,\<{E}_{\mu, a}, %{E}_{\mu, c}\> = 0
%\end{eqnarray}
%for all $\mu$ and all $a,b,c \in \{1, \ldots, n_\mu \}$.
%Now we consider the following term:
%\begin{eqnarray*}
%&& \,\<\overset{\centerdot}\I_{{\mu, a}}{E}_{\mu, b}, {\cal E}_{\mu, i}\>\big(\<\tr_\mu^\bot{\mathfrak T}^* + H_\mu, {\cal E}_{\mu, i}\>\,\<{E}_{\mu, a}, {E}_{\mu, b}\>
%-\<({A}_{\mu, i} - {T}^\sharp_{\mu, i}){E}_{\mu, a}, {E}_{\mu, b}\> - \<\I_{\mu, i}{E}_{\mu, a}, {E}_{\mu, b}\> \big) \\
%&& = \<\overset{\centerdot}\I_{E_{\mu, a}}{E}_{\mu, b}, {\cal E}_{\mu, i}\>\big(\<\tr_\mu^\bot{\mathfrak T}^* + H_\mu, {\cal E}_{\mu, i}\>\,\<{E}_{\mu, a}, {E}_{\mu, b}\>
%-\<( (h_\mu - T_\mu) ({E}_{\mu, a}, {E}_{\mu, b} ), {\cal E}_{\mu, i}\> - \<\I_{\mu, i}{E}_{\mu, a}, {E}_{\mu, b}\> \big)
%\end{eqnarray*}
%For fixed $\mu,a,b,\rho,i$, we have the following term in \eqref{dtSD1DkI}

Let $\mu \ne \rho$,
%and let ${\cal E}_{\mu, i} = E_{\rho, i}$ above, i.e., we consider ${\cal E}_{\mu, i}$ as a vector from a particular distribution.
%Then we obtain the following term of \eqref{dtSD1DkI}:
fix $a,b \in \{ 1, \ldots, n_\mu \}$ and $i \in \{ 1, \ldots, n_\rho \}$. Then we get the following term in~\eqref{dtSD1DkI}, which comes from one term with $\<\overset{\centerdot}\I_{{\mu, a}}{E}_{\mu, b}, {\cal E}_{\mu, \iota(\mu, \rho, i )}\>$, one term with $\<\overset{\centerdot}\I_{{\cal E}_{\rho, \iota(\rho, \mu, a)}}{\cal E}_{\rho, \iota(\rho, \mu, b) }, E_{\rho, i}\>$ and, if $k\ge 3$, $(k-2)$ terms with $\<\overset{\centerdot}\I_{{\cal E}_{\nu, \iota(\nu, \mu,a) }}{\cal E}_{\nu, \iota(\nu, \mu,b) }, {\cal E}_{\nu, \iota(\nu, \rho,i) }\>$:
\begin{eqnarray} \label{ELI2termaux}
&& \<\overset{\centerdot}\I_{{\mu, a}}{E}_{\mu, b}, E_{\rho, i}\> \big(\<\tr_\mu^\bot{\mathfrak T}^* + H_\mu, E_{\rho, i}\>\delta_{a,b}
%\<{E}_{\mu, a}, {E}_{\mu, b}\>
 {-}\< (h_\mu - T_\mu) ({E}_{\mu, a}, {E}_{\mu, b} ), E_{\rho, i}\>
 {-}\<\I_{{\rho, i}}{E}_{\mu, a}, {E}_{\mu, b}\>\nonumber \\
&&  +\,\big(\<\tr_\rho^\top{\mathfrak T}^*+\tH_\rho, {E}_{\rho, i}\>\delta_{a,b}
 %\< E_{\mu,a}, E_{\mu,b}\>
 -\< ({\tilde h}_\rho - \tilde{T}_\rho ) ({E}_{\mu, a}, {E}_{\mu, b}), E_{\rho, i}\>
 - \<\I_{{\rho,i}}{E}_{\mu, a}, {E}_{\mu, b}\> \big) \nonumber \\
&&
% \hskip-5mm
 +\sum\nolimits_{\nu \notin \{\mu, \rho\}} \!\big (%ijk -> abi
 \<\tr_\nu^\top{\mathfrak T}^* {-} H_\nu, E_{\rho,i}\>\delta_{a,b}
 %\nonumber \\ &&
 %\< E_{\mu, a}, {E}_{\mu, b}\>
 +\<\tr_\nu^\top{\mathfrak T}+ H_\nu, {E}_{\mu, b}\>\,\<{E}_{\mu, a}, {E}_{\rho, i}\>\big)\big).
\end{eqnarray}
We have $\<{E}_{\mu, a}, {E}_{\rho, i}\> =0$ as they belong to different, orthogonal distributions. Moreover,
\begin{eqnarray*}
 && \< ({\tilde h}_\rho - \tilde{T}_\rho ) ({E}_{\mu, a}, {E}_{\mu, b}), E_{\rho, i}\> = \<( h_\mu - T_\mu ) ({E}_{\mu, a}, {E}_{\mu, b}), E_{\rho, i}\>,\\
 && \<\tr_\rho^\top{\mathfrak T}^*+\tH_\rho, {E}_{\rho, i}\> + \sum\nolimits_{\nu \notin \{\mu, \rho\}} \<\tr_\nu^\top{\mathfrak T}^* - H_\nu, E_{\rho,i}\> =
 \<\tr_\mu^\bot{\mathfrak T}^* + H_\mu, E_{\rho, i}\>,\\
 &&\<\tH_\rho, E_{\rho, i}\> = \< \sum\nolimits_{\nu \notin \{\mu, \rho \}}H_\nu + H_\mu, E_{\rho, i}\>,\quad
 \tr^\perp_\mu \I^* = \sum\nolimits_{\nu \notin \{\mu, \rho \}}\tr^\top_\nu \I^* + \tr^\top_\rho \I^*\,.
\end{eqnarray*}
Using the above, we obtain that \eqref{ELI2termaux} vanishes for all $\<\overset{\centerdot}\I_{{\mu, a}}{E}_{\mu, b}, E_{\rho, i}\>$ if and only if the
second Euler-Lagrange equation \eqref{ELI2} holds.
%we get the second Euler-Lagrange equation \eqref{ELI2}.
%
%is the following:
%\begin{eqnarray} \label{ELI2t}
%% \<\overset{\centerdot}\I_{{\mu, a}}{E}_{\mu, b}, E_{\rho, i}\>\big(
% \<\tr_\mu^\bot{\mathfrak T}^* + H_\mu, E_{\rho, i}\>\,\<{E}_{\mu, a}, {E}_{\mu, b}\>
%-\< (h_\mu - T_\mu) ({E}_{\mu, a}, {E}_{\mu, b} ), E_{\rho, i}\> - \<\I_{{\rho, i}}{E}_{\mu, a}, {E}_{\mu, b}\> = 0
%%\big) = 0,
%\end{eqnarray}
%for all $\mu \in \{1, \ldots k \}$,
%% all $\rho \ne \mu$,
%all $a,b \in \{1, \ldots, n_\mu \}$ and all $i \in \{ 1, \ldots, n_\rho \}$.

%Next, we consider term
%\begin{eqnarray*}
%&& \<\overset{\centerdot}\I_{{\mu, a}}{\cal E}_{\mu, i}, {E}_{\mu, b}\>\big(\<\tr_\mu^\bot{\mathfrak T} - H_\mu, {\cal E}_{\mu, i}\>\,\<{E}_{\mu, a}, {E}_{\mu, b}\>
%+ \<({A}_{\mu, i} + {T}^\sharp_{\mu, i}){E}_{\mu, b}, {E}_{\mu, a}\> - \<\I_{\mu, i}{E}_{\mu, b}, {E}_{\mu, a}\>\big) = \\
%&&
%\<\overset{\centerdot}\I_{{\mu, a}}{\cal E}_{\mu, i}, {E}_{\mu, b}\>\big(\<\tr_\mu^\bot{\mathfrak T} - H_\mu, {\cal E}_{\mu, i}\>\,\<{E}_{\mu, a}, {E}_{\mu, b}\>
%+ \< (h_\mu + T_\mu)( {E}_{\mu, b}, {E}_{\mu, a}), {\cal E}_{\mu,i}\>
% - \<\I_{\mu, i}{E}_{\mu, b}, {E}_{\mu, a}\>\big)
%\end{eqnarray*}
%with $ {\cal E}_{\mu, i}= E_{\rho,i}$.

Let $\mu \ne \rho$, %and let ${\cal E}_{\mu, i} = E_{\rho, i}$ above, i.e., we consider ${\cal E}_{\mu, i}$ as a vector from a particular distribution. Then we obtain the following term of \eqref{dtSD1DkI}:
for fixed $a,b \in \{ 1, \ldots, n_\mu \}$ and $i\in\{1,\ldots, n_\rho \}$ we get the following term in \eqref{dtSD1DkI},
coming from $\<\overset{\centerdot}\I_{{\mu, a}} {\cal E}_{\mu, \iota(\mu, \rho, i )}, {E}_{\mu, b}\>$, $\<\overset{\centerdot}\I_{{\cal E}_{\rho, \iota(\rho, \mu, a)}} E_{\rho, i}, {\cal E}_{\rho, \iota(\rho, \mu, b) }\>$ and $\<\overset{\centerdot}\I_{{\cal E}_{\nu, \iota(\nu, \mu,a) }}{\cal E}_{\nu, \iota(\nu, \rho,i) }, {\cal E}_{\nu, \iota(\nu, \mu,b) }\>$:
%We have the following term in \eqref{dtSD1DkI}:
\begin{eqnarray} \label{ELI3termaux}
&&
%\hskip-4mm
\<\overset{\centerdot}\I_{{\mu, a}}{E}_{\rho, i}, {E}_{\mu, b}\>\big(\<\tr_\mu^\bot{\mathfrak T} {-} H_\mu, E_{\rho, i}\>\delta_{a,b}
%\<{E}_{\mu, a}, {E}_{\mu, b}\>
 {+}\< (h_\mu + T_\mu)( {E}_{\mu, b}, {E}_{\mu, a}), {E}_{\rho,i}\>
 {-}\<\I_{{\rho, i}}{E}_{\mu, b}, {E}_{\mu, a}\>\nonumber \\
&&
%\hskip-6mm
 +
%%%iaj->aib
 \,\<\tr_\rho^\top{\mathfrak T}-\tH_\rho, {E}_{\rho, i}\>\delta_{a,b}
%\< E_{\mu, a}, E_{\mu, b}\>
+ \< ({\tilde h}_{\rho} + {\tilde T}_\rho )(E_{\mu, b}, E_{\mu, a}), E_{\rho, i}\>
-\<\I_{{\rho, i}} E_{\mu, b}, E_{\mu, a}\>\nonumber \\
%%%
&&
%%%ijk->aib
 +\sum\nolimits_{\nu \notin \{\mu, \rho \}}\!\big(\<\tr_\nu^\top{\mathfrak T}^* - H_\nu, E_{\mu,b}\>\,\<E_{\mu, a}, E_{\rho, i}\>
 %\nonumber \\ &&
 +\<\tr_\nu^\top{\mathfrak T}+ H_\nu, E_{\rho, i}\>\delta_{a,b}
% \<E_{\mu, a}, E_{\mu,b}\>
 \big)\big).
\end{eqnarray}
We have $\<E_{\mu, a}, E_{\rho, i}\> =0$ and
\begin{eqnarray*}
 &&\< ({\tilde h}_{\rho} + {\tilde T}_\rho )(E_{\mu, b}, E_{\mu, a}), E_{\rho, i}\> = \< (h_{\mu} + T_\mu )(E_{\mu, b}, E_{\mu, a}), E_{\rho, i}\>,\\
%%%
 &&\< \tH_\rho, E_{\rho, i}\> = \< \sum\nolimits_{\nu \notin \{\mu, \rho \}}H_\nu + H_\mu, E_{\rho, i}\>,\quad
 \tr^\perp_\mu \I = \sum\nolimits_{\nu \notin \{\mu, \rho \}}\tr^\top_\nu \I + \tr^\top_\rho \I\,.
\end{eqnarray*}
%Hence,
Using the above, we obtain that \eqref{ELI3termaux} vanishes for all $\<\overset{\centerdot}\I_{{\mu, a}}{E}_{\rho, i}, {E}_{\mu, b}\>$ if and only if the third Euler-Lagrange equation \eqref{ELI3} holds.
%\begin{eqnarray} \label{ELI3t}
% \<\tr_\mu^\bot{\mathfrak T} - H_\mu, E_{\rho, i}\>\,\<{E}_{\mu, a}, {E}_{\mu, b}\>
%+ \< (h_\mu + T_\mu)( {E}_{\mu, b}, {E}_{\mu, a}), {E}_{\rho,i}\>
%- \<\I_{{\rho, i}}{E}_{\mu, b}, {E}_{\mu, a}\> = 0
%\end{eqnarray}
%for all $\mu \in \{1, \ldots k \}$, % all $\rho \ne \mu$,
%all $a,b \in \{1, \ldots, n_\mu \}$ and all $i \in \{ 1, \ldots, n_\rho \}$.

%Next, we consider terms like
%\begin{eqnarray*}
%\<\overset{\centerdot}\I_{{\mu, i}}{E}_{\mu, a}, {E}_{\mu, b}\> (\<({A}_{\mu, i}{-}{T}^\sharp_{\mu, i}){E}_{\mu, a}, {E}_{\mu, b}\>
%{-}\<({A}_{\mu, i}{+}{T}^\sharp_{\mu, i}){E}_{\mu, a}, {E}_{\mu, b}\> {-} \<\I_{\mu, a}{E}_{\mu, b}{+} \I^*_{\mu, b}{E}_{\mu, a}, {\cal E}_{\mu, i}\> )
%\end{eqnarray*}

Let $\mu \ne \rho$, %and let ${\cal E}_{\mu, i} = E_{\rho, i}$ above, i.e., we consider ${\cal E}_{\mu, i}$ as a vector from a particular distribution. Then we obtain the following term of \eqref{dtSD1DkI}:
for fixed $a,b \in \{ 1, \ldots, n_\mu \}$ and $i \in \{ 1, \ldots, n_\rho \}$ we get the following term in \eqref{dtSD1DkI}, coming from
%terms
$\<\overset{\centerdot}\I_{{\cal E}_{\mu, \iota(\mu, \rho, i )}}{E_{\mu, a}}, {E}_{\mu, b}\>$, $\<\overset{\centerdot}\I_{{\rho, i}} {\cal E}_{\rho, \iota(\rho, \mu, a)}, {\cal E}_{\rho, \iota(\rho, \mu, b) }\>$ and $\<\overset{\centerdot}\I_{{\cal E}_{\nu, \iota(\nu, \rho,i)}} {\cal E}_{\nu, \iota(\nu, \mu,a) }, {\cal E}_{\nu, \iota(\nu, \mu,b) }\>$:
\begin{eqnarray} \label{ELI4termaux}
&& \<\overset{\centerdot}\I_{{\rho, i}}{E}_{\mu, a}, {E}_{\mu, b}\>
\big(\< -2 T_\mu ({E}_{\mu, a}, {E}_{\mu, b} ), E_{\rho,i}\> -\<\I_{{\mu, a}}{E}_{\mu, b}+ \I^*_{{\mu, b}}{E}_{\mu, a}, {E}_{\rho, i}\> \nonumber \\
&&
%%%aij->iab
\< - 2\, {\tilde T}_{\rho} ({E}_{\mu,a}, {E}_{\mu,b} ), E_{\rho,i}\> - \<\I_{{\mu, a}}{E}_{\mu, b} + \I^*_{{\mu, b}}{E}_{\mu, a}, {E}_{\rho, i}\> \nonumber \\
%%% ijk->iab
&& + \sum\nolimits_{\nu \notin \{\mu, \rho \}}\big(\<\tr_\nu^\top{\mathfrak T}^* - H_\nu, {E}_{\mu,b}\>\,\<{E}_{\rho, i}, {E}_{\mu, b}\>
+\<\tr_\nu^\top{\mathfrak T} +H_\nu, {E}_{\mu, b}\>\,\<{E}_{\rho, i}, {E}_{\mu, b}\> \big)\big).
\end{eqnarray}
Using $\<{E}_{\rho, i}, {E}_{\mu, b}\> =0$, as $\mD_\mu \perp \mD_\rho$ and
%\[
 $\< {\tilde T}_{\rho} ({E}_{\mu,a}, {E}_{\mu,b} ), E_{\rho,i}\> = \< T_{\mu} ({E}_{\mu,a}, {E}_{\mu,b} ), E_{\rho,i}\>$,
%\]
we %get
 %the Euler-Lagrange equation
 obtain that \eqref{ELI4termaux} vanishes for all $\<\overset{\centerdot}\I_{{\rho, i}}{E}_{\mu, a}, {E}_{\mu, b}\>$ if and only if the Euler-Lagrange equation \eqref{ELI4} holds.
%\begin{equation}\label{ELI4t}
% \< -2\,T_\mu ({E}_{\mu, a}, {E}_{\mu, b} ), E_{\rho,i}\> - \<\I_{E_{\mu, a}}{E}_{\mu, b}{+} \I^*_{E_{\mu, b}}{E}_{\mu, a}, {E}_{\rho, i}\> = 0
%\end{equation}
%for all $\mu \in \{1, \ldots k \}$, % all $\rho \ne \mu$,
%all $a,b \in \{1, \ldots, n_\mu \}$ and all $i \in \{ 1, \ldots, n_\rho \}$.

Finally, let $\mu \ne \rho \ne \xi \ne \mu$,
%Let $\mu \ne \rho$, %and let ${\cal E}_{\mu, i} = E_{\rho, i}$ above, i.e., we consider ${\cal E}_{\mu, i}$ as a vector from a particular distribution. Then we obtain the following term of \eqref{dtSD1DkI}:
for fixed $a\in\{ 1, \ldots, n_\mu \}$, $j \in \{ 1, \ldots, n_\rho \}$ and $k \in \{ 1, \ldots, n_\xi \}$ we get the following term in \eqref{dtSD1DkI}, coming from
%terms
%
$\<\overset{\centerdot}\I_{{\mu, a}} {\cal E}_{\mu, \iota(\mu, \rho, j) }, {\cal E}_{\mu, \iota(\mu, \xi,k ) }\>$,
$\<\overset{\centerdot}\I_{{\cal E}_{\xi, \iota(\xi, \mu, i)}}{\cal E}_{\xi, \iota(\xi, \rho,j ) }, E_{\xi, k}\>$,
$\<\overset{\centerdot}\I_{{\cal E}_{\rho, \iota(\rho, \mu, i)}} E_{\rho, j}, {\cal E}_{\rho, \iota(\rho, \xi,k ) }\>$
and
$\<\overset{\centerdot}\I_{{\cal E}_{\nu, \iota(\nu, \mu, i)}} {\cal E}_{\nu, \iota(\nu, \rho,j) }, {\cal E}_{\nu, \iota(\nu, \xi, k) }\>$:
%We have in \eqref{dtSD1DkI} the following term:
\begin{eqnarray*}
%ijk -> mu rho lambda
&& \<\overset{\centerdot}\I_{{\mu, a}}{E}_{\rho, j}, {E}_{\xi, k}\> \big(
%%%aij->ijk
\< - 2 {\tilde T}_\mu ({E}_{\rho, j}, {E}_{\xi, k} ), E_{\mu, i }\>
 - \<\I_{{\rho, j}}{E}_{\xi, k} + \I^*_{\xi, k}{E}_{\rho, j}, {E}_{\mu, a}\> \\
%%%ija->ijk; mu->lambda
&&
+\,\<\tr_\xi^\top{\mathfrak T}^*+\tH_\xi, {E}_{\xi,k}\> \,\<{E}_{\mu, a}, {E}_{\rho, j}\>
-\< ({\tilde h}_\xi + {\tilde T}_\xi )( {E}_{\rho, j}, {E}_{\mu, a} ), E_{\xi,k}\>
-\<\I_{{\xi, k}}{E}_{\mu, a}, {E}_{\rho, j}\>  \\
%%%
&&
%%%iaj -> ijk
+\,\<\tr_\rho^\top{\mathfrak T}-\tH_\rho, {E}_{\rho, j}\>\,\<{E}_{\mu, a}, {E}_{\xi, k}\>
+ \< ({\tilde h}_\rho + {\tilde T}_\rho )( {E}_{\xi, k}, {E}_{\mu, a} ), E_{\rho, j} \>
 - \<\I_{{\rho, j}}{E}_{\xi, k}, {E}_{\mu, a}\> \\
&&
%ijk -> mu rho lambda
+ \sum\nolimits_{\nu \notin \{\mu, \rho, \xi \}}\big(\<\tr_\nu^\top{\mathfrak T}^* - H_\nu, {E}_{\xi,k}\>\,\<{E}_{\mu, a}, {E}_{\rho, j}\>
+\<\tr_\nu^\top{\mathfrak T} +H_\nu, {E}_{\rho, j}\>\,\<{E}_{\mu, a}, {E}_{\xi, k}\> \big) \big) .
\end{eqnarray*}
As $\mD_\mu, \mD_\rho, \mD_\xi$ are pairwise orthogonal, it reduces to the following:
\begin{eqnarray} \label{ELI5termaux}
%ijk -> mu rho lambda
&&\hskip-5mm \<\overset{\centerdot}\I_{{\mu, a}}{E}_{\rho, j}, {E}_{\xi, k}\> \big(
%%%aij->ijk
\< - 2 {\tilde T}_\mu ({E}_{\rho, j}, {E}_{\xi, k} ), E_{\mu, i }\>
- \<\I_{{\rho, j}}{E}_{\xi, k} + \I^*_{{\xi, k}}{E}_{\rho, j}, {E}_{\mu, a}\>
-\<\I_{{\xi, k}}{E}_{\mu, a}, {E}_{\rho, j}\> \nonumber \\
%%%ija->ijk; mu->lambda
&&\hskip-5mm -\< ({\tilde h}_\xi + {\tilde T}_\xi)( {E}_{\rho, j}, {E}_{\mu, a} ), E_{\xi,k}\>
% \nonumber \\
%%%
%&&
%%%iaj -> ijk
 +\,\< ({\tilde h}_\rho + {\tilde T}_\rho )( {E}_{\xi, k}, {E}_{\mu, a} ), E_{\rho, j} \> - \<\I_{{\rho, j}}{E}_{\xi, k}, {E}_{\mu, a}\> \big) .
\end{eqnarray}
Using $ \<\I^*_{{\xi, k}}{E}_{\rho, j}, {E}_{\mu, a}\> = \<\I_{{\xi, k}}{E}_{\mu, a}, {E}_{\rho, j}\> $, we obtain that \eqref{ELI5termaux} vanishes
for all $\<\overset{\centerdot}\I_{{\mu, a}}{E}_{\rho, j}, {E}_{\xi, k}\>$ if and only if the Euler-Lagrange equation \eqref{ELI5} holds.

%is the following:
%\begin{eqnarray} \label{ELI5t}
%&&\< - 2 {\tilde T}_\mu ({E}_{\rho, j}, {E}_{\xi, k} ), E_{\mu, i }\>
%\nonumber \\
%%%%ija->ijk; mu->lambda
%&&
%-\< ({\tilde h}_\xi + {\tilde T}_\xi )( {E}_{\rho, j}, {E}_{\mu, a} ), E_{\xi,k}\>
%- 2\<\I_{{\xi, k}}{E}_{\mu, a}, {E}_{\rho, j}\>\nonumber \\
%%%%
%&&
%%%%iaj -> ijk
%+ \< ({\tilde h}_\rho + {\tilde T}_\rho ) ({E}_{\xi, k}, {E}_{\mu, a} ), E_{\rho, j} \>
%- 2\<\I_{{\rho, j}}{E}_{\xi, k}, {E}_{\mu, a}\> =0
%\end{eqnarray}
%for all $\mu \in \{1, \ldots k \}$, all $\rho \ne \mu$, all $\xi \notin \{ \mu, \rho \}$, all $i \in \{1, \ldots, n_\mu \}$, all $j \in \{ 1, \ldots,
%\dim \mD_\rho \}$ and all $k \in \{ 1, \ldots, \dim \mD_\xi \}$.
Equations (\ref{ELI1}-e) are indeed all components of \eqref{E-delta-I-J}, because we considered all coefficients of $\< \overset{\centerdot}\I_X Y,Z\>$, where $X,Y,Z$ are from an orthonormal frame: either all $X,Y,Z$ are from the same distribution and yield \eqref{ELI1}, exactly two of them are from the same distribution -- and we get (\ref{ELI2}-d), or each of them is from a different distribution and we obtain \eqref{ELI5}.
%Hence, $\I$ is critical if and only if all equations (\ref{ELI1}-e) hold.
\end{proof}

On a manifold with two orthogonal distributions, $TM = \mD_1 \oplus \mD_2$, i.e., $\mD_2=\mD_1^\bot$, we have only equations (\ref{ELI1}-d)
for $\mu=1,2$, which were obtained in \cite{rz-connections}. %These equations hold also for $TM = \mD_1 \oplus \ldots \oplus \mD_k$
%with $k>2$, so their conclusions obtained in \cite{rz-connections} are still true. %can be proved in a similar way.
%Therefore,
Similarly to \cite[Theorem~1]{rz-connections} for $k=2$,
% as in \cite{rz-connections},
 we conclude the following

\begin{Theorem}\label{corI}
A contorsion tensor $\I$ is critical for the action \eqref{Eq-Smix-g} with fixed adapted metric on $(M;\mD_1,\ldots,\mD_k)$
for all variations of \,$\I$ if and only if all $\mD_\mu$ are {totally umbilical} and $\I$ satisfies the following linear algebraic system
for all $X,Y\in\mD_\mu$, $U\in\mD_\mu^\perp$ and all $\mu =1, \ldots, k$: %see \eqref{E-delta-I-J}:
\begin{subequations}
\begin{eqnarray}
%\label{ELconnectionNew1}
%\nonumber
%&& P_\mu (\I_U\, V +\I^{*}_V\, U) = -2\, {\widetilde{T}_\mu}(U, V), \\
%\nonumber
\label{ELconnectionNew2}
&& P_\mu \tr_\mu^\bot\I^* = \tH_\mu = - P_\mu \tr_\mu^\bot\I \quad {\rm if~} n_\mu > 1\,, \\
%\nonumber
\label{ELconnectionNewI4}
&& \< (\I -\I^{*})_U X, Y\> = 2\, \< {T}_\mu (X,Y), U\> \\
%\nonumber
\label{ELconnectionNew5}
&& \< (\I + \I^{*})_U X, Y\> = \<\tr_\mu^\bot(\I +\I^*),\,U\> \< X,Y\>\,, \\
%\nonumber
\label{ELconnectionNewI7}
&& P_\mu^\perp \tr_\mu^\bot(\I -\I^*) = (2-2/n_\mu)\,H_\mu\,, \\
%\nonumber
%dual
\label{ELconnectionNewI8}
&& P_\mu^\perp (\I_X\, Y +\I^{*}_Y\, X) = -2\,T_\mu (X, Y)\,,
%\nonumber
%\label{ELconnectionNew9}
%&& P_\mu^\perp \tr_\mu^\top\I^* = H_\mu = - P_\mu^\perp (\tr_\mu^\top\I) \quad {\rm for~} n^\bot_\mu>1, \\
%\nonumber
%\label{ELconnectionNew11}
%&& \< \I_X U -\I_X^{*} U, V\> = 2\, \< {{\tilde T}_\mu}(U,V), X\> \\
%\nonumber
%\label{ELconnection1ab}
%&& \< {\I}_X U + \I_X^{*} U, V\> = \<\tr_\mu^\top(\I +\I^*),\,X\> \< U, V\>, \\
%\label{ELconnectionNew14}
%&& P_\mu (\tr_\mu^\top(\I -\I^*)) = (2-2/n^\bot_\mu)\,\tH_\mu,
\end{eqnarray}
\end{subequations}
moreover, for all $X \in \mD_\mu$, $Y \in \mD_\rho, Z \in \mD_\xi$, where $\mu \ne \rho \ne \xi \ne \mu$, we get
\begin{equation} \label{ELI5XYZ}
 \<\I_Y Z, X\> + \<\I_{Z}X, Y\> = 0.
\end{equation}
\end{Theorem}

\begin{proof}
%The proof is similar to the proof \cite[Theorem~1]{rz-connections} for $k=2$.
%In particular,
Taking the difference of symmetric parts of (\ref{ELI2},c) we get the total umbilicity of each $\mD_\mu$;
 \eqref{ELconnectionNew2} follows from \eqref{ELI1};
 taking antisymmetric part of \eqref{ELI2} we get \eqref{ELconnectionNewI4};
 the sum of \eqref{ELI2} and \eqref{ELI3} yields \eqref{ELconnectionNew5};
 taking the difference of \eqref{ELI2} and \eqref{ELI3} with interchanged $E_{\mu, a}, E_{\mu, b}$ we get \eqref{ELconnectionNewI7},
and \eqref{ELconnectionNewI8} follows from \eqref{ELI4}.
Finally, from \eqref{ELI5} we~get
\[
%\label{ELI5XYZold}
 -2\,\<{\tilde T}_\mu (Y, Z ), X\>
%%ija->ijk; mu->lambda
-\< ({\tilde h}_\xi + {\tilde T}_\xi )( Y, X ), Z\>
%%%
%%%iaj -> ijk
 +\< ({\tilde h}_\rho + {\tilde T}_\rho )( Z, X ), Y\> = 2\<\I_Y Z, X\> + 2\<\I_{Z}X, Y\>\,,
\]
which is simplified to \eqref{ELI5XYZ}.
\end{proof}

\begin{Corollary}\label{corstatcritforall}
A contorsion tensor $\I$ of a statistical connection is critical for
%the action
\eqref{Eq-Smix-g} with fixed adapted metric $g$ on $(M;\mD_1,\ldots,\mD_k)$
%Let $\I$, corresponding to a statistical connection, be critical for the action \eqref{Eq-Smix-g} with fixed adapted metric
%\eqref{actionI}
for all variations of \,$\I$ if and only if for $\mu =1, \ldots, k$:
\begin{enumerate}
\item \label{item1} all $\mD_\mu$ with $n_\mu>1$ are integrable and totally geodesic,
\item \label{item2} $\< \I_X Y, Z\> =0$ for all $X,Y \in \mD_\mu$ and $Z \in \mD^\bot_\mu$,
\item \label{item3} if $n_\mu>1$ then
%$H_\mu =0$ %(by \eqref{ELconnectionNewI7} ), and
${\tilde H}_\mu =0$,
\item \label{item5} $\tr^\perp_\mu \I =0 = \tr^\top_\mu \I$.
%\end{enumerate}
%If additionally, all $\mD_\mu$ are pairwise mixed integrable and pairwise mixed totally geodesic, then
%\begin{enumerate}
\item%[5.]
\label{item4}
$\< \I_X Y, Z\> =0$ for $X,Y,Z$ -- each vector from a different distribution. %(by \eqref{ELI5XYZ} )
\end{enumerate}
\end{Corollary}

\begin{proof}
We use properties of a statistical connection: $\I = \I^* = \I^\wedge$ in Theorem \ref{corI} to prove necessity of the above conditions, their sufficiency is easily verified.
Claim \ref{item1} follows from \eqref{ELconnectionNewI4}, \eqref{ELconnectionNewI7} and Theorem \ref{corI};
claim \ref{item2} follows from claim \ref{item1} and \eqref{ELconnectionNewI8};
%claim \ref{item3}$_1$ follows from \eqref{ELconnectionNewI7};
%\ref{item4}
for the first equality of claim \ref{item5} we use claim \ref{item2} to get $P_\mu \tr_\mu^\perp \I=0$ and claim \ref{item2} with \eqref{ELconnectionNew5} to get $P_\mu^\perp \tr_\mu^\perp \I = 0$, the second equality of claim \ref{item5} follows from the first one, as for $k \ge 2$ we get
\begin{eqnarray*}
 \sum\nolimits_{\nu\ne\mu} \tr_\nu^\perp \I &=& (k-1) \tr^\top_\mu \I + (k-2)\sum\nolimits_{\nu\ne\mu} \tr^\top_\nu \I \\
 &=& (k-1) \tr^\top_\mu \I + (k-2) \tr^\perp_\mu \I, \quad \mu = 1, \ldots, k\,.
\end{eqnarray*}
Claim \ref{item3} follows from \eqref{ELconnectionNew2} and claim \ref{item5}.
Finally, claim 5 follows from \eqref{ELI5XYZ}.
\end{proof}

\begin{Remark} \rm
For the action \eqref{Eq-Smix-g} (with fixed adapted metric) restricted to contorsion tensors of metric-compatible connections,
all equations of Theorem~\ref{corI} remain true, with $\I = -\I^*$.
%For \eqref{ELconnectionNew1} this can be proven in the same way as in \cite{rz-connections}.
For \eqref{ELI5XYZ} this follows from the fact that \eqref{ELI5} is antisymmetric in $E_{\rho, j}, E_{\xi, k}$ when $\I = - \I^*$,
and for other equations of Theorem \ref{corI} it follows in the same way as in \cite[Theorem 2]{rz-connections}.
\end{Remark}

%%%NEW PART

An important class of metric connections are those with totally skew-symmetric torsion \cite{AF}, for which we have
\begin{equation}\label{totallyskewsymmetrictorsion}
 \I = - \I^\wedge\,.
\end{equation}

\begin{Corollary}
Let $\I$ be the contorsion tensor of a connection with totally skew-symmetric torsion, that is critical for the action \eqref{Eq-Smix-g} with fixed $g$. Then
all $\mD_\mu$ such that $\dim \mD_\mu >1$ are {totally geodesic} and integrable, and $\I=0$.
%and for all  $X \in \mD_\mu$, $Y \in \mD_\rho, Z \in \mD_\xi$ we have
%the mixed integrability type condition
%%\[
%%\< \I_X Y , Z \> =0
%%\]
%%and
%\[
%\< {\tilde T}_\rho ( Z, X ), Y\> = \<{\tilde T}_\xi (X, Y ), Z\> = \<{\tilde T}_\mu (Y, Z ), X\>,\quad \rho\ne\xi\ne\mu.
%\]
\end{Corollary}

\begin{proof}
By Theorem \ref{corI}, all distributions $\mD_\mu$ are {totally umbilical}. Let $\dim \mD_\mu>1$, then using \eqref{totallyskewsymmetrictorsion} and $\I = -\I^*$ in \eqref{ELconnectionNewI7}, we obtain $H_\mu=0$, i.e., $\mD_\mu$ is totally geodesic.
From (\ref{ELconnectionNewI4},e) together with \eqref{totallyskewsymmetrictorsion} and $\I = -\I^*$ we obtain $T_\mu =0$, i.e., $\mD_\mu$ is integrable.
Using \eqref{totallyskewsymmetrictorsion}, we get $\I_X X =0$ for all $X \in TM$, and from (\ref{ELconnectionNewI4},e) it follows that $\< \I_U X,Y \> =0 = \< \I_X Y , U \>$  for all $X,Y\in\mD_\mu$, $U\in\mD_\mu^\perp$ and all $\mu =1, \ldots, k$.
Let $X \in \mD_\mu$, $Y \in \mD_\rho, Z \in \mD_\xi$, where $\mu \neq \rho \neq \xi \neq \mu$.
By \eqref{totallyskewsymmetrictorsion} and $\I = -\I^*$, we get in~\eqref{ELI5XYZ}:
\[
0 = \< \I_Y Z ,X\> + \< \I_Z X ,Y\> = - \< \I_Z Y ,X\> + \< \I_Z X ,Y\> = 2 \< \I_Z X ,Y\> .
\]
Hence, all components of $\I$ vanish.
%OLD
%Let $X \in \mD_\mu$, $Y \in \mD_\rho, Z \in \mD_\xi$, where $\mu \neq \rho \neq \xi \neq \mu$. From \eqref{totallyskewsymmetrictorsion} we obtain in \eqref{ELI5XYZ}:
%\begin{equation} \label{ELI5XYZskew}
%4 \< \I_Y Z ,X \> =  -2\,\<{\tilde T}_\mu (Y, Z ), X\>
%+\< {\tilde T}_\xi ( X, Y ), Z\>
% +\,\< {\tilde T}_\rho ( Z, X ), Y\> .
%\end{equation}
%Interchanging $X,Y$ in \eqref{ELI5XYZskew} and using \eqref{totallyskewsymmetrictorsion} we get
%%\[
%%4 \< \I_X Z ,Y \> =  -2\,\<{\tilde T}_\rho (X, Z ), Y\>
%%+\< {\tilde T}_\xi ( Y, X ), Z\>
%% +\,\< {\tilde T}_\mu ( Z, Y ), X\>
%%\]
%%and hence
%\[
% 4 \< \I_Y Z ,X \> =  -2\,\<{\tilde T}_\rho (Z, X ), Y\>
% +\< {\tilde T}_\xi ( X, Y ), Z\> +\,\< {\tilde T}_\mu ( Y, Z ), X\> ;
%\]
%on the other hand, interchanging $X,Z$ in \eqref{ELI5XYZskew} together with \eqref{totallyskewsymmetrictorsion} yields
%%\[
%%-4 \< \I_Y Z ,X \> =  -2\,\<{\tilde T}_\xi (Y, X ), Z\>
%%-\< {\tilde T}_\mu ( Y, Z ), X\>
%% -\,\< {\tilde T}_\rho ( Z, X ), Y\>
%%\]
%%i.e.,
%\[
%4 \< \I_Y Z ,X \> =  -2\,\<{\tilde T}_\xi (X, Y ), Z\>
%+\< {\tilde T}_\mu ( Y, Z ), X\>
% +\,\< {\tilde T}_\rho ( Z, X ), Y\> .
%\]
%It follows from three above formulas that
%\[
%\< {\tilde T}_\rho ( Z, X ), Y\> = \<{\tilde T}_\xi (X, Y ), Z\> = \<{\tilde T}_\mu (Y, Z ), X\>
%\]
%and $\< \I_X Z ,Y \>=0$, which completes the proof.
\end{proof}

%%%

For action \eqref{Eq-Smix-g} with fixed adapted $g$
%\eqref{actionI}
restricted to contorsion tensors of statistical connections, we obtain the following generalization of \cite[Corollary 7]{rz-connections}.
%proof similar to that in ResultsInMath-correction.pdf

\begin{Theorem}\label{statisticalcritSmixI}
A contorsion tensor $\I$ of a statistical connection on $(M;\mD_1,\ldots,\mD_k)$ with fixed adapted metric $g$
is critical for the action \eqref{Eq-Smix-g}
%on $(M, g)$, is critical for the action \eqref{Eq-Smix-g} with fixed adapted metric
with respect to variations of \,$\I$ corresponding to statistical connections if and only if the following system is valid:
\begin{subequations}
\begin{equation}
\label{ELSmixIstat1}
 P_\mu \tr_\mu^\perp \I = 0,\quad \mu = 1, \ldots, k\,,
\end{equation}
%see \eqref{E-delta-I-J},
%for all $U,V \in \mD_\mu$
and for $\mu \ne \rho \ne \xi \ne \mu$ and all $X,U \in \mD_\mu$, $Y \in \mD_\rho$ and $Z \in \mD_\xi$ we get
\begin{eqnarray}
\label{ELSmixIstat2}
 &&
% 2\,\< \I_X U, Y\> = \< \tr_\mu^\perp \I, Y\>\,\<X,\,U\> \Leftrightarrow
 P_\mu^\bot (2\,\I_X U + \<X,\,U\>\tr_\mu^\perp \I ) = 0\,,\\
 \label{newELIstat}
 && \<\I_X Y, Z\> =0\,.
\end{eqnarray}
\end{subequations}
\end{Theorem}

\begin{proof}
For variations $\overset{\centerdot} \I$
%among contorsion tensors
corresponding to statistical connections, we have the following symmetries:
%\[
 $\< \overset{\centerdot}\I_X Y, Z\> = \< \overset{\centerdot}\I_Y X, Z\> = \<\overset{\centerdot}\I_X Z, Y\>,\ X,Y,Z\in TM$.
%\]
It follows that instead of \eqref{ELI1}, the first Euler-Lagrange equation is the sum of \eqref{ELI1} over all permutations of
$({E}_{\mu, a}, {E}_{\mu, b}, {E}_{\mu, c} )$ -- from that and $\I^* = \I$, we get
%\[
 $\<\tr_\mu^\bot{\mathfrak T}, {E}_{\mu, c}\>\delta_{a,b}
 %\<{E}_{\mu, a}, {E}_{\mu, b}\>
 +\<\tr_\mu^\bot{\mathfrak T}, {E}_{\mu, b}\>\delta_{a,c}
% \<{E}_{\mu, a}, {E}_{\mu, c}\>
 =0$,
%\]
and considering either ${E}_{\mu, a} \ne {E}_{\mu, b} \ne {E}_{\mu, c}$ or two of above are equal, we obtain \eqref{ELSmixIstat1}.

Similarly, instead of three separate Euler-Lagrange equations (\ref{ELI2}-d)
%\eqref{ELI3} and \eqref{ELI4},
we now have one Euler-Lagrange equation that is their sum, symmetrized in ${E}_{\mu, a}, {E}_{\mu, b}$, i.e.,
\begin{eqnarray*}
%\label{ELIsum24t}
&& \<\tr_\mu^\bot{\mathfrak T}^* + H_\mu, E_{\rho, i}\>\delta_{a,b}
%\<{E}_{\mu, a}, {E}_{\mu, b}\>
-\< (h_\mu - T_\mu) ({E}_{\mu, a}, {E}_{\mu, b} ), E_{\rho, i}\> - \<\I_{{\rho, i}}{E}_{\mu, a}, {E}_{\mu, b}\>\nonumber \\
&& +\,\<\tr_\mu^\bot{\mathfrak T} - H_\mu, E_{\rho, i}\>\delta_{a,b}
%\<{E}_{\mu, a}, {E}_{\mu, b}\>
+ \< (h_\mu + T_\mu)( {E}_{\mu, b}, {E}_{\mu, a}), {E}_{\rho,i}\>
- \<\I_{{\rho, i}}{E}_{\mu, b}, {E}_{\mu, a}\>\nonumber\\
&& -\,\< 2\, T_\mu ({E}_{\mu, a}, {E}_{\mu, b} ), E_{\rho,i}\> -\<\I_{{\mu, a}}{E}_{\mu, b}+ \I^*_{{\mu, b}}{E}_{\mu, a}, {E}_{\rho, i}\>\nonumber \\
%%%%symmetrized a,b
&& +\, \<\tr_\mu^\bot{\mathfrak T}^* + H_\mu, E_{\rho, i}\>\delta_{a,b}
%\<{E}_{\mu, b}, {E}_{\mu, a}\>
-\< (h_\mu - T_\mu) ({E}_{\mu, b}, {E}_{\mu, a} ), E_{\rho, i}\> - \<\I_{{\rho, i}}{E}_{\mu, b}, {E}_{\mu, a}\>\nonumber \\
&& +\, \<\tr_\mu^\bot{\mathfrak T} - H_\mu, E_{\rho, i}\>\delta_{a,b}
%\<{E}_{\mu, b}, {E}_{\mu, a}\>
+ \< (h_\mu + T_\mu)( {E}_{\mu, a}, {E}_{\mu, b}), {E}_{\rho,i}\>
- \<\I_{{\rho, i}}{E}_{\mu, a}, {E}_{\mu, b}\>\nonumber\\
&& -\, \< 2\, T_\mu ({E}_{\mu, b}, {E}_{\mu, a} ), E_{\rho,i}\> - \<\I_{{\mu, b}}{E}_{\mu, a} + \I^*_{{\mu, a}}{E}_{\mu, b}, {E}_{\rho, i}\> =0\,.
\end{eqnarray*}
Using $\I^* = \I$ and $\I_X Y = \I_Y X$ for $X,Y \in TM$ and dividing by 4, we get
%\begin{equation*}
%\label{ELIsum24tsimplified}
 $\<\tr_\mu^\bot{\mathfrak T}, E_{\rho, i}\>\delta_{a,b}
%\<{E}_{\mu, a}, {E}_{\mu, b}\>
 =- 2\,\<\I_{{\rho, i}}{E}_{\mu, a}, {E}_{\mu, b}\>$\,,
%\end{equation*}
and hence \eqref{ELSmixIstat2}.
%\begin{eqnarray} \label{ELI2t}
%% \<\overset{\centerdot}\I_{{\mu, a}}{E}_{\mu, b}, E_{\rho, i}\>\big(
%\<\tr_\mu^\bot{\mathfrak T}^* + H_\mu, E_{\rho, i}\>\,\<{E}_{\mu, a}, {E}_{\mu, b}\>
%-\< (h_\mu - T_\mu) ({E}_{\mu, a}, {E}_{\mu, b} ), E_{\rho, i}\> - \<\I_{{\rho, i}}{E}_{\mu, a}, {E}_{\mu, b}\> = 0,
%%\big) = 0, %\end{eqnarray}
%
%\begin{eqnarray} \label{ELI3t}
%\<\tr_\mu^\bot{\mathfrak T} - H_\mu, E_{\rho, i}\>\,\<{E}_{\mu, a}, {E}_{\mu, b}\>
%+ \< (h_\mu + T_\mu)( {E}_{\mu, b}, {E}_{\mu, a}), {E}_{\rho,i}\>
%- \<\I_{{\rho, i}}{E}_{\mu, b}, {E}_{\mu, a}\> = 0,
%\end{eqnarray}
%
%\begin{eqnarray} \label{ELI4t}
%&&
%\< -2 T_\mu ({E}_{\mu, a}, {E}_{\mu, b} ), E_{\rho,i}\>
%{-} \<\I_{{\mu, a}}{E}_{\mu, b}{+} \I^*_{E_{\mu, b}}{E}_{\mu, a}, {E}_{\rho, i}\> =0,
%\end{eqnarray}
Equation \eqref{newELIstat} follows from the fact that due to symmetries of $\overset{\centerdot}\I$ for variations corresponding to statistical connections, instead of \eqref{ELI5} we get
\begin{eqnarray} \label{ELI5tsum}
&& \sum \big(-\< 2\, {\tilde T}_\mu ({E}_{\rho, j}, {E}_{\xi, k} ), E_{\mu, a }\>
\nonumber
%%%ija->ijk; mu->lambda
-\,\< ({\tilde h}_\xi + {\tilde T}_\xi )( {E}_{\rho, j}, {E}_{\mu, a} ), E_{\xi,k}\> \nonumber \\
%%%
&&
 -\,2\,\<\I_{{\xi, k}}{E}_{\mu, a}, {E}_{\rho, j}\>
%%%iaj -> ijk
 +\< ({\tilde h}_\rho + {\tilde T}_\rho )( {E}_{\xi, k}, {E}_{\mu, a} ), E_{\rho, j}\> - 2\,\<\I_{{\rho, j}}{E}_{\xi, k}, {E}_{\mu, a}\> \big) =0\,,
\end{eqnarray}
where the sum is over all permutations of $( E_{\mu, a}, E_{\rho,j}, E_{\xi, k} )$. Since all ${\tilde h}$- and ${\tilde T}$- terms can be canceled out, \eqref{ELI5tsum} reduces to \eqref{newELIstat}.
\end{proof}

Existence of solutions of the Euler-Lagrange equations obtained in this section will be discussed in more detail in subsequent parts of the article,
along with variations of the metric.

\section{Extension of the class of Einstein metrics}
\label{sectionEinstein}

In this section we assume that all distributions in \eqref{Eq-mD-k-TM} are one-dimensional, then \eqref{Eq-Smix-g} is (up to constant factor $2$) the geometrical part of the Einstein-Hilbert action:
\begin{equation}\label{Eq-EH}
 \bar J : (g, \I) \mapsto \int_{M} \overline{\rm S}\ {\rm d}\vol_g\,,
\end{equation}
where $\overline{\rm S}$ is the scalar curvature of ${\bar \nabla} = \nabla + \I$ on $(M,g)$. % on some manifolds.
Thus, the action \eqref{Eq-Smix-g} restricted to adapted metrics allows us to extend the class of Einstein metrics.
In particular, we obtain critical pairs $(g, \I)$ for \eqref{Eq-Smix-g} with non-Einstein metrics of constant scalar curvature for $k=3$ and for $k>3$ using Hadamard matrices.
There is a rich literature on geometric constructions of metrics of constant scalar curvature, which we will not discuss.

\begin{Proposition}\label{propEinstein}
Let a pair $(g, \I)$ be critical for the
%Einstein-Hilbert
action \eqref{Eq-EH}.
Then on any open set $\Omega \subset M$, on which we have a decomposition \eqref{Eq-mD-k-TM}
of $TM$ into the sum of one-dimensional distributions $\mD_\mu$, the Euler-Lagrange equations \eqref{E-delta-g-J} and \eqref{E-delta-I-J},
given in details in Theorem~\ref{T-main01} and Proposition~\ref{P-03}, are satisfied.
\end{Proposition}

\begin{proof}
For one-dimensional distributions $\mD_1 , \ldots , \mD_k$ we have $ 2\, \overline{\rm S}_{\,\mD_1,\ldots,\mD_k} = \overline{\rm S}\,$.
Hence, if $g$ is critical for the action \eqref{Eq-EH},
%with respect to all variations of metric,
then it is also critical for the action \eqref{Eq-Smix-g} with respect to compactly supported adapted variations.
%with compact supports inside $\Omega$.
Therefore, it satisfies
%the Euler-Lagrange equations
\eqref{E-delta-g-J} and \eqref{E-delta-I-J}.
Notice that while the equations \eqref{E-delta-g-J} and \eqref{E-delta-I-J} are pointwise, \eqref{E-delta-g-J} contains covariant derivatives of quantities describing geometry of the distribution (e.g., $\Div{h}_\mu$), and to make sense requires the distributions to be
defined on some open set.
\end{proof}

%The following consequence of \eqref{ELI5XYZ} for one-dimensional distributions and
Using the formulation of the Euler-Lagrange equation \eqref{E-delta-I-J} given in Theorem \ref{corI}, from Proposition~\ref{propEinstein}
we obtain the following.
 %, %can be viewed as a partial characterization
%% in terms of metric,
%%of contorsion tensors critical for the~action
%see Section~\ref{sectionEinstein},

\begin{Corollary}
Let a pair $(g, \I)$ be critical for the
%Einstein-Hilbert
action \eqref{Eq-EH} on a smooth manifold $M$. %If $(M,g)$ is either parallelizable, or non-compact,
Then for any orthonormal vector fields $X,Y,Z$ we obtain \eqref{ELI5XYZ}.
%\begin{eqnarray*}% \label{ELI5XYZ}
% \<\I_Y Z, X\> + \<\I_{Z}X, Y\> = 0\,.
%\end{eqnarray*}
\end{Corollary}

If all distributions are one-dimensional, %the above
\eqref{ELI5XYZ} is in fact the only restriction from Theorem \ref{corI} for critical metric connections.

%\end{proof}
%\begin{Remark} \label{remarkEinstein} \rm
%Note also that in order to allow decomposition \eqref{Eq-mD-k-TM} into one-dimensional distributions on the set $\Omega$ containing %the supports of metric and contorsion variations, the manifold must be either parallelizable, or non-compact. %(then decomposition can be obtained on a proper subset of $M$).
%\end{Remark}

%\begin{Remark}\label{R-03}\rm
On the other hand, adapted variations of metric on an almost product manifold
%that we consider
can be also applied to the Einstein-Hilbert action. We discuss it as a functional of metric only and in Remark \ref{remarkEinsteinmetricconnection}, at the end of this section, we show how our results generalize also to arbitrary metric connection.
According to \cite[Proposition 2.3(2)]{RN-21},
$g$ is critical with respect to adapted variations of metric for the action
\begin{equation}\label{EHonlyg}
 J: g \mapsto \int_M {\rm S}\ {\rm d}\vol_g\,,
\end{equation}
where ${\rm S}$ is the scalar curvature of $(M,g)$,
% and $\Omega\subset M$,
if and only if
\begin{equation}\label{E-3dim}
 {\rm Ric}|_{\mD_\mu \times \mD_\mu} = \lambda\,g |_{\mD_\mu \times \mD_\mu},\quad
 \mu =1 , \ldots , k\,,
\end{equation}
where ${\rm Ric}$ is the Ricci tensor and a constant $k\lambda$ is the scalar curvature of $(M,g)$.
%%qqq
%In particular, if all distributions are one-dimensional,
%then
%${\rm S}_{\,\mD_1,\ldots,\mD_k} = 2\,{\rm S}$
%and
%$\Ric_{\,\mD} = \bigoplus\nolimits_{\mu=1}^k {\rm Ric}_{\,\mD\,|\,\mD_\mu \times \mD_\mu}$.
Such critical, non-Einstein metrics can be found, for example, as follows.

\begin{Example} \label{ex3dim}
%\label{Ex-warped}
\rm
The product of a surface of constant curvature $K\ne0$ and a real line or a~circle is a homogeneous space $(M^3,g)$ of scalar curvature $2K$.
Let $\partial_x,\partial_y,\partial_t$ be an adapted orthonormal frame on $(M,g)$. Then ${\rm Ric}_{xx}={\rm Ric}_{yy}=K$ and ${\rm Ric}_{tt}=0$.
For $\partial_1=\cos\alpha\,\partial_t+\sin\alpha\,\partial_x$ we find ${\rm Ric}_{11}(\alpha)=K\sin^2\alpha$.
Thus, ${\rm Ric}_{11}(\alpha)=2K/3$ for $\alpha=\arccos(1/\sqrt3)$.
Set $\partial_2=-a\,\partial_t+b\,\partial_x+c\,\partial_y$ and $\partial_3=-a\,\partial_t+b\,\partial_x-c\,\partial_y$ for positive numbers $a,b,c$.
From $1=\<\partial_2,\partial_2\>=\<\partial_3,\partial_3\>$, we find $a^2+b^2+c^2=1$.
Equating ${\rm Ric}_{22}={\rm Ric}_{33}=K(b^2+c^2)$ to $2K/3$ and using $b^2+c^2=1-a^2$, we obtain $a=1/\sqrt3$.
Then, from $0=\<\partial_1,\partial_2\>=\<\partial_1,\partial_3\>$ we find $b=a\cos\alpha/\sin\alpha=1/\sqrt 6$.
Thus, $c=1/\sqrt2$.
% we get ${\rm Ric}_{22}={\rm Ric}_{33}=\frac23 K$.
Condition \eqref{E-3dim} is then valid on $(M^3,g)$ for three distributions spanned by $\partial_1,\partial_2,\partial_3$.
%
%The~same values for $a,b,c$ and $\partial_1,\partial_2,\partial_3$ we get assuming ${\rm Ric}_{tt}=\varepsilon\ne0$.
%To show this, we take warped product representation of a hyperbolic 3-space or a 3-sphere of constant curvature.
%If we multiply by $\lambda>0$ the metric along the 2-dimensional factor,
%then the sectional curvature of the 2-dimensional factor will be different constant,
%and the mixed sectional curvature
%%of the resulting metric
%will be constant.
%Thus, the Ricci curvature of the resulting warped product satisfies \eqref{E-3dim}.
%%has the view~\eqref{Ricaab}.
\end{Example}

\begin{Example} \label{ex4dim} \rm
Let $(M,g)$ be a 4-dimensional Riemannian manifold of constant scalar curvature $c$ and let $X_\mu\ (\mu=1,2,3,4)$ be orthonormal smooth vector fields on $M$ such that
${\rm Ric}^\sharp(X_\mu) = f_\mu X_\mu$ for $f_\mu \in C^\infty(M)$. %with pointwise distinct values,
%i.e., $g(X_\mu , X_\nu ) = \delta_{\mu \nu}$. % and ${\rm Ric}(X_\mu , X_\nu) = f_\mu \delta_{\mu \nu}$ and $f_\mu(p) \neq f_\nu(p)$ for all $\mu \neq \nu$ and all $p \in M$.
We define the following vector fields:
$Y_1 = ( X_1 + X_2 + X_3 + X_4 )/2$, $Y_2 = ( -X_1 + X_2 - X_3 + X_4 )/2$, $Y_3 = ( -X_1 - X_2 + X_3 + X_4 )/2$, $Y_4 = ( X_1 - X_2 - X_3 + X_4)/2$.
Then ${\rm Ric}(Y_\mu , Y_\mu ) = \sum_{\,\nu=1}^{\,4} f_\nu = c$ for $\mu=1,2,3,4$, and $\{Y_\mu\}$ are orthonormal.
We define four one-dimensional distributions $\mD_\mu$, each spanned by $Y_\mu$. Then $g$ is a critical point of the Einstein-Hilbert action with respect to %adapted
variations of metric preserving the almost product structure $TM = \bigoplus_{\mu=1}^4 \mD_\mu$, but $g$ does not need to be Einstein.
%A~similar construction can be applied in dimension 8.
\end{Example}
%Some non-Einstein metrics with positive scalar curvature are presented in \cite{K-94}.
%\end{Remark}

%Other non-Einstein metrics satisfying \eqref{E-3dim} can be obtained on homogeneous spaces, we discuss below
%%in full detail
%only the $3$-dimensional case, where explicit formulas can be easily obtained. %, but this method can be generalized.

A Hadamard matrix $H_k$ is a $k \times k$-matrix, all entries of which have values in the set $\{-1,1\}$ and such that $\frac{1}{\sqrt{k}}\,H_k$ is an orthogonal matrix.
Such matrices are known to exist in some dimensions, e.g., $k=2^n$ for natural $n$, and $k=4m$ for natural $m$ such that $k<668$, see~\cite{Horadam}.

\begin{Lemma} \label{lemRicdiagkdim}
Let $(M,g)$ be a $k$-dimensional Riemannian manifold of constant scalar curvature, where $k=3$, or $k$ is such that a Hadamard matrix $H_k$ exists.
% $k$ is a dimension in which a Hadamard matrix exists.
Then for any $p \in M$ there exists a decomposition
%$TM = \mD_1+\mD_2+\mD_3$
of $T_pM$ into the sum of one-dimensional orthogonal subspaces $\mD_1, \ldots ,\mD_k$
such that \eqref{E-3dim} is valid,
%${\rm Ric}|_{\mD_\mu \times \mD_\mu} = \lambda \,g |_{\mD_\mu \times \mD_\mu}\ (\mu=1,\ldots,k)$,
where $k\lambda$ is the scalar curvature of $(M,g)$.
%$\lambda\in\RR$.
\end{Lemma}

\begin{proof}
%By \cite[Proposition 2.3(2)]{RN-21},
We can find the above decomposition satisfying \eqref{E-3dim}
if and only if there exists
an orthonormal frame %$\{ \partial_1,\partial_2,\partial_3\}$ of $T_pM$,
in which the matrix of ${\rm Ric}^\sharp$ has equal diagonal elements, i.e.,
\begin{eqnarray}\label{matrixeq1}
&& \sum\nolimits_{\,j,m} a_{ij} r_{jm} a_{im} = \lambda , \quad i=1,\ldots , k\,, \\
\label{matrixeq2}
&& \sum\nolimits_{\,j} a_{ij}a_{mj} = \delta_{im}, \quad  i,m=1,\ldots k\,.
\end{eqnarray}
Here $a_{ij}$ are entries of some orthogonal matrix $A$, and $r_{jm}$ are components of ${\rm Ric}^\sharp$ in some orthonormal basis.
%In particular,
We can assume that $r_{jm}=r_j \delta_{jm}$ and $r_1 , \ldots , r_k$ are not all equal, then \eqref{matrixeq1}~becomes
\begin{equation} \label{matrixeq1diag}
\sum\nolimits_{\,j} a_{ij}^2 r_j = \lambda , \quad i=1,\ldots , k\,.
\end{equation}
%NEW:
%We can assume that there exists an orthonormal basis in which ${\rm Ric}^\sharp$ is diagonal matrix with entries $r_1 , \ldots , r_k$, which are not all equal.
Suppose that $k=3$ and
\begin{equation} \label{r2small}
 r_2 \le r_3<r_1 \ \ {\rm or} \ \ r_2 < r_3 \leq r_1\,,
\end{equation}
then we get the inequalities
%\[
 $0 \le \frac{r_1-2r_2+r_3}{3(r_1-r_2)} \le 1$.
%\]
Let
\[
A_1 = \left(\begin{array}{ccc}
\cos \alpha & -\sin \alpha & 0 \\
\sin \alpha & \cos \alpha & 0 \\
0& 0 & 1 \\
\end{array} \right),\qquad
%\]
%and
%\[
A_2 = \left(\begin{array}{ccc}
1& 0 & 0 \\
0 & \cos \phi & -\sin \phi  \\
0 & \sin \phi & \cos \phi  \\
\end{array} \right),
\]
then the matrix $A_2 A_1 {\rm Ric}^\sharp A_1^T A_2^T$, where $A_i^T$ is the transpose of matrix $A_i\ (i=1,2)$, has all diagonal elements equal if and only if $\cos^2 \alpha = \frac{r_1-2r_2+r_3}{3(r_1-r_2)}$ and $\cos^2 \phi = \frac{1}{2}$. Hence, \eqref{matrixeq1} and \eqref{matrixeq2} hold
 for $A = A_2 A_1$ with $\alpha = \arccos\sqrt{\frac{r_1-2r_2+r_3}{3(r_1-r_2)}}$ and $\phi=\pi/4$.
If $k$ is such that there exists a Hadamard matrix $H_k$, then $A = \frac{1}{\sqrt{k}}\,H_k$ satisfies both \eqref{matrixeq2} and \eqref{matrixeq1diag}.
Indeed, \eqref{matrixeq2} holds because $A$ is an orthogonal matrix, and \eqref{matrixeq1diag} holds because $a_{ij}^2=\frac{1}{k}$ for all $i,j = 1 , \ldots , k$.
\end{proof}

\begin{Remark} \rm
We note that in the proof of Lemma \ref{lemRicdiagkdim}, in all considered dimensions the entries of orthogonal matrix $A$ are either constant, or, in case $k=3$, smoothly depending on eigenvalues of the Ricci tensor on any set where \eqref{r2small} holds. Hence, a smooth decomposition
\eqref{Eq-mD-k-TM}
%$TM = \mD_1 , \ldots, \mD_k$
on a neighborhood of any point of $M$ can obtained using these constructions. Moreover, if there exists a global orthonormal frame $X_1 , \ldots , X_k$ on $M$, such that every $X_j$ is everywhere an eigenvector of $ {\rm Ric}^\sharp$, and either: $k=3$ and $X_2$ is everywhere an eigenvector with the lowest eigenvalue of $ {\rm Ric}^\sharp$, or $k>3$ and there exists a Hadamard matrix $H_k$, then we can obtain another global frame (similarly as in Example~\ref{ex4dim}) and thus a global decomposition \eqref{Eq-mD-k-TM}.
%$TM = \mD_1 , \ldots, \mD_k$.
\end{Remark}

\begin{Proposition}
Let  $k=3$ or $k$ be such that there exists a Hadamard matrix $H_k$. Let $M=G/H$ be a $k$-dimensional homogeneous space, where $H$ is a maximal connected Lie subgroup of the compact connected Lie group $G$. Then there exist $k$ distinct, $G$-invariant decompositions of $TM$ in one-dimensional distributions: $TM = \mD^i_1+ \ldots +\mD^i_k$, and $k$ distinct $G$-invariant metrics $g_i\ (i=1,\ldots,k)$, such that distributions $\mD^i_1, \ldots , \mD^i_k$ are pairwise $g_i$-orthogonal, and each $g_i$ is critical for the Einstein-Hilbert action %\eqref{Eq-Smix-g}
with respect to variations adapted to the corresponding decomposition. Also, for $i \geq 2$ the metric $g_i$ is non-Einstein.
\end{Proposition}

\begin{proof}
For $i=1,\ldots,k$, let $f_i$ be a positive semidefinite bilinear form on $\RR^k$ with $(i-1)$-dimensional kernel, and let $T_i$ be a $G$-invariant $(0,2)$-tensor on $M$ corresponding to $f_i$. According to \cite[Theorem 1.1]{Pulemotov}, for each $i=1,\ldots ,k$ there exists a $G$-invariant metric $g_i$ such that $T_i$ is its Ricci tensor.
Then there exists a $G$-invariant orthonormal frame, for which ${\rm Ric}^\sharp(g_i)$ is diagonal.
%either
%of the form \eqref{Ric-abc}.
By Lemma~\ref{lemRicdiagkdim}, there exists a $G$-invariant orthonormal frame, in which ${\rm Ric}^\sharp(g_i)$ has equal elements on its diagonal.
Elements of this frame define the $g_i$-orthogonal decomposition $TM = \mD^i_1 \oplus \ldots \oplus \mD^i_k$ with one-dimensional distributions $\mD^i_1 , \ldots , \mD^i_k$,
and $g_i$ is critical for the Einstein-Hilbert action with respect to variations adapted to this decomposition.
For $i \geq 2$, the Ricci tensor of $g_i$ has non-trivial kernel, hence is not proportional to $g_i$,
so the metric obtained in this case is non-Einstein.
\end{proof}

\begin{Remark} \label{remarkEinsteinmetricconnection} \rm
According to \cite[Eq.~(17.10)]{ap}, the Euler-Lagrange equations of the action \eqref{Eq-EH} for variations of metric are those of the action \eqref{EHonlyg} with ${\rm Ric}$ replaced by the Ricci tensor $\overline{\rm Ric}$ of connection $\bar\nabla=\nabla+\I$.
Hence, similarly to \cite[Proposition 2.3(2)]{RN-21}, for adapted variations of metric, the Euler-Lagrange equation for the action \eqref{Eq-EH} is \eqref{E-3dim} with $\overline{\rm Ric}$ instead of ${\rm Ric}$. Action \eqref{Eq-Smix-g} for all distributions one-dimensional becomes, up to constant factor, \eqref{Eq-EH}, so in this case $\overline\Ric_{\,\mD} = \overline{\rm Ric}$ in \eqref{E-geom}.

Solutions of \eqref{E-3dim} given in Examples~\ref{ex3dim}, \ref{ex4dim} and Lemma~\ref{lemRicdiagkdim} require only constant scalar curvature and symmetry of the Ricci tensor, and therefore can be generalized to the case of metric-compatible connections (which always have symmetric Ricci tensors) with constant scalar curvature -- as long as those connections satisfy conditions of Theorem \ref{corI}, which in this case (all distributions are one-dimensional) reduce to:
\begin{equation}\label{E-connection2}
 \< ( \I - \I^\wedge )_Y Z , X \> =0\,,
\end{equation}
if each of $X,Y,Z$ belongs to a different distribution among $\mD_1 , \ldots , \mD_k$.
In particular, metric-compatible connections with $\I = \I^\wedge$ satisfy \eqref{E-connection2} for all decompositions
\eqref{Eq-mD-k-TM}.
%$TM = \mD_1 \oplus \ldots \oplus \mD_k$.
%For one-dimensional distributions and metric connections the only restriction for critical contorsion from
%Theorem \ref{corI} is equation \eqref{ELI5XYZ}.
\end{Remark}

\section{Critical metrics and statistical connections}
\label{sec:cont-stat}

In this part, we use the results of Sections~\ref{sec:adapted-metric} and \ref{sec:contorsion}
and study Euler-Lagrange equations of the action \eqref{Eq-Smix-g} with variations of both $g$ and $\I$,
i.e., vanishing of partial gradients $\delta_g \bar J_\mD = \lambda\,g$ and $\delta_\I\bar J_\mD=0$, see \eqref{E-delta-g-J} and \eqref{E-delta-I-J}, for statistical connections.
In Section~\ref{sec:contorsion-statistical} we study variations of $\I$ on locally twisted products.
In Section~\ref{sec:metric-stat} we consider variations of $\I$ among tensors corresponding to statistical connections, which give more %options
possibilities for critical~points.

\subsection{Twisted products}
\label{sec:contorsion-statistical}

In this section, we consider $(M, g;\mD_1,\ldots,\mD_k)$ with $k>2$ and adapted metric such that
all distributions
%$\mD_\mu$
are pairwise mixed totally geodesic and pairwise mixed integrable.
By Corollary~\ref{corstatcritforall}, for critical pairs $(g,\I)$ all distributions $\mD_\mu$
are also totally umbilical and integrable, which together with the above assumption gives us locally twisted products, see~\cite{MRS-99}.

%%%%%%%%%%%%%%%%%%%%%
%\begin{Definition}\rm
Let $(M_1,g_{1}),\ldots,(M_{k},g_{k})$ be pseudo-Riemannian manifolds, and $n_\mu=\dim M_\mu$.
%A general construction, called \textit{twisted product} was studied in \cite{MRS-99}: in this case,
A \textit{twisted product} is the product $M=M_1\times\ldots\times M_k$ with the metric $g=u_1^2\,g_{1}\oplus\ldots\oplus u_k^2\,g_{k}$, where $u_\mu$ for $\mu\ge 1$ are smooth positive functions on $M$, and $u=(u_1,\ldots,u_k)$ is called a twist function.
The %leaves
submanifolds
tangent to ${\cal D}_\mu\ (\mu\ge1)$ are totally umbilical with the mean curvature vectors
%\[
 $H_\mu=-n_\mu P^\bot_\mu\nabla(\log u_\mu)$.
%\]
%(Some authors study a special case of twisted product, called a \textit{multiply twisted product},
%where $u_1=1$ and $u_i:F_1\times F_i\to(0,\infty)$ for $i\ge 2$, e.g., \cite{wang}).
If $u$ is independent on $M_2,\ldots, M_{k}$ and $u_1\equiv1$, then we get a \textit{warped product}, see e.g., \cite{Dimitru}. %,PK2019}.
% with its warping function~$u$.
%\end{Definition}
%
By Proposition \ref{integrableDperp} below, all distributions tangent to the factors of the twisted product are pairwise mixed totally geodesic and mixed integrable. %(since all terms in \eqref{E-Kos} vanish).
%%
%Note that
Even if $\mD_\mu\ (1\le \mu\le k)$ are totally geodesic and integrable, they may not be pairwise mixed integrable,
e.g., when $\mD_\mu\ (\mu=1,2,3)$ are one-dimensional distributions on $SU(2)$ defined by the standard basis of its Lie algebra.
%one-dimensional distributions spanned by orthonormal vector fields on $S^3$ with its standard metric. %, such that each defines a Hopf fibration.
On the other hand, we get the following consequence of \cite[Theorem~1]{RS-99}:

\begin{Proposition} \label{integrableDperp}
Let $g$ be an adapted metric on $(M;\mD_1,\ldots,\mD_k)$ such that all $\mD_\mu^\perp$ are integrable.
Then all $\mD_\mu$ are integrable, pairwise mixed integrable and pairwise mixed totally geodesic.
\end{Proposition}

\begin{proof}
Since $\mD_\mu=\bigcap_{\,\nu\ne\mu}\mD_\nu^\bot$, each distribution $\mD_\mu$ is integrable
%(tangent to a foliation)
as the intersection of integrable distributions.
Without loss of generality, we can consider case $k=3$.
By \cite[Theorem 1]{RS-99}, $M$ is locally diffeomorphic to the product of neighborhoods in integral manifolds of $\mD_\mu$.
Hence, locally, $M = M_1 \times M_2 \times M_3$, where $TM_\mu=\mD_\mu$ for $\mu=1,2,3$.
Using the Koszul formula (expression of $\nabla$
%the Levi-Civita connection
explicitly in terms of the Riemannian metric)
\[
%\label{E-Kos}
 2 \< \nabla_X Y, Z\> = X (\<Y,Z\>) + Y (\<X,Z\>) - Z (\<Y,X\>) + \< [X,Y], Z\> - \< [Y,Z], X\> - \< [X,Z], Y\>\,,
\]
%Koszul formula
%\eqref{E-Kos}
one can show that any pair, e.g., $\mD_1$ and $\mD_2$, is mixed totally geodesic and mixed integrable,~i.e.,
\[
\< P_3 \nabla_{P^\perp_3 X}\, P^\perp_3 Y,\, Z\> = 0,\quad X \in TM_1,\quad Y \in TM_2,\quad Z \in TM_3,
\]
by extending $X,Y,Z$ to %constant
vector fields tangent to integral manifolds of $\mD_1 , \mD_2 , \mD_3$, with constant coefficients in some coordinate system.
%We will show that
%\[
%\< P_3 \nabla_{P^\perp_3 X}\, P^\perp_3 Y,\, Z\> = 0,\quad X \in TM_1,\quad Y \in TM_2,\quad Z \in TM_3,
%\]
%therefore, $\mD_1$ and $\mD_2$ are mixed totally geodesic and mixed integrable. Let $x^j_1, \ldots x^j_{n_j}$ be coordinates on $M_j$ and let
%\[
% X = X_1 \partial_{x^1_{1}} + \ldots + X_{n_1} \partial_{x^1_{n_j}},\quad
% Y = Y_1 \partial_{x^2_{1}} + \ldots + Y_{n_2} \partial_{x^2_{n_2}},\quad
% Z = Z_1 \partial_{x^3_{1}} + \ldots + Z_{n_3} \partial_{x^3_{n_3}}.
%\]
%As $\< P_3 \nabla_{P^\perp_3 X} (P^\perp_3 Y), Z\>$ is tensorial in $X,Y,Z$, we can extend vectors $X,Y,Z$ given at a point to vector fields on its neighborhood by setting functions $X_j$ to be constant. Then as $\< \partial_{x^i_{m}}, \partial_{x^j_{l}} \>= 0$ for $i\ne j$ and $[\partial_{x^i_{m}}, \partial_{x^j_{l}}] =0$ for all $i,j,l,m$,
%in the Koszul formula \eqref{E-Kos}
%\begin{equation*}
% 2 \< \nabla_X Y, Z\> = X (\<Y,Z\>) + Y (\<X,Z\>) - Z (\<Y,X\>) + \< [X,Y], Z\> - \< [Y,Z], X\> - \< [X,Z], Y\>
%\end{equation*}
%all terms vanish.
\end{proof}

Thus, all results in this section apply to decompositions \eqref{Eq-mD-k-TM} with all $\mD_\mu^\perp$ integrable.
%
%\begin{Remark}\rm
%Note that triply twisted product is mixed totally geodesic and mixed integrable.
%Let $M = M_1 \times M_2 \times M_3$, let $g_j$ be a Riemannian metric on $M_j$, and let $g = u_1 g_1 + u_2 g_2 + u_3 g_3$ be a metric on $M$,
%where $u_j>0$ are smooth functions on $M$. We show that on $(M,g)$ we~have
%\[
% \< P_3 \nabla_{P^\perp_3 X} (P^\perp_3 Y), Z\> = 0,\quad X \in TM_1,\quad Y \in TM_2,\quad Z \in TM_3.
%\]
%Let $x^j_1, \ldots x^j_{n_j}$ be coordinates on $M_j$ and let
%\[
% X = X_1 \partial_{x^1_{1}} + \ldots + X_{n_1} \partial_{x^1_{n_j}},\quad
% Y = Y_1 \partial_{x^2_{1}} + \ldots + Y_{n_2} \partial_{x^2_{n_2}},\quad
% Z = Z_1 \partial_{x^3_{1}} + \ldots + Z_{n_3} \partial_{x^3_{n_3}}.
%\]
%As $\< P_3 \nabla_{P^\perp_3 X} (P^\perp_3 Y), Z\>$ is tensorial in $X,Y,Z$, we can extend $X,Y,Z$ from a point to its neighborhood by setting functions $X_j$ %to be constant. Then as $\< \partial_{x^i_{m}}, \partial_{x^j_{l}} \>= 0$ for $i\ne j$ and $[\partial_{x^i_{m}}, \partial_{x^j_{l}}] =0$ for all $i,j,l,m$,
%in the Koszul formula:
%\begin{equation*}
% 2 \< \nabla_X Y, Z\> = X (\<Y,Z\>) + Y (\<X,Z\>) - Z (\<Y,X\>) + \< [X,Y], Z\> - \< [Y,Z], X\> - \< [X,Z], Y\>,
%\end{equation*}
%all terms vanish.
%\end{Remark}
%
%Note that it even if $\mD_j$, $j\le3$ are orthogonal, totally geodesic and integrable, they may not be mixed integrable (e.g., one dimensional distributions spanned by orthogonal unit fields on $S^3$, such that each defines a Hopf fibration).
%
%
In the next lemma, we find variations of each term
%in the expression
of $2\,\bar Q(\mD_\nu,g_t,\I)$ in \eqref{E-barQ} for statistical connection on a manifold with pairwise mixed totally geodesic and pairwise mixed integrable distributions.
%for particular choice of $\I$ (see Lemma~\ref{P-dT-3}).

\begin{Lemma}\label{propdeltaQforstatistical}
Let a contorsion tensor $\I$ of a statistical connection be critical for the action \eqref{Eq-Smix-g} with fixed adapted metric $g$ on $(M;\mD_1,\ldots,\mD_k)$, and let all $\mD_\mu$ be pairwise mixed totally geodesic and pairwise mixed integrable.
Then for $\mD_\mu$-variations of metric we have
\begin{eqnarray*}
%\label{deltaQbarforstatistical}
 &&\hskip-6mm 2\,{\delta_g\bar Q}_\mu({E}_{\mu,a},{E}_{\mu,b}) {=}
 \<\I_{{\mu,a}}, \I_{{\mu,b}}\>_{\,|\mD_\mu^\bot} {-}\<\I_{{\mu,a}}{E}_{\mu,b}, \tr_{\,\mu}^\bot\I\> {-}2\,\<(\tilde A_\mu)_{{E}_{\mu,a}},\I_{{\mu,b}}\>
 {+} 2\< h_\mu({E}_{\mu,b},\,\cdot), \I_{{\mu,a}}\>\nonumber \\
 && -\,\<\I_{{\mu,b}}{E}_{\mu,a}, H_\mu - \tilde H_\mu\> -\<\tr_{\,\mu}^\top\I, {E}_{\mu,a}\>\< \tilde H_\mu, {E}_{\mu,b}\>
 +\<\tr_{\,\mu}^\bot\I, {E}_{\mu,b}\> \<\tilde H_\mu, {E}_{\mu,a}\> \nonumber \\
%%%%%%%%%%%%%%%%%%%%%%%%%%%%%%%%%%%
 && +\sum\nolimits_{\,\nu\ne \mu}\big(\<\I_{{\mu,a}}, \I_{{\mu,b}}\>_{\,|\mD_\nu} -\<\I_{{\mu,a}}{E}_{\mu,b}, \tr_{\,\nu}^\top\I\> -2\,\<(A_\nu)_{{E}_{\mu,a}}, \I_{{\mu,b}}\>
+2\,\< \tilde h_\nu({E}_{\mu,b},\,\cdot),\, \I_{{\mu,a}}\>\nonumber \\
&& -\,\<\I_{{\mu,b}}{E}_{\mu,a}, \tilde H_\nu - H_\nu\> {-}\<\tr_{\,\nu}^\bot\I, {E}_{\mu,a}\>\< H_\nu, {E}_{\mu,b}\> {+}\<\tr_{\,\nu}^\top\I, {E}_{\mu,b}\> \<H_\nu, {E}_{\mu,a}\>\big) .
%, \  {E}_{\mu,a},{E}_{\mu,b}\in\mD_\mu\,.
\end{eqnarray*}
\end{Lemma}

\begin{proof}
From the last claim of Corollary~\ref{corstatcritforall} it follows that $\< \I_X Y,Z\>=0$ if each of $X,Y,Z$ belongs to a different distribution.
From \cite[Eqs.~(67)--(69)]{rz-3} we get, respectively:
\begin{eqnarray*}
\dt \sum\nolimits_{\xi} \< \I^*, \I^\wedge\>_{| V_\xi}
%&=& -\sum\nolimits_{\nu\ne\mu} \sum {B}(E_{\mu, i}, E_{\mu, j}) \< \I_{{\mu,i}} E_{\nu,a}, \I_{\nu,a} E_{\mu,j}\> \\
%dual
%&& - \sum {B}(E_{\mu, i}, E_{\mu, j}) \< \I_{{\mu,i}}{\cal E}_{\mu, a}, \I_{{\mu,a}}E_{\mu, j}\> \\
%
%&=& -2 \sum\nolimits_{\nu\ne\mu} \sum {B}(E_{\mu, i}, E_{\mu, j}) \< \I_{{\mu,i}} E_{\nu, a}, \I_{{\nu,a}}E_{\mu, j}\> \\
%
 &=& -2 \sum\nolimits_{\nu\ne\mu} \sum {B}(E_{\mu, i}, E_{\mu, j}) \< \I_{{\mu,i}} E_{\nu,a}, E_{\nu,b}\> \< E_{\nu,b}, \I_{\nu,a} E_{\mu,j}\> \\
&& -2 \sum\nolimits_{\nu\ne\mu} \sum {B}(E_{\mu, i}, E_{\mu, j}) \< \I_{{\mu,i}} E_{\nu,a}, E_{\mu,b}\> \< E_{\mu,b}, \I_{\nu,a} E_{\mu,j}\>\,, \\
%\end{eqnarray*}
%%%%%%%%%%%%%%%%%%%%%%%%%%%%%%%%%%%%%%%%%%%%%%%
%From \cite[Eq.~(68)]{rz-3} we get
%\begin{eqnarray*}
\dt\sum\nolimits_{\xi} \< \Theta, A_\xi\>
%&=& - 2\sum\nolimits_{\nu\ne\mu}\sum {B}(E_{\mu, i}, E_{\mu, j}) \< h_\nu (E_{\nu,a}, E_{\nu,b}), E_{\mu,i}\>\< E_{\mu,j}, \I_{\nu,a} E_{\nu,b}\> \\
%dual
%&& -2\sum {B}(E_{\mu, i}, E_{\mu, j}) \< {\tilde h}_\mu ({\cal E}_{\mu,a}, {\cal E}_{\mu,b}), E_{\mu,i}\>\< E_{\mu,j}, \I_{{\mu,a}}{\cal E}_{\mu,b}\> \\
%
 &=& -4\sum\nolimits_{\nu\ne\mu} \sum {B}(E_{\mu, i}, E_{\mu, j}) \< h_\nu (E_{\nu,a}, E_{\nu,b}), E_{\mu,i}\>
\< E_{\mu,j}, \I_{\nu,a} E_{\nu,b}\>,\\
%\end{eqnarray*}
%From \cite[Eq.~(69)]{rz-3} we get
%%%%%%%%%%%%%%%%%%%%%%%%%%%%%%%%%%%%%%%%%%%%
%\begin{eqnarray*}
 \dt \sum\nolimits_{\xi} \< \Theta, T^\sharp_\xi\>
%&=& -2\sum\nolimits_{\nu\ne\mu} \sum {B}(E_{\mu, i}, E_{\mu, j}) \< T_\nu (E_{\nu,a}, E_{\nu,b}), E_{\mu,i}\>\< E_{\mu,j}, \I_{\nu,a} E_{\nu,b}\> \\
%dual
%&& -2\sum {B}(E_{\mu, i}, E_{\mu, j}) \< {\tilde T}_\mu ({\cal E}_{\mu,a}, {\cal E}_{\mu,b}), E_{\mu,i}\>\< E_{\mu,j}, \I_{{\mu,a}}{\cal E}_{\mu,b}\> \\
&=& -4 \sum\nolimits_{\nu\ne\mu} \sum {B}(E_{\mu, i}, E_{\mu, j}) \< T_\nu (E_{\nu,a}, E_{\nu,b}), E_{\mu,i}\>
\< E_{\mu,j}, \I_{\nu,a} E_{\nu,b}\> .
\end{eqnarray*}
For critical statistical connections the above $\dt \< \Theta, {T}^\sharp_\xi\> =0$, because $T_\xi =0$ by Corollary \ref{corstatcritforall}.

Similarly, from \cite[Eq.~(70)]{rz-3} it follows that %for critical statistical connections
$\dt \< \Theta, {\tilde T}^\sharp_\xi \> =0$, because %either: connection is critical for all variations and then
all ${\tilde T}_\xi =0$, as all $\mD_\mu$ are integrable and pairwise mixed integrable. %, or it is critical among statistical connections and then from \eqref{ELSmixIstat2} we get sums of symmetric and antisymmetric in $E_{\mu, i}, E_{\mu,j}$ terms.

From \cite[Eq.~(71)]{rz-3} for statistical connections:
\begin{eqnarray*}
\dt \sum\nolimits_{\xi} \< \Theta, {\tilde A}_\xi\>
 &=& 4 \sum\nolimits_{\nu\ne\mu} \sum {B}(E_{\mu, i}, E_{\mu, j}) \< h_\mu (E_{\mu,k}, E_{\mu,j} ), E_{\nu,a}\> \< E_{\nu,a}, \I_{{\mu,k}}E_{\mu,i}\> \,.
\end{eqnarray*}
%From \cite[Eq.~(72)]{rz-3} it follows that for statistical connections $\dt \< \tr^\top_\nu \I, \tr^\perp_\nu \I^*\>=0$.
From \cite[Eqs.~(72)--(75)]{rz-3} for statistical connections, we obtain the following:
\begin{eqnarray*}
 && \dt \< \tr^\top_\nu \I, \tr^\perp_\nu \I^*\> = 0\,,\\
%%%%%%%%%%%%%%%%%%%%%%%%%%%
 && \dt \sum\nolimits_{\xi} \< \tr^\perp_\xi \I^*, \tr^\perp_\xi \I\>
% = -\frac{1}{2} \sum\nolimits_{\nu\ne\mu} \sum {B}(E_{\mu, i}, E_{\mu, j}) \< 2 \I_{\mu, j} E_{\mu,i}, \tr^\top_\nu \I\> \\
%dual
%&&
%- \frac{1}{2}\sum {B}(E_{\mu, i}, E_{\mu, j}) \< 2 \I_{\mu, j} E_{\mu,i}, \tr^\perp_\mu \I\>
%%&=& - \sum\nolimits_{\nu\ne\mu} \sum {B}(E_{\mu, i}, E_{\mu, j}) \< \I_{\mu, j} E_{\mu,i}, \tr^\top_\nu \I\> \\
%%&& - \sum\nolimits_{\nu\ne\mu} \sum {B}(E_{\mu, i}, E_{\mu, j}) \< \I_{\mu, j} E_{\mu,i}, \tr^\top_\nu \I\> \\
= -2\sum\nolimits_{\nu\ne\mu} \sum {B}(E_{\mu, i}, E_{\mu, j}) \< \I_{\mu, j} E_{\mu,i}, \tr^\top_\nu \I\>, \\
%\end{eqnarray*}
%From \cite[Eq.~(74)]{rz-3} for statistical connections, we have the following:
%\begin{eqnarray*}
 && \dt \sum\nolimits_{\xi} \< \tr^\top_\xi (\I^* - \I ), {\tilde H}_\xi - H_\xi\> =
 \sum {B}(E_{\mu, i}, E_{\mu, j}) \big(\< \tr^\perp_\mu \I, E_{\mu,i}\> \< E_{\mu,j}, {\tilde H}_\mu\> \\
%dual
&&\quad +
%\sum\nolimits_{\nu\ne\mu} \sum {B}(E_{\mu, i}, E_{\mu, j})
 \sum\nolimits_{\nu\ne\mu}\< \tr^\top_\nu \I, E_{\mu,j}\> \< E_{\mu,i}, H_\nu\> \big),\\
%\end{eqnarray*}
%From \cite[Eq.~(75)]{rz-3} for statistical connections, we have
%\begin{eqnarray*}
 && \dt \sum\nolimits_{\xi} \< \tr^\perp_\xi (\I^* - \I ), {\tilde H}_\xi - H_\xi\>
 = \sum {B}(E_{\mu, i}, E_{\mu, j})\big(\sum\nolimits_{\nu\ne\mu}(\< \I_{\mu,j} E_{\mu,i}, {\tilde H}_\nu - H_\nu\> \\
 &&\quad +\, \< \tr^\perp_\nu \I, E_{\mu,i}\> \< H_\nu, E_{\mu,j}\> )
%dual
 +\< \I_{\mu,j} E_{\mu,i}, H_\mu - {\tilde H}_\mu\> + \< \tr^\top_\mu \I, E_{\mu,i}\> \< {\tilde H}_\mu, E_{\mu,j}\> \big),
\end{eqnarray*}
respectively, and that completes the proof.
\end{proof}

Put $\I_{ Z }^\flat (X,Y) = \< \I_{ Z}X, Y\>$.
In the next theorem we find
%a preliminary form of
the Euler-Lagrange equation \eqref{E-delta-g-J} under our assumptions about distributions, for statistical connections, and with \eqref{E-delta-I-J} satisfied. This result will be improved in further corollaries, according to specific dimensions of the distributions.

\begin{Theorem}
Let $g$ be an adapted metric on $(M;\mD_1,\ldots,\mD_k)$ such that all $\mD_\mu$ are pairwise mixed totally geodesic and pairwise mixed integrable.
Then a pair $(g,\I)$, where $\I$ is the contorsion tensor of a statistical connection on $(M,g)$, is critical for the action \eqref{Eq-Smix-g} with respect to adapted variations of $g$ preserving the volume of $(M,g)$ and all variations of $\I$ if and only if
$(M,g)$ is locally a~twisted product, all conditions of Corollary~\ref{corstatcritforall} hold and the following Euler-Lagrange equations \eqref{E-delta-g-J} are valid:
\begin{eqnarray} \label{ELmixedgeodesicmixedintegrable}
 && \big( {\bar S}_{\mD_1 \ldots \mD_k}
 +\Div \big(\big(1 - \frac{2}{n_\mu}\, \big) H_\mu - {\tilde H}_\mu \big) + \lambda\big) g_\mu
 - {\tilde H}_\mu^\flat \otimes {\tilde H}_\mu^\flat
%%contorsion part:
 +\,\I_{{\tilde H}_\mu}^\flat
 \nonumber \\
 && +\sum\nolimits_{\nu\ne\mu} \big(\frac{2}{n_\nu} - 1 \big) (P_\mu H_\nu)^\flat \otimes (P_\mu H_\nu)^\flat = 0,\quad
 %X,Y \in \mD_\mu,\quad
 \mu = 1,\ldots, k\,.
\end{eqnarray}
\end{Theorem}

\begin{proof}
For totally umbilical, pairwise mixed totally geodesic distributions we obtain
\begin{eqnarray*}
\frac12\,\Upsilon_{h_\nu,h_\nu} = \frac{1}{n_\nu}\,H_\nu^\flat \otimes H_\nu^\flat,\quad
\frac12\,\Upsilon_{{\tilde h}_\mu, {\tilde h}_\mu}= \sum\nolimits_{\nu\ne\mu} \frac{1}{n_\nu}\,(P_\mu H_\nu)^\flat \otimes (P_\mu H_\nu)^\flat\, .
\end{eqnarray*}
For $\mu = 1,\ldots, k$ and $X,Y \in \mD_\mu \ne \mD_\nu$, we have
%\begin{equation*}
$( \Div {\tilde h}_\nu )(X,Y) = \frac{1}{n_\mu}\,\<X,Y\>\, \Div( P_\nu H_\mu )$,
%\end{equation*}
and thus
%\[
 $\sum\nolimits_{\nu\ne\mu} ( \Div {\tilde h}_\nu )(X,Y) = \frac{1}{n_\mu}\,\<X,Y\>\,\Div H_\mu$. %, \; \mu = 1,\ldots, k$.
%\]
For totally umbilical, pairwise mixed totally geodesic, integrable and pairwise mixed integrable distributions we~get
\[
\delta Q_\mu
%&=& (1 - \frac{2}{n_\mu} ) (\Div H_\mu) g_\mu - {\tilde H}_\mu^\flat \otimes {\tilde H}_\mu^\flat \\
%
%&& + \sum\nolimits_{\nu\ne\mu} \frac{1}{n_\nu}\,(P_\mu H_\nu)^\flat \otimes (P_\mu H_\nu)^\flat
%%%dual
%-\sum\nolimits_{\nu\ne\mu} (P_\mu H_\nu)^\flat \otimes (P_\mu H_\nu)^\flat \\
%
%&& +\sum\nolimits_{\nu\ne\mu} \frac{1}{n_\nu}\, (P_\mu H_\nu)^\flat \otimes (P_\mu H_\nu)^\flat
%
%+\sum\nolimits_{\nu\ne\mu} (\Div {\tilde H}_\nu) g_\mu \\
%\end{eqnarray*}
%i.e.,
%\begin{eqnarray*}
%\delta Q_\mu
= \Div\big(\big(1 -\frac{2}{n_\mu}\big) H_\mu +\sum\nolimits_{\nu\ne\mu}{\tilde H}_\nu\big) g_\mu - {\tilde H}_\mu^\flat \otimes {\tilde H}_\mu^\flat
 + \sum\nolimits_{\nu\ne\mu} \big(\frac{2}{n_\nu} - 1\big)\,(P_\mu H_\nu)^\flat \otimes (P_\mu H_\nu)^\flat\, .
\]
For such distributions, by the above and Lemma~\ref{propdeltaQforstatistical},
the Euler-Lagrange equations for the action \eqref{Eq-Smix-g}
%for all distributions: totally umbilical, integrable, pairwise mixed totally geodesic, pairwise mixed integrable; and
for all variations of $\I$ and adapted variations of $g$ preserving the volume of $\Omega$ are
%the following:
\begin{eqnarray*}
&& 2\,\< X,Y\> \Div\big(\big(1 -\frac{2}{n_\mu}\big) H_\mu +\sum\nolimits_{\nu\ne\mu}{\tilde H}_\nu \big) -2\,\< {\tilde H}_\mu, X\>\<{\tilde H}_\mu, Y\> \\
&& +\,2 \sum\nolimits_{\nu\ne\mu} \Big[ \big(\frac{2}{n_\nu} -1 \big) \< H_\nu, X \> \< H_\nu, Y \> \\
%
%%contorsion part:
%
&& +\,2
%\sum\nolimits_{\nu\ne\mu}
\sum \big(\< \I_{X} E_{\nu,a}, E_{\nu,b}\> \< E_{\nu,b}, \I_{\nu,a} Y\>
 +\<\I_{X} E_{\nu,a}, E_{\mu,b}\> \< E_{\mu,b}, \I_{\nu,a} Y\>\big) \\
&& -\,2
%\sum\nolimits_{\nu\ne\mu}
\frac{1}{n_\nu}\,\big(\< H_\nu, X\> \< Y, \tr^\top_\nu \I\> +
 \< H_\nu, Y\> \< X, \tr^\top_\nu \I\>\big) \\
&& +\,
%\sum\nolimits_{\nu\ne\mu}
\frac{4}{n_\mu} \< P_\nu H_\mu, \I_{ X}Y\>
 -2\,
%\sum\nolimits_{\nu\ne\mu}
\< \I_{Y}{X}, \tr^\top_\nu \I\> \\
&& +\,\frac{1}{2}\,
%\sum\nolimits_{\nu\ne\mu}
(\< \tr^\top_\nu \I, Y\> \< X, H_\nu\> + \< \tr^\top_\nu \I, X\> \< Y, H_\nu\> ) \\
&& -\,
%\sum\nolimits_{\nu\ne\mu}
 \big(\< \I_{Y} X, {\tilde H}_\nu - H_\nu\>
 - \frac{1}{2}\,\<\tr^\perp_\nu \I, X\> \<H_\nu, Y\> - \frac{1}{2}\,\<\tr^\perp_\nu \I, Y\> \<H_\nu, X\> \big) \Big] \\
%%%%%%%%%%%%%%%%%%%
&& +\,\frac{1}{2}\,(\< \tr^\perp_\mu \I, X\> \< Y, {\tilde H}_\mu\> + \< \tr^\perp_\mu \I, Y\> \< X, {\tilde H}_\mu\>) \\
&& -\,\<\I_{Y}X, H_\mu -{\tilde H}_\mu\> +\frac{1}{2}\,\<\tr^\top_\mu\I, X\>\<{\tilde H}_\mu, Y\> +\frac{1}{2}\,\<\tr^\top_\mu\I, Y\>\<{\tilde H}_\mu, X\> \\
%%%%%conformal term corrected
&& +\,2\,( \bar{\rm S}_{\mD_1 \ldots \mD_k} - \Div {\cal H} + \lambda ) \< X,Y\> =0,\qquad X,Y \in \mD_\mu ,\quad \mu=1, \ldots, k\,,
\end{eqnarray*}
with ${\cal H}$ given by \eqref{defcalH}. %in Section~\ref{sec:prel}.
By claim 5 of Corollary~\ref{corstatcritforall}, $\< \I_X Y, Z\> =0$ when $X,Y,Z$ are not all from the same distribution; using claim 4 of Corollary~\ref{corstatcritforall}, \eqref{calH} and $H_\mu = \sum\nolimits_{\nu\ne\mu} P_\nu H_\mu$, we reduce the above Euler-Lagrange equations %reduces
to the following:
\begin{eqnarray*}
&& \big(\bar{\rm S}_{\mD_1 \ldots \mD_k} +\Div \big( \big(1 - \frac{2}{n_\mu} \big) H_\mu - {\tilde H}_\mu \big) + \lambda \big)\< X,Y\> \\
&& + \sum\nolimits_{\nu\ne\mu} \big(\frac{2}{n_\nu} - 1 \big) \< X, H_\nu\> \<Y, H_\nu\> - \< {\tilde H}_\mu, X\> \< {\tilde H}_\mu, Y\>
%\\
%%contorsion part:
%
%&&
+\,\< \I_{Y} X, {\tilde H}_\mu\> = 0\,,
\end{eqnarray*}
for all $X,Y \in \mD_\mu$ and $\mu=1, \ldots, k$,
which %, by \eqref{E-Dk-Smix},
yields \eqref{ELmixedgeodesicmixedintegrable}.
\end{proof}

The next corollaries consider alternatives for the number $j$ of one-dimensional distributions among our $k$ distributions:
%either
at most one distribution is one-dimensional (i.e., $j\le1$) in~Corollary~\ref{C2-sec4},
there are $j$ one-dimensional distributions for some $j\in [2, k-1]$ in Proposition~\ref{coralldim} and Corollary~\ref{coralldimNew}, and all distributions are one-dimensional (i.e., $j=k$) in~Corollary~\ref{C3-sec4}.
The second is the only case when non-trivial metrics (i.e., not metric products) and connections (i.e., not Levi-Civita) can be critical.
%not only products are possible as critical metrics.

\begin{Corollary}\label{C2-sec4}
Let $g$ be an adapted metric on $(M;\mD_1,\ldots,\mD_k)$ such that all $\mD_\mu$ are pairwise mixed totally geodesic and pairwise mixed integrable,
and $n_\mu>1$ for all $\mu\ge2$. Then a pair $(g,\I)$, where $\I$ is the contorsion tensor of a statistical connection on $(M,g)$, is critical for the action \eqref{Eq-Smix-g} with respect to adapted variations of $g$ preserving the volume of $(M,g)$ and all variations of $\I$
if and only if $(M,g)$ is locally a product, i.e., all $\mD_\mu$ are integrable and totally geodesic, and $\I$ is contorsion of any statistical connection satisfying
claims \ref{item2}, \ref{item5} and \ref{item4} of~Corollary~\ref{corstatcritforall}.
\end{Corollary}

\begin{proof}
%Let $(g,\I)$ be a critical pair.
By claim \ref{item1} of Corollary~\ref{corstatcritforall}, we get $H_\nu=0$ when $n_\nu >1$.
So if all $n_\nu>1$, then a critical pair $(g,\I)$, where $\I$ is statistical, can exist only when $g$ is the product metric.

Suppose now that $n_1=1$ and $n_\nu>1$ for all $1< \nu \le k$.
Then for all $1< \nu \le k$ we obtain from claim \ref{item3} of Corollary~\ref{corstatcritforall} that $H_\nu=0$ and from
claim \ref{item3} of Corollary~\ref{corstatcritforall} we obtain for all $1< \nu \le k$ that ${\tilde H}_\nu =0 = P_\nu H_1$, and it follows that also $H_1=0$.
Again, we obtain that $g$ is the product metric.
For a metric product of integral manifolds of
%distributions
$\mD_\mu$, the Euler-Lagrange equations \eqref{ELmixedgeodesicmixedintegrable} all become $\lambda =0$.
In that case, all statistical $\I$ satisfying claims \ref{item2}, \ref{item5} and 5 of Corollary~\ref{corstatcritforall}
are critical (also for variations of metric not preserving the volume of $(M,g)$).
%
%On the other hand, it is easily verified that a metric product with statistical connection satisfying claims \ref{item2}, \ref{item5} and 5 of Corollary~\ref{corstatcritforall} satisfies the Euler-Lagrange equation \eqref{ELmixedgeodesicmixedintegrable} and all equations of Corollary~\ref{corstatcritforall}.
\end{proof}

\begin{Proposition}\label{coralldim}
Let $g$ be an adapted metric on $(M;\mD_1,\ldots,\mD_k)$ such that all $\mD_\mu$ are pairwise mixed totally geodesic and pairwise mixed integrable,
and there exists $j\in [2, k-1]$ such that $n_\mu=1$ for $1 \le \mu \le j$ and $n_\rho >1$ for $j+1 \le \rho \le k$.
%, where $n_\nu = \dim \mD_\nu$.
%Let ${\cal H} = \sum_{\xi} H_\xi$.
%%
Then a pair $(g,\I)$, where $\I$ is the contorsion tensor of a~statistical connection,
%on $(M,g)$,
is critical for the action \eqref{Eq-Smix-g} with respect to adapted variations of $g$ preserving the volume of $(M,g)$ and all variations of $\I$
if and only if
\,$\I$ satisfies claims \ref{item2}, \ref{item5} and %5
\ref{item4} of Corollary~\ref{corstatcritforall}
and for all $j+1\le\rho\le k:\,\mD_\rho$ is totally geodesic, ${\tilde H}_\rho=0$ and %for all $X,Y \in \mD_\rho$:
\begin{subequations}
\begin{equation} \label{ELtwistedproductdimbigSmix}
%&&
 \sum\nolimits_{\,\nu=1}^j
 %(\frac{2}{n_\nu} - 1 )
 (P_\rho H_\nu)^\flat \otimes (P_\rho H_\nu)^\flat + \big( \frac{m\lambda}{2-m} \big) g_\rho =0 \,,
\end{equation}
and for all $1\le\mu\le j$:
\begin{equation}\label{ELtwistedproductdimoneSmix}
{\rm S}_{\,\mD_\mu, \mD_\mu^\perp} = \frac{2\lambda}{2-m},
%old
% \sum\nolimits_{\,\nu=1}^j \big(\| P_\mu H_\nu\|^2 - \frac{1}{2}\|H_\nu\|^2\big)
%  -\| {\tilde H}_\mu \|^2
% +\Div ( {\cal H} - H_\mu - {\tilde H}_\mu) + \frac{1}{2}\,\|{\cal H}\|^2 + \lambda = 0\,.
\end{equation}
\end{subequations}
where $m=\dim M$ and $\lambda$ is a constant.
\end{Proposition}

\begin{proof}
From Corollary~\ref{corstatcritforall} we obtain for $j+1 \le \rho \le k$ that $H_\rho=0 = {\tilde H}_\rho = P_\rho {\cal H}$ and thus ${\cal H}=\sum\nolimits_{\,\nu=1}^j H_\nu$.
By %\eqref{E-QI} ?!
\eqref{eqvarstat} and Corollary \ref{corstatcritforall}, we get
%\begin{equation} \label{SbarS}
$\bar{\rm S}_{\,\mD_1 \ldots \mD_k} = {\rm S}_{\,\mD_1 \ldots \mD_k}$.
Hence, the Euler-Lagrange equation \eqref{ELmixedgeodesicmixedintegrable} for all $X,Y \in \mD_\rho$ becomes %\eqref{ELtwistedproductdimbig1},
\begin{equation} \label{ELtwistedproductdimbig1}
%&&
 \sum\nolimits_{\,\nu=1}^j
 %(\frac{2}{n_\nu} - 1 )
 (P_\rho H_\nu)^\flat \otimes (P_\rho H_\nu)^\flat + \big( {\rm S}_{\,\mD_1 \ldots \mD_k} + \lambda \big) g_\rho  =0 \,,
\end{equation}
and for all $1 \le \mu \le j$ we %find from
have the following form of \eqref{ELmixedgeodesicmixedintegrable}:
\begin{equation}\label{ELtwistedproductdimone1}
- \Div (
%(1 - \frac{2}{n_\mu} )
H_\mu + {\tilde H}_\mu ) + %\bar
{\rm S}_{\,\mD_1 \ldots \mD_k} + \lambda  %g_\mu \\
 + \sum\nolimits_{\,\nu=1}^j
%(\frac{2}{n_\nu} - 1 )
\| P_\mu H_\nu\|^2 - \| {\tilde H}_\mu \|^2 = 0\,.
\end{equation}
%
%\end{equation}
For $1 \le \mu \le j$ we obtain
%\[
 $\| {\tilde h}_\mu \|^2
 %= \sum\nolimits_{\nu\ne\mu} \| P_\mu H_\nu \|^2
 = \sum\nolimits_{\nu =1}^j \| P_\mu H_\nu \|^2$\,,
%\]
and then, using the formula for the mixed scalar curvature of $\mD_\mu$ \cite{Walczak}, we get
\begin{equation} \label{Smixtwproddim1}
%\[
%\label{SmixDmu}
 {\rm S}_{\,\mD_\mu, \mD_\mu^\perp} = \Div ( {\tilde H}_\mu + H_\mu ) + \| {\tilde H}_\mu \|^2 - \sum\nolimits_{\,\nu =1}^j \| P_\mu H_\nu \|^2\, ,
%\]
\end{equation}
using the above in \eqref{ELtwistedproductdimone1} yields that the Euler-Lagrange equation \eqref{ELmixedgeodesicmixedintegrable} for $1 \leq \mu \leq j$ is
% \eqref{ELtwistedproductdimone}.
\begin{equation}\label{ELtwistedproductdimone}
 {\rm S}_{\,\mD_1 \ldots \mD_k} + \lambda - {\rm S}_{\,\mD_\mu, \mD_\mu^\perp} =0 .
%old
% \sum\nolimits_{\,\nu=1}^j \big(\| P_\mu H_\nu\|^2 - \frac{1}{2}\|H_\nu\|^2\big)
%  -\| {\tilde H}_\mu \|^2
% +\Div ( {\cal H} - H_\mu - {\tilde H}_\mu) + \frac{1}{2}\,\|{\cal H}\|^2 + \lambda = 0\,.
\end{equation}
For $j+1 \le \rho \le k$ from the formula for the mixed scalar curvature of $\mD_\rho$ we find ${\rm S}_{\,\mD_\rho, \mD_\rho^\perp}= - \sum_{\nu =1}^j \| P_\rho H_\nu \|^2$, hence taking trace of \eqref{ELtwistedproductdimbig1} yields
\begin{equation} \label{trELtwistedproductdimbig1}
( {\rm S}_{\,\mD_1 \ldots \mD_k} + \lambda ) n_\rho - {\rm S}_{\,\mD_\rho, \mD_\rho^\perp} =0 .
\end{equation}
Summing all equations \eqref{ELtwistedproductdimone} for $1 \leq \mu \leq j$ and all equations \eqref{trELtwistedproductdimbig1} for $j+1 \leq \rho \leq k$, and using the analogue of \eqref{E-Dk-Smix} for the Levi-Civita connection: $2 {\rm S}_{\,\mD_1 \ldots \mD_k} = \sum\nolimits_{\mu=1}^j {\rm S}_{\,\mD_\mu, \mD_\mu^\perp} + \sum\nolimits_{\rho=j+1}^k {\rm S}_{\,\mD_\rho, \mD_\rho^\perp}$ \cite{r-EH-k}, we obtain $m ({\rm S}_{\,\mD_1 \ldots \mD_k} + \lambda) - 2{\rm S}_{\,\mD_1 \ldots \mD_k} =0$,  and hence ${\rm S}_{\,\mD_1 \ldots \mD_k} = \frac{m \lambda}{2-m}$, which used in \eqref{trELtwistedproductdimbig1} yields ${\rm S}_{\,\mD_\mu, \mD_\mu^\perp} = \frac{2\lambda}{2-m}$. Using the above in \eqref{ELtwistedproductdimbig1} and \eqref{trELtwistedproductdimbig1} yields
(\ref{ELtwistedproductdimbigSmix},b). On the other hand, it can be verified that if $\I$ satisfies claims \ref{item2}, \ref{item5} and %5
\ref{item4} of Corollary~\ref{corstatcritforall}, for all $j+1\le\rho\le k:\,\mD_\rho$ is totally geodesic, ${\tilde H}_\rho=0$ and (\ref{ELtwistedproductdimbigSmix},b) hold, then the Euler-Lagrange equations \eqref{ELmixedgeodesicmixedintegrable} are valid.
\end{proof}

\begin{Remark} \rm
%Under assumptios of
In Proposition \ref{coralldim},
equations (\ref{ELtwistedproductdimbig1},b) are equivalent to, respectively:
\begin{subequations}
\begin{eqnarray}\label{ELtwistedproductdimbigold}
 \sum\nolimits_{\,\nu=1}^j(P_\rho H_\nu)^\flat \otimes (P_\rho H_\nu)^\flat
 +\big( \Div {\cal H} + \frac{1}{2}\,\| {\cal H} \|^2 - \frac{1}{2} \sum\nolimits_{\nu=1}^j \| H_\nu \|^2 + \lambda\big) g_\rho = 0\,, \\
%\end{eqnarray}
%and %for all $1\le\mu\le j$:
%\begin{eqnarray}
\label{ELtwistedproductdimoneold}
 \sum\nolimits_{\,\nu=1}^j \big(\| P_\mu H_\nu\|^2 - \frac{1}{2}\|H_\nu\|^2\big)
  -\| {\tilde H}_\mu \|^2
 +\Div ( {\cal H} - H_\mu - {\tilde H}_\mu) + \frac{1}{2}\,\|{\cal H}\|^2 + \lambda = 0\,.
\end{eqnarray}
\end{subequations}
%To obtain (\ref{ELtwistedproductdimbigold},b),
Indeed, from the formulas for mixed scalar curvatures ${\rm S}_{\,\mD_\mu, \mD_\mu^\perp}$ and ${\rm S}_{\,\mD_\rho, \mD_\rho^\perp}$ we get
\begin{eqnarray*}
 {\rm S}_{\,\mD_1,\ldots,\mD_k} &=& \frac{1}{2}\sum_{\mu=1}^j \big(\Div({\tilde H}_\mu + H_\mu ) + \|{\tilde H}_\mu\|^2
 - \sum_{\nu =1}^j \| P_\mu H_\nu \|^2 \big) -\frac{1}{2} \sum_{\rho=j+1}^k \sum_{\nu =1}^j \| P_\rho H_\nu \|^2 \\
% &=& \Div {\cal H} + \frac{1}{2} \sum\nolimits_{\,\mu=1}^j \| P_\mu {\cal H} \|^2 - \frac{1}{2} \sum\nolimits_{\nu=1}^j \| H_\nu \|^2 \\
 &=& \Div {\cal H} + \frac{1}{2}\, \| {\cal H} \|^2 - \frac{1}{2} \sum\nolimits_{\,\nu=1}^j \| H_\nu \|^2\, .
\end{eqnarray*}
Using the above in \eqref{ELtwistedproductdimbig1} and, together with \eqref{Smixtwproddim1}, in \eqref{ELtwistedproductdimone}, yields (\ref{ELtwistedproductdimbigold},b), respectively.
\end{Remark}

\begin{Corollary} \label{coralldimNew}
Let $g$ be an adapted metric on $(M;\mD_1,\ldots,\mD_k)$ such that all $\mD_\mu$ are pairwise mixed totally geodesic and pairwise mixed integrable,
and there exists $j\in [2, k-1]$ such that $n_\mu=1$ for $1 \le \mu \le j$ and $n_\rho >1$ for $j+1 \le \rho \le k$.
%, where $n_\nu = \dim \mD_\nu$.
%Let ${\cal H} = \sum_{\xi} H_\xi$.
%%
Suppose that $n_\rho > j$ for at least one $\rho \in \{ j+1, \ldots, k \}$.
Then a pair $(g,\I)$, where $\I$ is the contorsion tensor of a~statistical connection,
%on $(M,g)$,
is critical for the action \eqref{Eq-Smix-g} with respect to adapted variations of $g$ preserving the volume of $(M,g)$ and all variations of $\I$
if and only if \,$\I$ satisfies claims \ref{item2}, \ref{item5} and \ref{item4} of Corollary~\ref{corstatcritforall};
for all $j+1\le\xi\le k:\,\mD_\xi$ is totally geodesic, %${\tilde H}_\xi=0$
%and %for all $X,Y \in \mD_\rho$:
and
%for all $\nu \in \{ 1, \ldots, j \}$
we have
\begin{subequations}
\begin{eqnarray}\label{ELtwistedproductdimbigNew}
P_\xi H_\nu =0\,,\quad
\nu \in \{ 1, \ldots, j \},\\
%\end{equation}
%and for all $1\le\mu\le j$:
%\begin{equation}
\label{ELtwistedproductdimoneNew}
{\rm Ric} |_{\mD_\mu \times \mD_\mu} =0\,,\quad 1\le\mu\le j.
\end{eqnarray}
%and
%\begin{equation}\label{ELtwistedproductH}
%\| {\cal H} \|^2 - \sum\nolimits_{\,\nu =1}^j \| {\tilde H}_\nu \|^2 = - 2 \lambda
%\end{equation}
%for some constant $\lambda$.
\end{subequations}
\end{Corollary}

\begin{proof}
If $n_\rho > j$, then comparing ranks of tensors in
\eqref{ELtwistedproductdimbig1} yields
\begin{eqnarray}\label{divHlambda}
\nonumber
 & P_\rho H_\nu =0,\quad \nu \in \{ 1, \ldots, j \},\\
 & {\rm S}_{\,\mD_1 \ldots \mD_k}  = - \lambda\,.
% \Div {\cal H} + \frac{1}{2}\,\| {\cal H} \|^2 - \frac{1}{2} \sum\nolimits_{\,\nu=1}^j \| H_\nu \|^2 = - \lambda\,.
\end{eqnarray}
Moreover, for every $\xi \in \{ j+1, \ldots, k \}$ we get from the Euler-Lagrange equation \eqref{ELtwistedproductdimbig1} written for $\mD_\xi$ and \eqref{divHlambda}:
%\[
 $\sum\nolimits_{\,\nu =1}^j (P_\xi H_\nu)^\flat \otimes (P_\xi H_\nu)^\flat =0$, %\quad\xi \in \{ j+1, \ldots, k \}\,,
%\]
thus for all $E_{\xi, i}$ we obtain $\sum\nolimits_{\nu =1}^j \< H_\nu, E_{\xi,i} \>^2 =0$, but then also
\[
0 = \sum\nolimits_{i =1}^{n_\xi} \sum\nolimits_{\nu =1}^j \< H_\nu, E_{\xi,i} \>^2 =
 \sum\nolimits_{\nu =1}^j \sum\nolimits_{i =1}^{n_\xi} \< H_\nu, E_{\xi,i} \>^2 = \sum\nolimits_{\nu =1}^j \| P_\xi H_\nu \|^2
\]
and \eqref{ELtwistedproductdimbigNew} follows. Using \eqref{divHlambda} in \eqref{ELtwistedproductdimone}, we get \eqref{ELtwistedproductdimoneNew}. On the other hand, if (\ref{ELtwistedproductdimbigNew},b) hold, then for $\lambda=0$ all terms in \eqref{ELtwistedproductdimbig1} and \eqref{ELtwistedproductdimone} vanish.
\end{proof}

\begin{Example} \rm
A simple example of a critical metric %satisfying assumptions from Remark \ref{remexample},
is the following twisted product (with distributions $\mD_\mu$ defined as tangent to its factors): for a manifold $(M_1,g_1)$ let $M = M_1 \times \RR \times \RR$ and let $g= g_1 + e^{-2f_1(s) } dt^2 + e^{-2f_2(t)} ds^2$, where $f_1,f_2$ are linear functions. Let $\I$ correspond to a statistical connection satisfying conditions 2, 4 and 5 of Corollary~\ref{corstatcritforall} and such that $\< \I_X Y, Z \> =0$ unless $X,Y,Z \in TM_1$. Then $(g,\I)$ is critical for the action \eqref{Eq-Smix-g}, with respect to adapted variations of metric and all variations of contorsion.
\end{Example}

\begin{Corollary}\label{C3-sec4}
Let $g$ be an adapted metric on $(M;\mD_1,\ldots,\mD_k)$ such that all $\mD_\mu$ are one-dimensio\-nal, pairwise mixed totally geodesic and pairwise mixed integrable.
%and $n_\mu=1$ for $\mu\ge1$.
Then a pair $(g,\I)$, where $\I$ is the contorsion tensor of a~statistical connection,
%on $(M,g)$,
is critical for the action \eqref{Eq-Smix-g} with respect to adapted variations of $g$ preserving the volume of $(M,g)$ and all variations of $\I$
if and only if $\I=0$ and there exists a constant $\lambda$ such that for all $1\le\mu\le k$ we have
\begin{equation} \label{ELtwistedproductdimonekRic}
{\rm Ric} |_{\mD_\mu \times \mD_\mu} =
%{\rm S}_{\,\mD_\mu, \mD_\mu^\perp} =
 -\frac{2\lambda}{k-2}\, .
\end{equation}
%Then there are no statistical connections other than the Levi-Civita connection such that a pair $(g,\I)$ is critical for
%\eqref{Eq-Smix-g} with respect to %for adapted variations of $g$ and
%all variations of $\I$ with fixed $g$. Also, there exists a constant $\lambda$ such that for all $1\le\mu\le k$ we have
%\begin{equation} \label{ELtwistedproductdimonekRic}
%{\rm Ric} |_{\mD_\mu \times \mD_\mu} =
%%{\rm S}_{\,\mD_\mu, \mD_\mu^\perp} =
% -\frac{2\lambda}{k-2}\, .
%\end{equation}
%old
%Also, for all $1\le\mu\le k$ we have
%\begin{equation}\label{ELtwistedproductdimonek}
% \sum\nolimits_{\,\nu=1}^k \big(\| P_\mu H_\nu\|^2 - \frac{1}{2}\|H_\nu\|^2\big)
%  -\| {\tilde H}_\mu \|^2
% +\Div ( {\cal H} - H_\mu - {\tilde H}_\mu) + \frac{1}{2}\,\|{\cal H}\|^2 + \lambda = 0
%\end{equation}
%for some constant $\lambda$, or, equivalently,
%\begin{equation} \label{ELtwistedproductdimonekRic}
%{\rm Ric} |_{\mD_\mu \times \mD_\mu} =
%%{\rm S}_{\,\mD_\mu, \mD_\mu^\perp} =
% -\frac{2\lambda}{k-2}\, .
%\end{equation}

\end{Corollary}

\begin{proof}
By claims \ref{item2} and 5 of Corollary~\ref{corstatcritforall}
the only non-zero components of $\I$ may be $\< \I_X X, X\>$ where $X \in \mD_\mu$ is a unit vector.
However, by \ref{item5} of Corollary~\ref{corstatcritforall} we get for a unit vector $X \in \mD_\mu$: $\I_X X = \tr^\top_\mu \I =0$.
It follows that $\I=0$. %, but then it is not a contorsion tensor of a non-trivial statistical connection.
Equation \eqref{ELtwistedproductdimonekRic} can be proved similarly as %\eqref{ELtwistedproductdimone}
\eqref{ELtwistedproductdimoneSmix}: let $j=k$, then from \eqref{ELmixedgeodesicmixedintegrable} we obtain \eqref{ELtwistedproductdimone1} and use in it \eqref{Smixtwproddim1} and $2{\rm S}_{\,\mD_1 \ldots \mD_k} = \sum\nolimits_{\mu=1}^k {\rm S}_{\,\mD_\mu, \mD_\mu^\perp}$.
%
 %Equation \eqref{ELtwistedproductdimonek} can be proved similarly as \eqref{ELtwistedproductdimoneold} and its
%
%. To obtain its equivalent form, we use
%%\[
% ${\rm Ric} |_{\mD_\mu \times \mD_\mu} = {\rm S}_{\,\mD_\mu, \mD_\mu^\perp} = \Div ( {\tilde H}_\mu + H_\mu ) + \| {\tilde H}_\mu \|^2 - \sum\nolimits_{\,\nu =1}^k \| P_\mu H_\nu \|^2$,
%%\]
%and its sum over $\mu=1 , \ldots , k$ in \eqref{ELtwistedproductdimonek}, to get
%%together with \eqref{SbarS} and \eqref{E-Dk-Smix}, to get $%-{\rm Ric}(E_\mu, E_\mu)
%$\frac{1}{2}\sum\nolimits_{\,\nu} {\rm S}_{\mD_\nu,\mD^\bot_\nu} -{\rm S}_{\,\mD_\mu, \mD_\mu^\perp} + \lambda =0$, that yields~\eqref{ELtwistedproductdimonekRic}.
%%
%%$ -{\rm S}_{\,\mD_\mu, \mD_\mu^\perp} + {\rm S}_{\,\mD_1 \ldots \mD_k} + \lambda =0$, which, by $ 2\,{\rm S}_{\,\mD_1,\ldots,\mD_k} = \sum\nolimits_{\,\mu} {\rm S}_{\,\mD_\mu,\mD^\bot_\mu}$, %\eqref{E-Dk-Smix} and \eqref{SbarS},
%%yields \eqref{ELtwistedproductdimonekRic}.
\end{proof}

\begin{Corollary}
Let $(M,g)$ be a warped product of $k>2$ manifolds,
%i.e., $M = M_1 \times M_2 \times \ldots \times M_k$,
%and $g = g_1 + e^{-2 f_2} g_2 + \ldots + e^{-2 f_k} g_k$ where $f_\mu : M_1 \rightarrow \RR$ for $\mu=2, \ldots, k$.
%Let
where $M_1$ is a closed manifold.
Then a pair $(g,\I)$, where $\I$ is the contorsion tensor of a~statistical connection, is critical for the action \eqref{Eq-Smix-g}  with respect to volume-preserving adapted variations of $g$ %preserving the volume of $(M,g)$
and all variations of $\I$, if and only if $(M,g)$ is a metric product
% (i.e., all $f_i$ are constant).
and $\I$ satisfies conditions 2, 4 and 5 of Corollary~\ref{corstatcritforall}.
\end{Corollary}

\begin{proof}
We can assume $g = g_1 + e^{-2 f_2} g_2 + \ldots + e^{-2 f_k} g_k$, where $f_\mu : M_1 \rightarrow \RR$ for $\mu=2, \ldots, k$.
For warped products we have $H_1 = 0$ and $H_\mu = \nabla f_\mu$, ${\tilde H}_\mu=0$ for all $\mu>1$.
If $\dim \mD_\mu=1$, the Euler-Lagrange equation \eqref{ELmixedgeodesicmixedintegrable} for $\mD_\mu$ is \eqref{ELtwistedproductdimoneold}, which becomes
\begin{equation}\label{ELtwistedproductdimoneoldtwp}
\sum\nolimits_{\nu=2}^k \Delta f_\nu - \Delta f_\mu -\frac{1}{2} \sum\nolimits_{\nu=2}^k \| \nabla f_\nu \|^2 + \frac{1}{2} \|  \sum\nolimits_{\nu=2}^k \nabla f_\nu \|^2 + \lambda =0,
\end{equation}
and if $\dim \mD_\mu>1$, then, by \eqref{ELtwistedproductdimbigold}, the Euler-Lagrange equation \eqref{ELmixedgeodesicmixedintegrable} on $\mD_\mu \times \mD_\mu$ is
\begin{equation}\label{ELtwistedproductdimbigoldtwp}
\sum\nolimits_{\nu=2}^k (P_\mu \nabla f_\nu)^\flat \otimes (P_\mu \nabla f_\nu)^\flat + \big( \sum\nolimits_{\nu=2}^k \Delta f_\nu -\frac{1}{2} \sum\nolimits_{\nu=2}^k \| \nabla f_\nu \|^2 + \frac{1}{2} \|  \sum\nolimits_{\nu=2}^k \nabla f_\nu \|^2 + \lambda \big)g_\mu = 0 .
\end{equation}

%From Corollary \ref{corstatcritforall} we find that if $n_\mu >1$ for some $\mu>1$, then $H_\mu=0$, and hence $f_\mu$ is constant.

(i) Let $n_\mu = 1$ for all $\mu>1$.
Taking differences of \eqref{ELtwistedproductdimoneoldtwp} for $\mD_2$ and $\mD_\mu$ for any $\mu>2$, we get
%\[
$\Delta f_2 - \Delta f_\mu = 0$. %\ (\mu>2)$.
%\]
Since $M_1$ is closed, it follows that $f_2 - f_\mu$ is constant and for all $\mu>2$
\begin{equation} \label{gradfi}
\nabla f_2 = \nabla f_\mu . %\quad (\mu>2)\,.
\end{equation}
Hence, ${\cal H} = (k-1)\nabla f_2$ and $\sum\nolimits_{\mu } \|H_\mu\|^2=(k-1)\|\nabla f_2\|^2$, and \eqref{ELtwistedproductdimoneoldtwp} for %$\mD_2$
$\mu=2$ yields the equality
%\[
$(k-2)\, \Delta f_2 + \frac12\,(k-1)(k-2)\, \| \nabla f_2\|^2 + \lambda =0$.
%\]
Thus,
%\[
$\frac{1}{a}\, e^{-a f_2} \Delta\, e^{af_2} = \frac{\lambda}{k-2}$,
%\]
where $a = \frac12\,(k-1)$. Since $e^{af_2}>0$, we obtain $\lambda=0$ and $f_2={\rm const}$.
From \eqref{gradfi} we conclude that all $f_i$ are constant and $(M,g)$ is the metric product.

(ii) Suppose now that $n_\mu >1$ for some $\mu>1$. From Corollary \ref{corstatcritforall} we find that %if $n_\mu >1$ for some $\mu>1$, then
$H_\mu=0=\nabla f_\mu$, and hence $f_\mu$ is constant. Let $\xi\in[2, k]$.
If $n_\xi >1$, then, by the same argument as above, $f_\xi$ is constant. %then from Corollary~\ref{corstatcritforall} we get $\nabla f_\nu=0$, hence $f_\nu$ is constant.
If $n_\xi =1$, then taking the difference of \eqref{ELtwistedproductdimoneoldtwp} for $\mD_\xi$ and the trace of \eqref{ELtwistedproductdimbigoldtwp} divided by $n_\mu$,
we find
%\[
$-\Delta f_\xi = 0$,
%\]
and since $M_1$ is compact, we get $f_\xi={\rm const}$. Thus, all $f_i$ are constant and $(M,g)$ is the metric product.
\end{proof}

\subsection{Variations of $\I$ corresponding to statistical connections}
\label{sec:metric-stat}

%Results from the previous section show that, similarly as in case $k=2$ considered in \cite{rz-3}, it is difficult to find non-trivial examples of metrics and statistical connections critical for action \eqref{Eq-Smix-g} among all adapted metrics and connections. <-not quite true anymore
In this section we still consider adapted variations of metric, but restrict variations of $\I$ to tensors corresponding to statistical connections.
Using Theorem~\ref{statisticalcritSmixI} %(instead of Theorem \ref{corI})
we examine $\delta_g{\bar Q}_\mu$ in Proposition~\ref{C-vr-terms}.

\begin{Lemma}\label{P-06}
Let $g$ be an adapted metric on $(M;\mD_1,\ldots,\mD_k)$,
%Let a contorsion tensor $\I$ of a statistical connection be critical for the action \eqref{Eq-Smix-g} with fixed adapted metric $g$ on
%$(M;\mD_1,\ldots,\mD_k)$
and let the contorsion tensor $\I$ of a statistical connection be critical for
the action \eqref{Eq-Smix-g}
%with fixed $g$ on $(M;\mD_1,\ldots,\mD_k)$
%\eqref{actionI}
restricted to contorsion tensors of statistical connections on $(M,g)$.
Then for the $\mD_\mu$-variation of metric we have
\begin{eqnarray*}
%\label{deltabarQamongstatistical}
&& 2\delta_g{\bar Q}_\mu = \sum\nolimits_{\nu\ne\mu} \frac{n_\nu}{2}\,(P_\mu \tr^\perp_\nu \I )^\flat \otimes (P_\mu \tr^\perp_\nu \I )^\flat
 - \frac{1}{2}\, \| P_\mu^\perp \tr^\perp_\mu \I \|^2 g_\mu \nonumber \\
&&\hskip-6mm -\,3\sum\nolimits_{\nu\ne\mu} \!{\rm Sym}( (P_\mu H_\nu )^\flat{\otimes} (P_\mu \tr^\perp_\nu \I )^\flat)
 {+} 2\< h_\mu, P_\mu^\perp \tr^\perp_\mu \I \>
% \\&&
 {+} \sum\nolimits_{\nu\ne\mu} \!{\rm Sym}( (P_\mu \tr^\top_\nu \I )^\flat {\otimes} (P_\mu H_\nu )^\flat ) \nonumber \\
&&\hskip-6mm +\,\frac{1}{2}\, \big\< \sum\nolimits_{\nu\ne\mu} (P_\mu^\perp H_\nu -{\tilde H}_\nu), \tr_\mu^\perp \I \big\>g_\mu
% - \frac{1}{2}\, \big\< \sum\nolimits_{\nu\ne\mu} {\tilde H}_\nu, \tr_\mu^\perp \I \big\> g_\mu
 +2\,\I_{{\tilde H}_\mu}^\flat
% \\ &&
 - \frac{1}{2}\, \< H_\mu, \tr_\mu^\perp \I \>g_\mu - {\rm Sym}( (P_\mu \tr^\top_\mu \I)^\flat \otimes {\tilde H}_\mu^\flat ) \,.
\end{eqnarray*}
\end{Lemma}

\begin{proof}
It follows from \eqref{newELIstat} that $\< \I_X Y, Z \>=0$ if each of $X,Y,Z$ belongs to a different distribution.
The computation of all formulas \cite[Eqs.~(67)--(75)]{rz-3} gives the same results as Lemma~\ref{propdeltaQforstatistical}, %, although we use in them \eqref{newELIstat} instead of the assumption that distributions are pairwise mixed integrable and pairwise mixed totally geodesic;
but now we obtain $\dt \< \Theta, {T}^\sharp_\xi \> =0$ from $\sum_{a,b} \< T(E_{\nu,a}, E_{\nu,b}), {\cal E}_{\nu,i} \> \< {\cal E}_{\nu,j}, \I_{\nu,a} E_{\nu,b} \>= 0$ by $\I=\I^\wedge$, and we obtain $\dt \< \Theta, {\tilde T}^\sharp \> =0$ using \eqref{ELSmixIstat2}, which yields sums of symmetric and antisymmetric in $E_{\nu, a}, E_{\nu,b}$ terms:
\begin{eqnarray*}
 \dt\sum\nolimits_{\xi} \<\Theta, {\tilde T}^\sharp_\xi \>
 &=& \sum\nolimits_{\nu\ne\mu} \Big(
 \sum {B}(E_{\mu,i}, E_{\mu,j}) \<2 {\tilde T}_\nu ({\cal E}_{\nu,k}, E_{\mu,j}), E_{\nu, a}\> \<{\cal E}_{\nu,k}, \I_{\nu,a} E_{\mu,i}\> \\
&& - \sum {B}(E_{\mu,i}, E_{\mu,j} ) \< 2 {\tilde T}_\nu ( E_{\mu,i}, {\cal E}_{\nu,k} ), E_{\nu, a} \> \< E_{\mu,j}, \I_{\nu,a} {\cal E}_{\nu,k} \> \\
&& +\sum {B}(E_{\mu,i}, E_{\mu,j})\< 2{\tilde T}_\nu ({\cal E}_{\nu,k}, E_{\mu,j}), E_{\nu, a}\>\< E_{\nu,a}, \I_{{\cal E}_{\nu,k}} E_{\mu,i} \>\Big) \\
%dual
&& + \sum {B}(E_{\mu,i}, E_{\mu,j} ) \< 2 T_\mu ( E_{\mu,k}, E_{\mu,j} ), {\cal E}_{\mu, a} \> \< E_{\mu,k}, \I_{{\cal E}_{\mu,a} } E_{\mu,i} \> \\
&& - \sum {B}(E_{\mu,i}, E_{\mu,j} ) \< 2 T_\mu ( E_{\mu,i}, E_{\mu,k} ), {\cal E}_{\mu, a} \> \< E_{\mu,j}, \I_{{\cal E}_{\mu,a} } E_{\mu,k} \> \\
&& + \sum {B}(E_{\mu,i}, E_{\mu,j} ) \< 2 T_\mu ( E_{\mu,k}, E_{\mu,j} ), {\cal E}_{\mu, a} \>\< {\cal E}_{\mu,a}, \I_{{\mu,k} } E_{\mu,i} \> \\
&=& 6 \sum\nolimits_{\nu\ne\mu} \sum {B}(E_{\mu,i}, E_{\mu,j} ) \< T_\mu ( E_{\mu,i}, E_{\mu,j} ), E_{\nu, a} \> \< \tr^\perp_\mu \I, E_{\nu,a} \> =0\,.
\end{eqnarray*}
Using \eqref{ELSmixIstat2} in equations from the proof of Lemma~\ref{propdeltaQforstatistical}, we also simplify other formulas:
\begin{eqnarray*}
 &&\hskip-4mm
 \dt \sum\nolimits_{\xi} \< \I^*, \I^\wedge \>_{| V_\xi}
% = -2\sum\nolimits_{\nu\ne\mu}\sum{B}(E_{\mu, i}, E_{\mu, j})\<\I_{{\mu, i}} E_{\nu,a}, E_{\nu,b}\> \<E_{\nu, b}, \I_{\nu, a} E_{\mu, j} \> \\
%&&-2\sum\nolimits_{\nu\ne\mu}\sum{B}(E_{\mu, i}, E_{\mu, j})\<\I_{{\mu, i}}E_{\nu, a}, E_{\mu, b}\>\<E_{\mu, b}, \I_{\nu, a} E_{\mu, j}\> \\
%&=& -\frac{1}{2}\sum\nolimits_{\nu\ne\mu}\sum {B}(E_{\mu, i}, E_{\mu, j} ) \< E_{\nu, a}, E_{\nu, b} \> \< {E_{\mu, i}}, \tr^\perp_\nu \I \> \<E_{\nu, b}, %E_{\nu, a} \> \< E_{\mu, j}, \tr^\perp_\nu \I \> \\
%&& -\frac{1}{2} \sum\nolimits_{\nu\ne\mu} \sum {B}(E_{\mu, i}, E_{\mu, j} )
% \< E_{\mu, i}, E_{\mu, b} \> \< E_{\nu, a}, \tr^\perp_\mu \I \>
% \< E_{\mu, b}, E_{\mu, j} \> \< E_{\nu, a}, \tr^\perp_\mu \I \> \\
%&=& -\frac{n_\nu}{2} \sum\nolimits_{\nu\ne\mu} \sum {B}(E_{\mu, i}, E_{\mu, j} ) \< {E_{\mu, i}}, \tr^\perp_\nu \I \> \< E_{\mu, j}, \tr^\perp_\nu \I \> \\
%&& -\frac{1}{2} \sum {B}(E_{\mu, i}, E_{\mu, j} )
% \< P_\mu^\perp \tr^\perp_\mu \I, P_\mu^\perp \tr^\perp_\mu \I \> \< E_{\mu, i}, E_{\mu, j} \> \\
%&&
= \< {B}_\mu, \sum\nolimits_{\nu\ne\mu} \big(-\frac{n_\nu}{2}\, (P_\mu\tr^\perp_\nu\I)^\flat\otimes(P_\mu\tr^\perp_\nu\I)^\flat\big)
- \frac{1}{2}\, \| P_\mu^\perp \tr^\perp_\mu \I \|^2 g_\mu \>\,,\\
%\end{eqnarray*}
%%%%
%\begin{eqnarray*}
&&\hskip-4mm \dt \sum\nolimits_{\xi} \< \Theta, A_\xi \>
% = -4 \sum\nolimits_{\nu\ne\mu} \sum {B}(E_{\mu, i}, E_{\mu, j} ) \< h_\nu (E_{\nu,a}, E_{\nu,b}), E_{\mu,i} \>
%\< E_{\mu,j}, \I_{\nu,a} E_{\nu,b} \> \\
%&=& -2 \sum\nolimits_{\nu\ne\mu} \sum {B}(E_{\mu, i}, E_{\mu, j} ) \< h_\nu (E_{\nu,a}, E_{\nu,b}), E_{\mu,i} \>
%\< E_{\nu,a} E_{\nu,b} \> \< E_{\mu,j}, \tr^\perp_\nu \I \> \\
%&=& -2 \sum\nolimits_{\nu\ne\mu} \sum {B}(E_{\mu, i}, E_{\mu, j} ) \< H_\nu, E_{\mu,i} \>
%\< E_{\mu,j}, \tr^\perp_\nu \I \> \\
%&&
= -2 \< {B}_\mu, \sum\nolimits_{\nu\ne\mu} {\rm Sym}( (P_\mu H_\nu )^\flat \otimes (P_\mu \tr^\perp_\nu \I )^\flat ) \>\,,\\
%\end{eqnarray*}
%%%%
%\begin{eqnarray*}
 &&\hskip-4mm \dt \sum\nolimits_{\xi} \< \Theta, {\tilde A}_\xi \>
% = 4 \sum\nolimits_{\nu\ne\mu}\sum {B}(E_{\mu, i}, E_{\mu, j} ) \< h_\mu ( E_{\mu,k}, E_{\mu,j} ), E_{\nu,a} \> \< E_{\nu,a}, \I_{{\mu,k} } E_{\mu,i} \> \\
%& = & 2 \sum\nolimits_{\nu\ne\mu} \sum {B}(E_{\mu, i}, E_{\mu, j} ) \< h_\mu ( E_{\mu,k}, E_{\mu,j} ), E_{\nu,a} \>
%\< E_{\mu,k}, E_{\mu,i} \> \< E_{\nu,a}, \tr^\perp_\mu \I \> \\
%&=& 2 \sum\nolimits_{\nu\ne\mu} \sum {B}(E_{\mu, i}, E_{\mu, j} ) \< h_\mu ( E_{\mu,i}, E_{\mu,j} ), E_{\nu,a} \> \< E_{\nu,a}, \tr^\perp_\mu \I \> \\
%&&
= 2 \< {B}_\mu, \< h_\mu, P_\mu^\perp \tr^\perp_\mu \I \> \>\,.
\end{eqnarray*}
From \cite[Eq.~(72)]{rz-3}, for statistical connections we get $\dt \< \tr^\top_\nu \I, \tr^\perp_\nu \I^* \>=0$.
From \eqref{ELSmixIstat1} we~get
\begin{eqnarray*}
&& \dt \sum\nolimits_{\xi} \< \tr^\top_\xi \I^*, \tr^\perp_\xi \I \>
%&=& - 2 \sum\nolimits_{\nu\ne\mu} \sum {B}(E_{\mu, i}, E_{\mu, j} ) \< \I_{\mu, j} E_{\mu,i}, \tr^\top_\nu \I \> \\
%&=& - \sum\nolimits_{\nu\ne\mu} \sum {B}(E_{\mu, i}, E_{\mu, j} ) \< E_{\mu, j}, E_{\mu,i} \> \< P_\mu^\perp \tr^\top_\nu \I, \tr^\perp_\mu \I \> \\
%&& - 2 \sum\nolimits_{\nu\ne\mu} \sum {B}(E_{\mu, i}, E_{\mu, j} ) \< \I_{\mu, j} E_{\mu,i}, P_\mu \tr^\top_\nu \I \> \\
%&=& - \sum {B}(E_{\mu, i}, E_{\mu, j} ) \< E_{\mu, j}, E_{\mu,i} \> \< \tr^\perp_\mu \I, P_\mu^\perp \tr^\perp_\mu \I\> \\
%&& - 2 \sum {B}(E_{\mu, i}, E_{\mu, j} ) \< \I_{\mu, j} E_{\mu,i}, P_\mu \tr^\perp_\mu \I \>\\
= \< {B}_\mu, - \| P_\mu^\perp \tr^\perp_\mu \I \|^2 g_\mu \>\,,\\
%\end{eqnarray*}
%Using \eqref{ELSmixIstat1}, we get
%\begin{eqnarray*}
 && \dt \sum\nolimits_{\xi} \< \tr^\top_\xi ( \I^* - \I ), {\tilde H}_\xi - H_\xi \>
% =\sum {B}(E_{\mu, i}, E_{\mu, j} ) \big(\< \tr^\perp_\mu \I, E_{\mu,i} \> \< E_{\mu,j}, {\tilde H}_\mu \> \\
%dual
%&& +\sum\nolimits_{\nu\ne\mu}\< \tr^\top_\nu \I, E_{\mu,j} \> \< E_{\mu,i}, H_\nu \> \big) \\
%&=& \sum\nolimits_{\nu\ne\mu} \sum {B}(E_{\mu, i}, E_{\mu, j} ) \< \tr^\top_\nu \I, E_{\mu,j} \> \< E_{\mu,i}, H_\nu \> \\
%&=& \< B, \sum\nolimits_{\nu\ne\mu} {\rm Sym}( (P_\mu \tr^\top_\nu \I )^\flat \otimes (P_\mu H_\nu )^\flat ) \>,\\
%\end{eqnarray*}
%%%%
%\begin{eqnarray*}
% &&\dt \sum\nolimits_{\xi} \< \tr^\perp_\xi ( \I^* - \I ), {\tilde H}_\xi - H_\xi \> \\
% &=& \sum {B}(E_{\mu, i}, E_{\mu, j} )
% \big(\sum\nolimits_{\nu\ne\mu}( \< \I_{\mu,j} E_{\mu,i}, {\tilde H}_\nu - H_\nu \> + \< \tr^\perp_\nu \I, E_{\mu,i} \> \< H_\nu, E_{\mu,j} \> ) \\
%dual
% && +\, \< \I_{\mu,j} E_{\mu,i}, H_\mu - {\tilde H}_\mu \> + \< \tr^\top_\mu \I, E_{\mu,i} \> \< {\tilde H}_\mu, E_{\mu,j} \> \big) \\
% &&
 = \big\< {B}, \I_{ H_\mu - {\tilde H}_\mu }^\flat
+ {\rm Sym}( (P_\mu \tr^\top_\mu \I )^\flat \otimes {\tilde H}_\mu^\flat ) \\
&& + \sum\nolimits_{\nu\ne\mu} ( \< \I_{{\tilde H}_\nu - H_\nu }^\flat  +\,{\rm Sym}( (P_\mu \tr^\perp_\nu \I )^\flat \otimes
(P_\mu H_\nu)^\flat ) )
 \big\>.
\end{eqnarray*}
%
%old
% = \big\< {B}, \sum\nolimits_{\nu\ne\mu} ( \< \I_{{\tilde H}_\nu - H_\nu }^\flat \\
%&& +\,{\rm Sym}( (P_\mu \tr^\perp_\nu \I )^\flat \otimes
%(P_\mu H_\nu)^\flat ) ) + \I_{ H_\mu - {\tilde H}_\mu }^\flat
%+ {\rm Sym}( (P_\mu \tr^\top_\mu \I )^\flat \otimes {\tilde H}_\mu^\flat ) \big\>.
%\end{eqnarray*}
%
Using \eqref{ELSmixIstat2}, we get on $\mD_\mu \times \mD_\mu$:
\begin{eqnarray*}
 && \I_{H_\mu}^\flat = \frac{1}{2}\, \< H_\mu, \tr_\mu^\perp \I \>g_\mu\,, \quad
 \sum\nolimits_{\nu\ne\mu} \I_{H_\nu}^\flat
% = \sum\nolimits_{\nu\ne\mu} ( \I_{P_\mu H_\nu}^\flat + \frac{1}{2}\, \< P_\mu^\perp H_\nu, \tr_\mu^\perp \I \>g_\mu )
 = \I_{{\tilde H}_\mu}^\flat + \frac{1}{2}\, \< \sum\nolimits_{\nu\ne\mu} P_\mu^\perp H_\nu, \tr_\mu^\perp \I \>g_\mu\,, \\
 && \sum\nolimits_{\nu\ne\mu} \I_{{\tilde H}_\nu}^\flat = \frac{1}{2}\, \< \sum\nolimits_{\nu\ne\mu} {{\tilde H}_\nu}, \tr_\mu^\perp \I \>g_\mu\,.
\end{eqnarray*}
Thus,
%\begin{eqnarray*}
% &&\dt \sum\nolimits_{\xi} \< \tr^\perp_\xi ( \I^* - \I ), {\tilde H}_\xi - H_\xi \> \\
% && = \big\< {B}, \frac{1}{2}\, \< \sum\nolimits_{\nu\ne\mu} {{\tilde H}_\nu}, \tr_\mu^\perp \I \>g_\mu
% - \I_{{\tilde H}_\mu}^\flat - \frac{1}{2}\, \< \sum\nolimits_{\nu\ne\mu} P_\mu^\perp H_\nu, \tr_\mu^\perp \I \>g_\mu \\
% && + \sum\nolimits_{\nu\ne\mu} {\rm Sym}( (P_\mu \tr^\perp_\nu \I )^\flat \otimes
% (P_\mu H_\nu)^\flat ) + \frac{1}{2}\, \< H_\mu, \tr_\mu^\perp \I \>g_\mu \\
% && - \I_{ {\tilde H}_\mu }^\flat + {\rm Sym}( (P_\mu \tr^\top_\mu \I )^\flat \otimes {\tilde H}_\mu^\flat ) \big\>,
%\end{eqnarray*}
%which can be simplified to:
\begin{eqnarray*}
 && \dt \sum\nolimits_{\xi} \< \tr^\perp_\xi ( \I^* - \I ), {\tilde H}_\xi - H_\xi \>
 = \big\< {B}, \frac{1}{2}\, \< \sum\nolimits_{\nu\ne\mu} {{\tilde H}_\nu}, \tr_\mu^\perp \I \>g_\mu - 2\I_{{\tilde H}_\mu}^\flat \\
 && -\,\frac{1}{2}\, \< \sum\nolimits_{\nu\ne\mu} P_\mu^\perp H_\nu, \tr_\mu^\perp \I \>g_\mu
  +\sum\nolimits_{\nu\ne\mu} {\rm Sym}( (P_\mu \tr^\perp_\nu \I )^\flat \otimes
 (P_\mu H_\nu)^\flat ) \\
 && +\,\frac{1}{2}\, \< H_\mu, \tr_\mu^\perp \I \>g_\mu
  +{\rm Sym}( (P_\mu \tr^\top_\mu \I )^\flat \otimes {\tilde H}_\mu^\flat ) \big\>\,.
\end{eqnarray*}
%Finally,
By \eqref{E-barQ} and \eqref{E-dt-barQ}, from the above computations we get the required formula for $\delta_g{\bar Q}_\mu$.
\end{proof}

%From the above proposition and \eqref{ElmixDD-b} we obtain the following
The following result combines Theorem~\ref{T-main01} (variations of $g$ with $\I$ fixed) and Theorem~\ref{statisticalcritSmixI} (variations of $\I$ with $g$ fixed) for variations among statistical connections.

\begin{Theorem}\label{T-06}
A pair $(g,\I)$, where $g$ is an adapted metric and $\I$ is the contorsion tensor of a statistical connection on $(M,\mD_1,\ldots,\mD_k)$,
%for $k>2$,
is critical for the action \eqref{Eq-Smix-g} with respect to volume-preserving adapted variations of $g$ and variations of \,$\I$ corresponding to
statistical connections if and only if it satisfies the Euler-Lagrange equations {\rm (\ref{ELSmixIstat1}-c)} %see \eqref{E-delta-I-J},
and \eqref{ElmixDDvp-b}, which has the
%following~
form:
\begin{eqnarray}\label{ELamongstatistical}
&& \sum\nolimits_{\nu\ne\mu} \frac{n_\nu}{2}\, (P_\mu \tr^\perp_\nu \I )^\flat \otimes (P_\mu \tr^\perp_\nu \I )^\flat
- \frac{1}{2}\, \| P_\mu^\perp \tr^\perp_\mu \I \|^2 g_\mu \nonumber \\
&& -\,3 \sum\nolimits_{\nu\ne\mu} {\rm Sym}( (P_\mu H_\nu )^\flat \otimes (P_\mu \tr^\perp_\nu \I )^\flat ) \nonumber \\
&& +\,2\, \< h_\mu, P_\mu^\perp \tr^\perp_\mu \I \>
+ \sum\nolimits_{\nu\ne\mu} {\rm Sym}( (P_\mu \tr^\top_\nu \I )^\flat \otimes (P_\mu H_\nu )^\flat ) \nonumber \\
&& +\,\frac{1}{2}\,\<\sum\nolimits_{\nu\ne\mu} (P_\mu^\perp H_\nu - {\tilde H}_\nu), \tr_\mu^\perp \I \>g_\mu + 2\,\I_{{\tilde H}_\mu}^\flat
% \nonumber \\ &&
 {-} \frac{1}{2}\, \< H_\mu, \tr_\mu^\perp \I \>g_\mu - {\rm Sym}( (P_\mu \tr^\top_\mu \I )^\flat \otimes {\tilde H}_\mu^\flat ) \nonumber \\
 %%%%
&& +\, 2 \big(
 -\Div{h}_{\mu} -{\cal K}_{\mu}^\flat -\tilde{H}_{\mu}^\flat\otimes\tilde{H}_{\mu}^\flat
 +\frac12\Upsilon_{\tilde h_{\mu},\tilde h_{\mu}} +\frac12\Upsilon_{\tilde T_{\mu}, \tilde T_{\mu}}
 +\,2\,{\cal T}_{\mu}^\flat + (\Div H_{\mu})\,g_\mu \nonumber \\
 && +\sum\nolimits_{\,\nu\ne {\mu}}\big(-\Div\tilde{h}_\nu|_{\mD_{\mu}} - (P_{\mu}\tilde{\cal K}_\nu)^\flat -(P_{\mu}{H}_\nu)^\flat\otimes(P_{\mu}{H}_\nu)^\flat \nonumber \\
 && +\,\frac12\Upsilon_{P_{\mu}h_\nu,P_{\mu}h_\nu} +\frac12\Upsilon_{P_{\mu}T_\nu,P_{\mu}T_\nu} +2\,(P_{\mu}\tilde {\cal T}_\nu)^\flat
 +(\Div \tilde H_\nu)\,g_\mu\big)\big) \nonumber \\
%%%%
 && +\,2\big(\overline{\rm S}_{\,\mD_1,\ldots,\mD_k} -\frac12\Div\sum\nolimits_{\,\nu}(H_\nu + \tilde{H}_\nu) +\lambda\big)\,g_\mu=0,\quad
\mu=1, \ldots, k\,.
\end{eqnarray}
%and the Euler-Lagrange equations \eqref{ElmixDDvp-b} %see \eqref{E-delta-g-J},
%are satisfied,
%where $\delta Q_\mu$ is given in Proposition~\ref{C-vr-terms} and $\delta_g\bar Q_\mu$ is given in Lemma~\ref{P-06}.
\end{Theorem}

\begin{proof}
Equation \eqref{ELamongstatistical} is the Euler-Lagrange equation \eqref{ElmixDDvp-b}, with $\delta Q_\mu$ given in Proposition~\ref{C-vr-terms} and $\delta_g\bar Q_\mu$ given in Lemma~\ref{P-06}.
\end{proof}

Using Theorem~\ref{T-06}, we will show how to obtain examples of critical pairs $(g,\I)$
of the action \eqref{Eq-Smix-g} with respect to adapted variations of metric and variations of \,$\I$ corresponding to statistical connections,
from critical metrics of this action with fixed Levi-Civita connection.
%\eqref{Eq-Smix-g0}
%for adapted variations of metric.
%If all $\mD_\mu$ are one-dimensional, then, according to Remark~\ref{remarkEinstein}, we can use for this purpose any Einstein metric.

\begin{Lemma}\label{L-barS-S}
Let $\I$ be the contorsion tensor of a statistical connection on $(M,\mD_1,\ldots,\mD_k)$ with an adapted metric $g$
%for $k>2$,
such that $\tr^\top_\mu \I=0$ for all $\mu \in \{1, \ldots, k\}$. If \,$\I$ is critical for the action \eqref{Eq-Smix-g} with fixed $g$, with respect to variations of \,$\I$ corresponding to statistical connections, then $\overline{\rm S}_{\,\mD_1,\ldots,\mD_k} = {\rm S}_{\,\mD_1,\ldots,\mD_k}$.
\end{Lemma}

\begin{proof}
From the assumption $\tr^\top_\mu \I=0$ for all $\mu \in \{1, \ldots, k\}$ we obtain
%\begin{equation*}
 $\sum\nolimits_{\,\nu} \<\tr_{\,\nu}^\bot\I,\,\tr_{\,\nu}^\top\I\> =0 $\,.
%\end{equation*}
From (\ref{ELSmixIstat2},c) we obtain for all $\mu \in \{1, \ldots, k\}$:
\begin{eqnarray*}
\<\I,\,\I\>_{\,|\,V_\mu} \eq \sum \< \I_{E_{\mu,a}} {\cal E}_{\mu, b}, \I_{{\cal E}_{\mu, b}} E_{\mu, a} \>
= \sum\nolimits_{\nu \ne \mu} \sum \< \I_{E_{\mu,a}} E_{\nu, b}, \I_{E_{\nu, b}} E_{\mu, a} \> \\
\eq  \sum\nolimits_{\nu \ne \mu} \sum \< \I_{E_{\mu,a}} {E}_{\nu, b}, E_{\nu, c} \> \< E_{\nu, c}, \I_{{E}_{\nu, b}} E_{\mu, a} \> \\
&&+ \sum\nolimits_{\nu \ne \mu} \sum \< \I_{E_{\mu,a}} {E}_{\nu, b}, E_{\mu, c} \> \< E_{\mu, c}, \I_{{E}_{\nu, b}} E_{\mu, a} \> \\
\eq \sum\nolimits_{\nu \ne \mu} \sum \< \tr^\perp_\nu \I, {E_{\mu,a}} \> \< {E}_{\nu, b}, E_{\nu, c} \> \< \tr^\perp_\nu \I, {E_{\mu,a}} \> \< {E}_{\nu, b}, E_{\nu, c} \> \\
&&+ \sum\nolimits_{\nu \ne \mu} \sum \< \tr^\perp_\mu \I, {E_{\nu,b}} \> \< {E}_{\mu,a}, E_{\mu, c} \> \< \tr^\perp_\mu \I, {E_{\nu,b}} \> \< {E}_{\mu,a}, E_{\mu, c} \> =0\,.
\end{eqnarray*}
Hence, from \eqref{eqvarstat}, we get $\overline{\rm S}_{\,\mD_1,\ldots,\mD_k} = {\rm S}_{\,\mD_1,\ldots,\mD_k}$.
\end{proof}

Using Theorem~\ref{T-06} and Lemma~\ref{L-barS-S}, we get the following

\begin{Corollary}
Let $g$ be a critical adapted metric for the action \eqref{Eq-Smix-g0} with respect to adapted variations on $(M;\mD_1,\ldots,\mD_k)$.
Suppose that
there exists $\mu\in\{1,\ldots,k\}$ such that
either $n_\mu\ge3$,
or $n_\mu=2$ and ${\tilde H}_\mu=0$.
%for some $\mu\in\{1,\ldots,k\}$.
%either $\dim M>k$ or $\dim M=k$ (i.e., all $\mD_\mu$ are one-dimensional) and at least one $\tilde H_\mu$ is zero.
Then there exists a contorsion tensor $\I\ne0$ of a statistical connection on $(M,g)$ such that the pair $(g,\I)$ is critical
for the action \eqref{Eq-Smix-g} with respect to adapted variations of $g$ and variations of \,$\I$ corresponding to statistical connections.
\end{Corollary}

\begin{proof}
The Euler-Lagrange equations %\eqref{E-delta-g-J}
%$\delta_g \bar J_\mD = \lambda\,g$
%in Theorem~\ref{T-06}
\eqref{ELamongstatistical} do not contain derivatives of $\I$ and thus,
%Thus, \eqref{ELamongstatistical}
with a given metric $g$, are algebraic equations for $\I$.
By Lemma~\ref{L-barS-S}, we get $\overline{\rm S}_{\,\mD_1,\ldots,\mD_k} = {\rm S}_{\,\mD_1,\ldots,\mD_k}$,
thus the expression in the last four lines of \eqref{ELamongstatistical} does not depend on $\I$.
Moreover, by \eqref{ElmixDDvp} and since $g$ is critical, this expression vanishes.
We can restrict ourselves to $\I$ such that $\tr_\mu^\bot\I=0$ for all $\mu$. Then \eqref{ELSmixIstat1} is valid, and similarly as in the proof of claim \ref{item5} in Corollary~\ref{corstatcritforall}, we get $\tr_\mu^\top\I=0$ for all $\mu$.
Thus, \eqref{ELamongstatistical} reduces to the linear (with respect to $\I$) system
\begin{equation}\label{ELamongstatistical-A}
 2\,\I_{{\tilde H}_\mu}^\flat = 0,\quad \mu=1, \ldots, k\,,
\end{equation}
which must be compatible with (\ref{ELSmixIstat2},c).
From \eqref{newELIstat} we conclude that all components of $\I$ with indices from three different distributions vanish.
% two of indices of $\I$ should be in the same distribution.
Then, by \eqref{ELSmixIstat2}, assumption $\tr_\mu^\bot\I=0$ for all $\mu$, and since $\I$ corresponds to a statistical connection, the only nonzero components of such $\I$ appear when its three indices belong to the same distribution.
Let $n_\mu\ge3$ for some $\mu$, then the number $\frac12\,n_\mu(n_\mu+1)$ of independent equations in \eqref{ELamongstatistical-A} is smaller than
the number
$\frac{n_\mu (n_\mu -1)(n_\mu -2)}{6} + n_\mu (n_\mu -1) + n_\mu$
of independent components of $\I$ along $\mD_\mu$
%, which~is
%\[
%\frac{n_\mu (n_\mu -1)(n_\mu -2)}{6} + n_\mu (n_\mu -1) + n_\mu
%\]
(components along $\mD_\nu$ for all $\nu\ne\mu$ can be chosen to be zero).
For example, if $n_\mu=3$, then $\frac12\,n_\mu(n_\mu+1)=6$ and (given its symmetries $\I^*=\I=\I^\wedge$) $\I$ with zero trace on $\mD_\mu$ has $10-1=9$ independent components along $\mD_\mu$ and \eqref{ELamongstatistical-A} gives 6 equations.
If $n_\mu=2$ and ${\tilde H}_\mu=0$ for some $\mu$, then there are no independent equations in \eqref{ELamongstatistical-A} for such $\mu$,
but there are $4-1=3$ independent components of $\I$ along $\mD_\mu$.
\end{proof}

%Example:no non-trivial critical statistical connection on S3

\begin{Corollary}
Let $\mD_\mu\ (\mu=1,2,3)$ be distributions determined by unit orthonormal vector fields $\xi_1, \xi_2, \xi_3$ on a $3$-dimensional manifold. %sphere $(S^3,g)$ with the metric $g$ induced from Euclidean space $\mathbb{R}^4$. <-no need
Then $\I =0$ is the only contorsion of a statistical connection critical for the action \eqref{Eq-Smix-g} with fixed~$g$, with respect to variations of \,$\I$ corresponding to statistical connections.
\end{Corollary}

\begin{proof}
From \eqref{ELSmixIstat2} we obtain
%\[
$\< \I_{\xi_2} \xi_2, \xi_1 \> = \frac{1}{2}\,\< \xi_1, \I_{\xi_1} \xi_1 + \I_{\xi_3 } \xi_3 \>$.
%\]
By this and \eqref{ELSmixIstat2}, we~get
\begin{eqnarray*}
\< \I_{\xi_3} \xi_3, \xi_1 \>
%&=& \frac{1}{2}\,\< \xi_1, \I_{\xi_1} \xi_1 + \I_{\xi_2 } \xi_2 \> \\
&=& \frac{1}{2}\,\< \xi_1, \I_{\xi_1} \xi_1 \> + \frac{1}{4}\,\< \I_{\xi_1 } \xi_1, \xi_1 \> + \frac{1}{4}\,\< \I_{\xi_3 } \xi_3, \xi_1 \>\, .
\end{eqnarray*}
Hence, $\< \I_{\xi_3} \xi_3, \xi_1 \> = \< \I_{\xi_1 } \xi_1, \xi_1 \>$.
Similarly, $\<\I_{\xi_2} \xi_2, \xi_1\> = \<\I_{\xi_1}\xi_1, \xi_1\>$. Hence, from \eqref{ELSmixIstat1} we obtain
%\[
 $0 = \< \xi_1, \tr_1^\perp \I \> = 2 \< \xi_1, \I_{\xi_1} \xi_1 \>$.
%\]
Similarly, $\< \I_{\xi_2} \xi_2, \xi_2 \> = \< \I_{\xi_3} \xi_3, \xi_3 \> =0$. From \eqref{newELIstat} we get $\< \I_{\xi_1} \xi_2, \xi_3 \> =0$,
and by \eqref{ELSmixIstat2} we find
\begin{eqnarray*}
 \< \I_{\xi_1} \xi_1, \xi_2 \>
 %= \frac{1}{2}\,\< \xi_2, \I_{\xi_2} \xi_2 + \I_{\xi_3} \xi_3 \> = \frac{1}{2}\,\< \I_{\xi_3} \xi_3, \xi_2 \> \\
 =\frac{1}{4}\,\< \xi_2, \I_{\xi_1} \xi_1 \> + \frac{1}{4}\,\< \xi_2, \I_{\xi_2} \xi_2 \> =\frac{1}{4}\,\< \I_{\xi_1} \xi_1, \xi_2 \>\,.
\end{eqnarray*}
Hence, $\< \I_{\xi_1} \xi_1, \xi_2 \> =0$. Therefore, all components of $\I$ in the frame
%orthonormal basis
$\{ \xi_1,\xi_2,\xi_3 \}$ vanish.
\end{proof}

\section{Critical metrics and metric-compatible connections}
\label{sec-metricconnections}

In this part, we apply the results of Sections~\ref{sec:adapted-metric} and \ref{sec:contorsion}
and consider particular cases of pairs $(g,\I)$ critical for \eqref{Eq-Smix-g}, where $\I$ is the contorsion of a metric-compatible connection for $g$.
We restrict ourselves to cases where $\delta_g\bar Q_\mu$, which usually has a complicated form, can be written explicitly
using our results for two distributions \cite{rz-3}.
First we consider semi-symmetric connections, which are critical only among semi-symmetric connections -- this condition is sufficient to determine them
%these connections
in terms of metric.
The second case we examine are metric connections on 3-Sasaki manifolds, which carry four naturally defined totally geodesic orthogonal distributions.

\subsection{Semi-symmetric connections}
\label{sec:contorsion_semi_symmetric}

%Here, we restrict variations of the mixed scalar curvature on $(M,g;\mD_1,\ldots\mD_k)$ to semi-symmet\-ric connections
%and obtain meaningful Euler-Lagrange equations, which allow us to present the Ricci type tensor $\overline\Ric_{\,\mD}$ explicitly.

A useful case of metric-compatible connections are semi-symmetric connections introduced by K.\,Yano~\cite{Yano}.

\begin{Definition}\rm
%A linear connection $\bar\nabla$ on $M$ is \textit{semi-symmetric} if its torsion tensor $S$ satisfies
%$S(X,Y)=\omega(Y)X-\omega(X)Y$, where $\omega$ is a one-form on $M$.
A linear connection $\bar\nabla$ on $(M,g)$ is said to be \textit{semi-symmetric} if
\begin{equation}\label{Uconnection}
 \bar\nabla_XY=\nabla_XY + \<U, Y\> X -\<X,Y\>U,\quad X,Y\in\mathfrak{X}_M\,,
\end{equation}
where $U$
%=\omega^\sharp$
is a given vector field on $M$.
%, see \cite{Yano}.
In this case,
%the contorsion tensor is
%\begin{equation*}
%\label{Uconnection-2}
$\I_XY=\<U, Y\> X -\<X,Y\>U$\,.
%\end{equation*}
\end{Definition}

%Clearly, if $U=0$ then $\I=0$ and ${\bar \nabla}$ is the Levi-Civita connection.

%For an arbitrary $\bar\nabla$, the tensor $\overline\Ric_{\,\mD}$ in \eqref{E-geom} has a long expression and is not explicitly specified here.
%For semi-symmetric connections, using Theorem~\ref{propUconnectionEL}, we present
%the tensor
%$\overline\Ric_{\,\mD}$ explicitly
%(by its restrictions on subbundles~$\mD_\mu$).
%The following theorem %(on a simplified version of Einstein-Cartan type gravity)
%generalizes \cite[Theorem~6(a)]{rz-3} with $k=2$.

\begin{Theorem}
%\label{propUconnectionEL}
A pair $(g, U)$, where $g$ is an adapted metric on $(M;\mD_1,\ldots\mD_k)$, and a vector field $U$ corresponds to a semi-symmet\-ric connection
%$\I$
on $M$, is critical for \eqref{Eq-Smix-g} with respect to adapted variations of metric preserving the volume of $\Omega$
if and only if the following Euler-Lagrange equations
with ${\delta Q}_\mu$ from Proposition~\ref{C-vr-terms}
%, see \eqref{E-delta-g-J},
are satisfied for all $\mu\in\{1,\ldots,k\}$ and some~$\lambda\in\RR$:
\begin{eqnarray}\label{UELD}
\nonumber
&&{\delta Q}_\mu -\frac12\,\big({n_\mu^\perp(n_\mu-1)} + \sum\nolimits_{\,\nu\ne\mu} n_\nu(n_\nu^\perp-1)\big) P_\mu U^\flat\otimes P_\mu U^\flat \\
\nonumber
&& +\,\frac14\, \Div\big( (n_\mu-n_\mu^\bot) P_\mu^\bot U + \sum\nolimits_{\,\nu\ne\mu} (n_\nu^\bot-n_\nu) P_\nu U\big) g_\mu \\
&& +\,\big(\,{\rm S}_{\,\mD_1,\ldots,\mD_k} -\frac12\,\Div \sum\nolimits_{\nu} (  - n_\nu^\perp P_\nu U - n_\nu P_\nu^\perp U +  H_\nu +\tilde{H}_\nu) +\lambda\big) g_\mu = 0\,.
%,\quad 1\le\mu\le k.
%%%%%%%%%%%%%%%%%%%
\end{eqnarray}
\end{Theorem}

%PROOF CORRECTED
\begin{proof}
Let $\bar\nabla$ be a semi-symmetric connection on $(M,g,\mD_\mu,\mD_\mu^\bot)$, then \eqref{E-barQ} reduces to
\begin{equation}\label{QforUconnection}
-2\,\bar Q(\mD_\nu,g,U) =(n_\nu-n_\nu^\bot) \< U,H_\nu - \tilde H_\nu\> + n_\nu^\bot n_\nu \<U,\,U\> -n_\nu\<P^\bot_\nu U, P^\bot_\nu U\>
-n_\nu^\bot\< P_\nu U, P_\nu U\>\,,
\end{equation}
see \cite[Lemma~6(a)]{rz-3} for $k=2$.
For a $\mD_\mu$-variation of metric, up to divergences of compactly supported vector fields we get (see \cite[Lemma~6(b)]{rz-3} for $k=2$),
\[
%\label{dtQgforUconnection}
\dt\bar Q(\mD_\nu,g_t,U)|_{\,t=0}
=\bigg\{
\begin{array}{cc}
\<{B}_\mu, \frac14\,(n_\nu^\bot-n_\nu)(\Div P_\nu U) g_\mu -\frac12\,n_\nu(n_\nu^\bot-1)P_\nu^\perp U^\flat\otimes P_\nu^\perp U^\flat\>, & \nu\ne\mu\,, \\
\hskip-2mm\<{B}_\mu, \frac14(n_\mu-n_\mu^\bot)(\Div P_\mu^\bot U) g_\mu -\frac12n_\mu^\perp(n_\mu-1) P_\mu U^\flat\otimes P_\mu U^\flat\>
& \nu = \mu\,.
\end{array}
\]
From \cite[Lemma 6]{rz-3} we get
%\[
 $P_\mu \tr_\mu^\perp \I = -n_\mu^\perp P_\mu U$
 %,\quad
 and
 $P^\perp_\mu \tr_\mu \I = -n_\mu P_\mu^\perp U$\,.
%\]
Using the above and  %\eqref{E-prop-X} and slightly modifying \eqref{ElmixDD-b} (which was obtained for statistical connections)
\eqref{E-dt-barQ} in \eqref{ElmixDDvp-b}, we get \eqref{UELD}.
%One can prove (b)
\end{proof}

\begin{Remark}\rm
 One can present \eqref{UELD} in the equivalent form of \eqref{E-geom}, see \cite[Section~3.4]{rz-3} for $k=2$,
using the Ricci type tensor $\overline\Ric_{\,\mD\,|\,\mD_\mu\times\mD_\mu} = -{\delta Q}_\mu -{\delta_g\bar Q}_\mu +\rho_\mu\,g_\mu$,
see \eqref{E-main-0ij-kk}.
%and its trace $\overline{\cal S}_{\mD}$
%explicitly and .
%An equivalent form of \eqref{UELD} is \eqref{E-geom} with $\overline\Ric_{\,\mD}$ given by its restrictions on subbundles $\mD_\mu$,
%CHECK THIS FORMULA:
%\begin{eqnarray}\label{E-Ric-D-semi-sym}
%\nonumber
% && \overline\Ric_{\,\mD\,|\,\mD_\mu\times\mD_\mu} = -{\delta Q}_\mu +\rho_\mu \,g_\mu \\
%\nonumber
% && +\,\frac12\,n_\mu^\bot(n_\mu-1)P_\mu U^\flat\otimes P_\mu U^\flat -\frac14\,(n_\mu-n_\mu^\bot)\big(\Div P_\mu^\bot U +\frac{Z_\mu}{2-n}\,\big) g_\mu \\
% && +\sum\nolimits_{\,\nu\ne\mu}\big(\,\frac12\,n_\nu^\bot(n_\nu-1) P_\nu^\bot U^\flat\otimes P_\nu^\bot U^\flat -\frac14\,(n_\nu^\bot-n_\nu)
% \Div P_\nu U\cdot g_\mu\,\big),
%\end{eqnarray}
%also
%CHECK THIS FORMULA, I DID NOT VERIFY IT:
%\[
% $\overline{\cal S}_\mD
%= \tr_g \overline\Ric_{\,\mD} = -\tr_g \sum\nolimits_{\,\mu}{\delta Q}_\mu +\sum\nolimits_{\,\mu} \big(n_\mu\,\rho_\mu+\frac{2}{2-n} Z_\mu\big)$,
%%\]
%where $n=\dim M >2$ and
%%CHECK THIS FORMULA, I DID NOT VERIFY IT:
%\begin{eqnarray*}
% Z_\mu \eq \frac12\,n_\mu^\bot(n_\mu-1)\<P_\mu U,P_\mu U\> -\frac14\,n_\mu(n_\mu-n_\mu^\bot)\Div P_\mu^\bot U \\
% \plus\sum\nolimits_{\,\nu\ne\mu}\big(\,\frac12\,n_\nu(n_\nu^\bot-1)\<P_\nu^\bot U,P_\nu^\bot U\> -\frac14\,n_\nu^\bot(n_\nu^\bot-n_\nu)\Div P_\nu U\,\big).
%\end{eqnarray*}
%Indeed, replacing $\overline\Ric_{\,\mD}$ in \eqref{E-geom} by \eqref{E-Ric-D-semi-sym}, we get an identity.
\end{Remark}

We consider variations of a semi-symmetric connection only among connections that also satisfy \eqref{Uconnection} and obtain the following result.

\begin{Corollary}
A semi-symmetric connection defined by $U$ is critical for the action \eqref{Eq-Smix-g} with fixed $g$, with respect to variations among semi-symmetric connections if and only if
\begin{equation} \label{Ucritcontorsion}
\sum\nolimits_{\,\mu} \big( (n_\mu - n_\mu^\perp) ( H_\mu - {\tilde H}_\mu )
+ 2 n_\mu (n_\mu^\perp - 1) P_\mu^\perp U + 2n_\mu^\perp (n_\mu -1) P_\mu U\big) =0\,.
\end{equation}
\end{Corollary}

\begin{proof}
The proof follows from \eqref{E-Q1Q2-gen} and \eqref{QforUconnection}, similarly as in \cite[Proposition 10]{rz-3}.
\end{proof}

\begin{Corollary} \label{corsemisymmetricharmdim}
Let $k>2$ or $n_1+n_2 >2$. If all $\mD_\mu$ are harmonic or all $\mD_\mu$ have equal dimension, then the Levi-Civita connection is the only semi-symmetric connection critical for the action \eqref{Eq-Smix-g} with fixed metric, with respect to variations among semi-symmetric connections.
\end{Corollary}

\begin{proof}
First, we assume that all $\mD_\mu$ are harmonic. Then we get from \eqref{Ucritcontorsion} for each $\mu =1,\ldots, k$:
\[
%\label{Ucritcontorsion}
\big( 2n_\mu^\perp (n_\mu -1) + \sum\nolimits_{\,\nu \neq \mu} 2 n_\nu (n_\nu^\perp - 1) \big) P_\mu U  =0\,.
\]
The coefficient before $P_\mu U$ is non-negative and vanishes only for $k=2$ and $n_1=n_2=1$, which contradicts the assumption about dimensions of distributions. Hence, it follows that $P_\mu U=0$ for all $\mu =1 ,\ldots , k$, and therefore %hence
$U=0$.
%
%Now,
For the second case, let $\dim \mD_\mu = m$ for all $\mu =1 ,\ldots , k$, then $\dim M = km$ and from \eqref{calH}, and $P_\nu^\perp = U - P_\nu U$, we obtain in \eqref{Ucritcontorsion}
\begin{equation} \label{Ucritcontorsion1}
 \big( 2m ( km - m -1) (k-1) + 2 (km-m)(m-1) \big) U = 0\,.
\end{equation}
Since the coefficient before $U$ is $2m(k-1)(km-2)\ne0$, for $k>2$ or $m>1$ we get $U=0$.
\end{proof}

From %Remark~\ref{remarkEinstein}
Proposition~\ref{propEinstein}
and %the above,
\eqref{Ucritcontorsion1}, similarly as in the proof of Corollary \ref{corsemisymmetricharmdim}
%Corollary \ref{corsemisymmetricharmdim},
we recover the following result, obtained %in more generality
in \cite{Fasihi}.

\begin{Corollary}
%\label{corEinsteinsemisymmetric}
If $\dim M>2$, then the Levi-Civita connection is the only semi-symmetric connection critical for the Einstein-Hilbert action.
\end{Corollary}

By \eqref{Ucritcontorsion}, connection critical among semi-symmetric connections is determined by metric. Solving \eqref{Ucritcontorsion} for $U$ and inserting the result in \eqref{UELD} yields equation that determines metrics $g$ admitting semi-symmetric contorsions $\I$ such that the pair $(g,\I)$ is critical for \eqref{Eq-Smix-g} restricted to adapted metrics and contorsions of semi-symmetric connections.

%Together, \eqref{UELD} and \eqref{Ucritcontorsion} determine metrics $g$ admitting semi-symmetric contorsions $\I$ such that the pair $(g,\I)$ is critical for \eqref{Eq-Smix-g} restricted to adapted metrics and contorsions of semi-symmetric connections.

%\begin{Corollary}
%Let $g$ be a critical adapted metric for the action \eqref{Eq-Smix-g0} with respect to adapted variations on $(M;\mD_1,\ldots,\mD_k)$.
%Suppose that $n_j=1$, $\Div H_j =0$ and $\tilde H_j=0$ for some $j\in[1,\ldots,k]$.
%Then there exists a vector field $U$ tangent to $\mD_j$ such that the pair $(g,\I)$ with $\I$ given by \eqref{Uconnection-2}, is critical
%for the action \eqref{Eq-Smix-g} with respect to adapted variations of $g$ and all variations of $U$.
%\end{Corollary}

%\begin{proof}
%Using \eqref{UcriticalforI}, components of $U$ can be easily expressed through the mean curvature vectors of the distributions.
%By conditions, \eqref{UELD} take the following form:
%\begin{eqnarray}\label{UELD-2}
%\nonumber
% && -\frac12\,{n_\mu(n_\mu^\bot-1)}\,P_\mu^\bot U^\flat\otimes P_\mu^\bot U^\flat +\frac14\,(n_\mu-n_\mu^\bot)\Div P_\mu^\bot U \cdot g_\mu \\
% && -\sum\nolimits_{\,\nu\ne\mu}\big(\,\frac12\,n_\nu^\bot(n_\nu-1)\,P_\nu U^\flat\otimes P_\nu U^\flat
% -\frac14\,(n_\nu^\bot-n_\nu)\Div P_\nu U\cdot g_\mu \,\big) = 0 .
%%%%%%%%%%%%%%%%%%%
%\end{eqnarray}
%\end{proof}

%Critical metrics and metric connections for
\subsection{3-Sasaki manifolds}
\label{sec:5-Sasaki}

In this section, we consider metric-compatible connections on a 3-Sasaki manifold $(M^{4m+3},g)$ with $m>0$ (see \cite{blair,bg}),
which has four naturally defined orthogonal distributions.
Let $\mD_1,\mD_2,\mD_3$ be the one-dimensional distributions spanned by Reeb fields $\xi_1, \xi_2, \xi_3$, and let $\mD_4$ be their orthogo\-nal complement.
We~will write ${\tilde T}^\sharp_{\xi_a }$ instead of ${\tilde T}^\sharp_{a, \xi_a}$ for $a\le 3$.
From \cite{rz-2} we get
\begin{equation}\label{Eq-3Sas}
 \nabla_Y \xi_a = - {\tilde T}^\sharp_{\xi_a} Y\quad (Y \in TM,\quad a\le 3)\,,
\end{equation}
and $\tT_{\xi_a} \tT_{\xi_a} = - \id |_{\,\mD_a^\perp }$, $\tT_{\xi_a}{\xi_b} = {\xi_c}$ for even permutation of $a,b,c$.
Using \eqref{Eq-3Sas}, we can formulate
%the~following

\begin{Lemma} \label{lemSasakigeodint}
For a 3-Sasaki manifold $(M,g)$, the distributions $\mD_\mu\ (\mu\le 4)$ are totally geodesic, pairwise mixed totally geodesic,
and $\mD_4$ is mixed integrable with every $\mD_a\ (a\le 3)$.
\end{Lemma}

%\begin{proof}
%For $X,Y \in \mD_4$ and $a\le 3$ we obtain
%\begin{eqnarray*}
%\< \nabla_X Y + \nabla_Y X, \xi_a\> &=& - \< Y, \nabla_X \xi_a\> - \< X, \nabla_Y \xi_a\> = \< Y, \tT_{\xi_a} X\> + \< X, \tT_{\xi_a} Y\> =0,\\
%\nabla_{\xi_a} \xi_a &=& 0.
%\end{eqnarray*}
%For $X \in \mD_4$, $a,b \in \{1,2,3\}$, $a \ne b$, and $c \notin \{a,b\}$, we get $-{\tilde T}^\sharp_{\xi_a} \xi_c = \pm \xi_b$ and
%\begin{eqnarray*}
%\< [X, \xi_a], \xi_b\> & = & \< \nabla_X \xi_a - \nabla_{\xi_a} X, \xi_b\>
%= \< - {\tilde T}^\sharp_{\xi_a} X, \xi_b\> + \< X, \nabla_{\xi_a} \xi_b\> \\
%&=& \pm \< {\tilde T}^\sharp_{\xi_a} X, {\tilde T}^\sharp_{\xi_a} \xi_c\> + \< X, \nabla_{\xi_a} \xi_b\> \\
%&=& \mp \< {\tilde T}^\sharp_{\xi_a}{\tilde T}^\sharp_{\xi_a} X, \xi_c\> \pm \< X, \xi_c\>
%= \pm 2 \< X, \xi_c\> =0,
%\end{eqnarray*}
%as $X \perp \mD_c$. Similarly, we can prove that $\< \nabla_X \xi_a + \nabla_{\xi_a} X, \xi_b\>=0$, also
%\[
%\< \nabla_{\xi_a}\xi_b + \nabla_{\xi_b}\xi_a, \xi_c \> = \< \tT_{\xi_b} \xi_a + \tT_{\xi_a} \xi_b, \xi_c\> = \< \pm \xi_c \mp \xi_c, \xi_c\> =0,
%\]
%but
%\[
%\< \nabla_{\xi_a}\xi_b - \nabla_{\xi_b}\xi_a, \xi_c \> = \< \tT_{\xi_b} \xi_a - \tT_{\xi_a} \xi_b, \xi_c\> = \pm 2,
%\]
%that completes the proof.
%\end{proof}

Let ${\rm Sym}(F (X,Y)) = \frac{1}{2} ( F(X,Y) + F(Y,X))$ for all $X,Y \in TM$ and all $(s,2)$-tensors $F$.

In what follows we consider %particular terms of these equations.
terms appearing in \cite{rz-3} for contact manifolds, that will also appear in our Euler-Lagrange equations below.
We define the following tensors on $M$:
% for $X,Y \in TM$:
\begin{eqnarray*}
&& \phi_\nu (X,Y) = (\I + \I^\wedge)_{P_\nu^\perp X}\, P_\nu^\perp Y, \quad \phi_\nu^\top (X,Y) = P_\nu \phi_\nu (X,Y)\,,\\
&& {\tilde \phi}_\nu (X,Y) = (\I + \I^\wedge)_{ P_\nu^\top X}\, P_\nu^\top Y, \quad {\tilde \phi}_\nu^\perp (X,Y) = P_\nu^\perp {\tilde \phi}_\nu (X,Y)\,,\\
%OLD && \chi_\nu (X,Y) = \frac{1}{2} \sum (\<X, \I_{{\cal E}_{\nu,j}}E_{\nu,a}\> \< Y, {\tilde T}^\sharp_{\nu, E_{\nu, a}}{\cal E}_{\nu,j}\> +
%\<Y, \I_{{\cal E}_{\nu,j}}E_{\nu,a}\> \<X, {\tilde T}^\sharp_{\nu, E_{\nu, a}}{\cal E}_{\nu,j}\> )\,,\\
&& \chi_\nu (X,Y) = {\rm Sym} \big( \sum (\<X, \I_{{\cal E}_{\nu,j}}E_{\nu,a}\> \< Y, {\tilde T}^\sharp_{\nu, E_{\nu, a}}{\cal E}_{\nu,j}\> \big) \,,\\
&& {\tilde \chi}_\mu (X,Y) = {\rm Sym} \big( \sum (\< X, \I_{{\mu,j}}{\cal E}_{\mu,a}\> \< Y, T^\sharp_{\mu, {\cal E}_{\mu,a}}E_{\mu, j} \> \big) \,.
\end{eqnarray*}
Next, we calculate on a 3-Sasaki manifold values of some previously defined tensors. %defined in previous sections.

\begin{Lemma}\label{lemmadeltaQbarterms}
Let $(M,g)$ be a 3-Sasaki manifold, and let $\I$ be the contorsion tensor critical for the action \eqref{Eq-Smix-g} with fixed $g$. Then for all $X,Y \in \mD_4$ and
%all $\xi_a\ (
$a=1,2,3$:
\begin{subequations}
\begin{eqnarray}\label{tildechi4}
&& {\tilde\chi}_4(X,Y) = \frac{1}{2} {\rm Sym} \big( \sum\nolimits_{b}\<{\tilde\phi}_4(\tT_{\xi_b} Y, X) , \,\xi_b\> \big) + 3\,\<X,Y\>\,,\\
\label{chiaXY}
&& \chi_a (X,Y) = \frac{1}{2} {\rm Sym}\big( \< {\tilde \phi}_4 (\tT_{\xi_a} Y, X ), \xi_a\> \big) + \<  X,   Y\>  \, \\ %%it is not {\tilde \chi}_4 (X,Y)\,, \\
\label{widetildeTa}
&& \widetilde{\cal T}^\flat_a = -g_a^\perp, \quad {\cal T}^\flat_4 = -3\,g_4\,, \\
%%%%%%%%%%%%%%%%%%%%%%%%%%%%%%%%%%%%%%%%%%%%%
\label{UpsilonTa}
&& \Upsilon_{T_a, T_a}=0, \quad \Upsilon_{ {\tilde T}_4, {\tilde T}_4}= 0\,, \\
%%%%%%%%%%%%%%%%%%%%%%%%%%%%%%%%%%%%%%%%%%%%%
\label{chi23xi1}
&& (\chi_2 + \chi_3)(\xi_1,\xi_1) = 0\,, \quad % \< \xi_1, \I_{\xi_3}\xi_2 - \I_{\xi_2}\xi_3\>, \\
(\chi_1 + \chi_3)(\xi_2,\xi_2) = 0\,, \quad % \< \xi_2, \I_{\xi_1} \xi_3 - \I_{\xi_3} \xi_1 \>, \\
(\chi_1 + \chi_2)(\xi_3, \xi_3) = 0\,, \\ % \< \xi_3, \I_{\xi_2} \xi_1 - \I_{\xi_1} \xi_2\> \\
%%%%%%%%%%%%%%%%%%%%%%%%%%%%%%%%%%%%%%%%%%%%%
\label{chi4xia}
&& \chi_4 (\xi_a, \xi_a)=0, \quad {\tilde \chi}_a (\xi_a, \xi_a) =0\,, \\
%%%%%%%%%%%%%%%%%%%%%%%%%%%%%%%%%%%%%%%%%%%%%
\label{widetildecalT4}
&& \widetilde{\cal T}^\flat_4 (\xi_a, \xi_a) =0, \quad {\cal T}^\flat_a (\xi_a, \xi_a) = -2\,, \\
%%%%%%%%%%%%%%%%%%%%%%%%%%%%%%%%%%%%%%%%%%%%%
\label{UpsilonT4xia}
&& \Upsilon_{ T_4, T_4}(\xi_a, \xi_a) = 2\,n_4, \quad \Upsilon_{ {\tilde T}_a, {\tilde T}_a}(\xi_a, \xi_a ) = 4 + 2\,n_4\,.
\end{eqnarray}
\end{subequations}
\end{Lemma}

\begin{proof}
Let $\{\xi_1, \xi_2, \xi_3, E_{1}, \ldots, E_{n_4} \}$ be an orthonormal basis on $M$.
For $X,Y \in \mD_4$ we get
\[
{\tilde \chi}_4 (X,Y) = \frac{1}{2} \sum\nolimits_{a,j} \big(\< X, \I_{ E_{j}}\xi_a\> \< Y, T^\sharp_{4, \xi_a}E_{j} \> +
\< Y, \I_{ E_{j}}\xi_a\> \< X, T^\sharp_{4, \xi_a}E_{j}\> \big),
\]
and by \eqref{ELconnectionNewI8} we obtain for all $a=1,2,3$: %\eqref{tildechi4}:
\begin{eqnarray*}
&& \sum\nolimits_{j} \< X, \I_{ E_{j}}\xi_a \> \< Y, T^\sharp_{4, \xi_a}E_{j} \>
= -\sum\nolimits_{j} \< \I_{ E_{j}}X, \xi_a \> \< Y, T^\sharp_{4, \xi_a}E_{j} \> \\
&& = -\frac{1}{2}\sum\nolimits_{j} \big(\< \I_{ E_{j}}X + \I_X E_{j}, \xi_a\> \< Y, T^\sharp_{4, \xi_a}E_{j} \>
+\< \I_{ E_{j}}X - \I_X E_{j}, \xi_a\> \< Y, T^\sharp_{4, \xi_a}E_{j} \>\big) \\
&& = \sum\nolimits_{j}\big( - \frac{1}{2}\,\< {\tilde \phi}_4 (E_{j}, X ), \xi_a\> \< \xi_a, T_4 (E_{j}, Y)\>
+ \< T_4 ({ E_{j} }, X ), \xi_a\> \< \xi_a, T_4 (E_{j}, Y )\>\big) \\
%&& = \sum\nolimits_{a,j} \big(\frac{1}{2}\,\< {\tilde \phi}_4 (E_{j}, X ), \xi_a\> \< \tT_{\xi_a} Y, E_{j}\>
%+ \< \tT_{\xi_a} X, { E_{j} }\> \< \tT_{\xi_a} Y, E_{j}\> \big) \\
&& = \big(\frac{1}{2}\,\< {\tilde \phi}_4 (\tT_{\xi_a} Y, X ), \xi_a\> + \< \tT_{\xi_a} X, \tT_{\xi_a} Y\>\big)
= \frac{1}{2} \< {\tilde \phi}_4 (\tT_{\xi_a} Y, X ), \xi_a\> + \<X,Y\>\,.
\end{eqnarray*}
Summing the above over $a=1,2,3$ we get \eqref{tildechi4}.

%which is \eqref{tildechi4}.
%$\tT_1 \xi_2 = \xi_3$ etc.
Equations \eqref{widetildeTa}$_1$ follow from $\tT_{\xi_a} \tT_{\xi_a} = - \id$,
similarly, we obtain \eqref{widetildeTa}$_2$, as for $a\le 3$ we get
%\begin{eqnarray*}
${\cal T}^\flat_4 = \sum\nolimits_{\,a} (\tT_{\xi_a} \tT_{\xi_a} )^\flat = -3\, g_4$.
%\end{eqnarray*}
Next, we get $\chi_1(X,Y)$ (and similarly, $\chi_2(X,Y)$ and $\chi_3(X,Y)$):
\begin{eqnarray*}
&&\hskip-4mm
\chi_1 (X,Y) = {\rm Sym}\big( \frac{1}{2} \< {\tilde \phi}_4 (\tT_{\xi_1} Y, X ), \xi_1\> {+} \<P_4 X, P_4 Y\>
 {+}\< X, \I_{\xi_2}\xi_1\> \< Y, \tT_{\xi_1}\xi_2\> {+}\< X, \I_{\xi_3} \xi_1\>\< Y, \tT_{\xi_1}\xi_3\>\big)\\
&&\hskip-6mm = {\rm Sym}\big( \frac{1}{2} \< {\tilde \phi}_4 (\tT_{\xi_1} Y, X ), \xi_1\> + \<P_4 X, P_4 Y\>
 +\< X, \I_{\xi_2} \xi_1\> \< Y, \xi_3\> -\< X, \I_{\xi_3} \xi_1\> \< Y, \xi_2\> \big),\quad X,Y \perp \mD_1\,.
%\chi_2 (X,Y) \eq {\rm Sym}\big( \frac{1}{2} \< {\tilde \phi}_4 (\tT_{\xi_2} Y, X ), \xi_2\> {+} \<P_4 X, P_4 Y\>
% {+}\< X, \I_{\xi_1} \xi_2\> \< Y, \tT_{\xi_2} \xi_1\> {+}\< X, \I_{\xi_3}\xi_2\>\< Y, \tT_{\xi_2}\xi_3\>\big) \\
%\eq {\rm Sym}\big( \frac{1}{2} \< {\tilde \phi}_4 (\tT_{\xi_2} Y, X ), \xi_2\> + \<P_4 X, P_4 Y\> \\
%&&+ \< X, \I_{\xi_3} \xi_2\> \< Y, \xi_1\> - \< X, \I_{\xi_1} \xi_2\> \< Y, \xi_3\> \big), \quad X,Y \perp \mD_2\,,\\
%\chi_3 (X,Y) \eq {\rm Sym}\big( \frac{1}{2} \< {\tilde \phi}_4 (\tT_{\xi_3} Y, X ), \xi_3\> {+} \<P_4 X, P_4 Y\>
% {+}\< X, \I_{\xi_1} \xi_3\>\< Y, \tT_{\xi_3} \xi_1\> {+}\< X, \I_{\xi_2}\xi_3\>\< Y, \tT_{\xi_3} \xi_2\> \big) \\
% \eq {\rm Sym}\big( \frac{1}{2} \< {\tilde \phi}_4 (\tT_{\xi_3} Y, X ), \xi_3\> + \<P_4 X, P_4 Y\> \\
%&&+ \< X, \I_{\xi_1} \xi_3\> \< Y, \xi_2\> -\< X, \I_{\xi_2} \xi_3\> \< Y, \xi_1\> \big),\quad X,Y \perp \mD_3\,.
\end{eqnarray*}
Hence,
%for $X,Y \in \mD_4$
we get \eqref{chiaXY}. We get \eqref{chi23xi1}
from the above computations and \eqref{ELI5XYZ} for metric-compatible connections, which yields for all $a,b,c =1,2,3$, $a \ne b \ne c \ne a$ and all $X \in \mD_4$:
\begin{equation} \label{3Sasaki3different}
0 = \< \I_X \xi_a - \I_{\xi_a} X, \xi_b\> = \< \I_{\xi_b} \xi_a - \I_{\xi_a}{\xi_b}, X\> = \< \I_{\xi_b} \xi_a - \I_{\xi_a}{\xi_b}, \xi_c\>\, .
\end{equation}
We obtain \eqref{chi4xia}$_1$ and \eqref{widetildecalT4}$_1$ from $\tT_{4,X} =0$ for all $X \in \mD_4$.

We obtain for $a,b,c \in \{ 1,2,3 \}$, $a \ne b \ne c \ne a$:
\[
 {\cal T}^\flat_a (\xi_a, \xi_a) = \sum\nolimits_{\,i} \< \tT_{4, E_{i}} \tT_{4, E_{i}} \xi_a, \xi_a\>
 + \< \tT_{\xi_b} \tT_{\xi_b} \xi_a, \xi_a\> + \< \tT_{\xi_c} \tT_{\xi_c} \xi_a, \xi_a\> = -2\,,
\]
which is \eqref{widetildecalT4}$_2$.
For $a\le 3$ \eqref{UpsilonTa}$_1$ follows from the fact that each $\mD_a$ is integrable.

For $a\le 3$ we obtain \eqref{UpsilonT4xia}$_1$ as follows:
\begin{eqnarray*}
&& \frac{1}{2} \Upsilon_{ T_4, T_4}(\xi_a, \xi_a) = \sum\nolimits_{i,j} \< \xi_a, T_a(E_{i}, E_{j} )\> \<\xi_a, T_a(E_{i}, E_{j} )\> \\
&&=\sum\nolimits_{i,j} \< \tT_{\xi_a} E_{i}, E_{j}\> \< \tT_{\xi_a} E_{i}, E_{j}\>
%&=& \sum\nolimits_{i} \< \tT_{\xi_a} E_{i}, \tT_{\xi_a} E_{i} \> \\
= - \sum\nolimits_{i} \< \tT_{\xi_a} \tT_{\xi_a} E_{i}, E_{i} \> = n_4\,,
\end{eqnarray*}
as $\tT_{\xi_a}\tT_{\xi_a} = - \id |_{\,\mD^\perp_a}$.
For \eqref{UpsilonTa}$_2$, let $X,Y \in \mD_4$, then
\[
\frac{1}{2}\Upsilon_{ {\tilde T}_4, {\tilde T}_4}(X,Y) = \sum\nolimits_{a,b} \< X, {\tilde T}_4 (\xi_a, \xi_b)\> \< Y, {\tilde T}_4 (\xi_a, \xi_b)\>=0\,,
\]
because $[\xi_a,\xi_b] \perp \mD_4$.
For \eqref{UpsilonT4xia}$_2$, let $a,b,c=1,2,3$ and $a \ne b \ne c \ne a$, then
\begin{eqnarray*}
\frac{1}{2}\,\Upsilon_{{\tilde T}_a, {\tilde T}_a}(\xi_a, \xi_a )
&=& \sum\nolimits_{b,c} \<\xi_a, {\tilde T}_a (\xi_b, \xi_c)\> \<\xi_a, {\tilde T}_a (\xi_b, \xi_c)\> \\
&& +\sum\nolimits_{i,j} \<\xi_a, {\tilde T}_a (E_{i}, E_{j})\> \<\xi_a, {\tilde T}_a (E_{i}, E_{j} )\>
%\\ &=& \sum\nolimits_{b,c} \<\tT_{\xi_a} \xi_b, \xi_c\> \<\tT_{\xi_a}\xi_b, \xi_c\>
%+\sum\nolimits_{i} \<\tT_{\xi_a} E_{i}, \tT_{\xi_a} E_{i}\>
= 2 + n_4\, .
\end{eqnarray*}
For \eqref{chi4xia}$_2$, let $a\le 3$, ${\cal E}_{a,j} \in \mD_a^\perp$, then
%\begin{eqnarray*}
 ${\tilde \chi}_a (\xi_a, \xi_a) = \sum\nolimits_{j} \< \xi_a, \I_{\xi_a}{\cal E}_{a,j}\> \< \xi_a, T^\sharp_{a, {\cal E}_{a,j}} \xi_a\> =0$.
%\end{eqnarray*}
%that completes the proof.
\end{proof}

By Lemma~\ref{lemSasakigeodint}, on a 3-Sasaki manifold $(M,g)$, each $\mD_\mu\ (\mu\le 4)$ is totally geodesic and has totally geodesic orthogonal complement. Thus, the assumptions of \cite[Lemma~4]{rz-3} are satisfied.

We additionally assume the following condition:
\begin{equation} \label{assumptionXYZ}
\< \I_{X}Y, Z \>=0\,, \quad X \in \mD_\mu ,\ Y \in \mD_\rho,\ Z \in \mD_\xi , \quad \mu \neq \rho \neq \xi \neq \mu\,,
\end{equation}
%for $X,Y,Z$ each from a different distribution -
which is consistent with \eqref{ELI5XYZ}.  Note that from \eqref{3Sasaki3different} it follows that, e.g., characteristic connection on 7-dimensional 3-Sasaki manifolds \cite{AgricolaFriedrich} does not satisfy \eqref{ELI5XYZ}. Then, following the proof of \cite[Lemma~4]{rz-3}, we obtain for each $\mD_\mu$ and for a metric-compatible connection $\nabla+\I$:
\begin{eqnarray*}
2\delta_g\bar Q_\mu &=& \sum\nolimits_{\nu\ne\mu}\big(\<\phi_\nu, \frac{1}{2} \tr_\nu^\top \I\> - 2 \Div\phi_\nu^\top + 7\,\chi_\nu
-\Div (P_\nu \tr^\perp_\nu \I)\,g_\nu^\perp - 2\,\widetilde{\cal T}_\nu^\flat - \frac{3}{2}\Upsilon_{T_\nu,T_\nu}\big) \\
%%%%%
%dual
%%%%%
&& +\, \< {\tilde \phi}_\mu, \frac{1}{2} \tr_\mu^\perp \I\> - 2 \Div {\tilde \phi}_\mu^\perp + 7 {\tilde \chi}_\mu
-\Div (P_\mu^\perp \tr^\top_\mu \I )\,g_\mu - 2\, {\cal T }_\mu^\flat - \frac{3}{2} \Upsilon_{{\tilde T}_\mu,{\tilde T}_\mu}, \\
\end{eqnarray*}
and, from Proposition \ref{C-vr-terms}, we get
\begin{equation} \label{deltaQmetric}
\delta {Q}_\mu = \sum\nolimits_{\nu\ne\mu} \big(2\,\widetilde{\cal T }_\nu^\flat + \frac{1}{2} \Upsilon_{T_\nu,T_\nu}\big) + 2\,{\cal T }_\mu^\flat + \frac{1}{2} \Upsilon_{{\tilde T}_\mu,{\tilde T}_\mu},
\end{equation}
%
%%\begin{eqnarray*}
%\delta {Q}_\mu &=& \sum\nolimits_{\nu\ne\mu} \big(2\,\widetilde{\cal T }_\nu^\flat + \frac{1}{2} \Upsilon_{T_\nu,T_\nu}\big)
%%\\
%%%%%%
%%dual
%%%%%%
%%&&
%+ 2\,{\cal T }_\mu^\flat + \frac{1}{2} \Upsilon_{{\tilde T}_\mu,{\tilde T}_\mu}.
%\end{eqnarray*}
therefore, for $\mu\le 4$:
\begin{eqnarray} \label{deltaQQbarSasaki}
&& 2\,( \delta {Q}_\mu + \delta_g \bar Q_\mu) = \!\sum\limits_{\nu\ne\mu}\big(\< \phi_\nu, \frac{1}{2}\tr_\nu^\top \I\> {-} 2\Div\phi_\nu^\top
{+} 7\chi_\nu {-} \Div (P_\nu \tr^\perp_\nu \I)\,g_\nu^\perp {+} 2\widetilde{\cal T }_\nu^\flat {-}\frac{1}{2}\Upsilon_{T_\nu,T_\nu}\big)\nonumber \\
%%%%%
%dual
%%%%%
&& +\, \< {\tilde \phi}_\mu, \frac{1}{2} \tr_\mu^\perp \I\> - 2 \Div {\tilde \phi}_\mu^\perp + 7 {\tilde \chi}_\mu
-\Div (P_\mu^\perp \tr^\top_\mu \I ) \,g_\mu + 2\, {\cal T }_\mu^\flat -\frac{1}{2} \Upsilon_{{\tilde T}_\mu,{\tilde T}_\mu} .
\end{eqnarray}
Using the above formulas, we obtain the following presentation of the Euler-Lagrange equations for the action \eqref{Eq-Smix-g}
on a 3-Sasaki manifold.

\begin{Theorem}
Let $(M,g)$ be a 3-Sasaki manifold, and $\I$ be the contorsion tensor critical for the action \eqref{Eq-Smix-g} with fixed $g$
such that \eqref{assumptionXYZ} holds.
%$\< \I_{X}Y, Z \>=0$ for $X,Y,Z$ each from a different distribution.
Then the pair $(g,\I)$ is critical for the action \eqref{Eq-Smix-g} with respect to adapted variations of $g$ preserving the volume of $(M,g)$ and all variations of $\I$ if and only if all equations of Theorem~\ref{corI} hold,  %i.e., \eqref{E-delta-I-J},
and the following Euler-Lagrange equations \eqref{E-delta-g-J} are valid for some $\lambda\in\RR$, all $X,Y \in \mD_4$ and $a=1,2,3$:
\begin{subequations}
\begin{eqnarray} \label{ELSasakiD4}
&& \sum\nolimits_{\nu \ne 4} \big(\< \phi_\nu, \frac{1}{2} \tr_\nu^\top \I\> - 2 \Div \phi_\nu^\top \big) (X,Y)
%\nonumber \\ &&
 + 7 \sum\nolimits_{\,a} {\rm Sym} ( \< {\tilde \phi}_4 (\tT_{\xi_a} Y, X ) , \xi_a\> ) \nonumber\\
&& +\,30\, \< X,Y \> %g_4(X,Y)
%+ 7 {\rm Sym}(\< X, \I_{\xi_2} \xi_1 - \I_{\xi_1} \xi_2\> \< Y, \xi_3\> ) \\
%&& + 7 {\rm Sym}(\< X, \I_{\xi_1} \xi_3 - \I_{\xi_3} \xi_1\> \< Y, \xi_2\> )
%+ 7 {\rm Sym}(\< X, \I_{\xi_3} \xi_2 - \I_{\xi_2} \xi_3\> \< Y, \xi_1\> ) =0 \\
- \sum\nolimits_{\nu \ne 4} \Div (P_\nu \tr^\perp_\nu \I )\, \< X,Y \> \nonumber\\
%%%%%
%dual
%%%%%
&& +\,\< {\tilde \phi}_4, \frac{1}{2} \tr_4^\perp \I\> (X,Y) - 2 \Div {\tilde \phi}_4^\perp (X,Y)
%\nonumber\\ &&
- \Div( P^\perp_4 \tr^\top_4 \I )\, \< X,Y \> \nonumber\\
&& +\,(2\,\overline{\rm S}_{\,\mD_1,\mD_2,\mD_3,\mD_4}
-\sum\nolimits_{\,i}\Div( \,P_i\tr_{\,i}^\bot \I + \,P_i^\bot\tr_{\,i}^\top \I ) )g_4(X,Y) = \lambda\, \< X,Y \> \,,\\
%\end{eqnarray}
%and
%\begin{eqnarray}
\label{ELSasakiDa}
&& -\,9 \Div (P_a^\perp\I_{\xi_a} \xi_a) %+4 \< \I_{\xi_2}\xi_1 + \I_{\xi_1}\xi_2, \xi_3\> -4 \< \I_{\xi_3}\xi_1 + \I_{\xi_1}\xi_3, \xi_2\>\nonumber \\
%&&
 - \sum\nolimits_{\nu \ne a} \Div (P_\nu \tr^\perp_\nu \I ) -10 -2n_4 \nonumber\\
&& +\,2\,\overline{\rm S}_{\,\mD_1,\mD_2,\mD_3,\mD_4}
-\sum\nolimits_{\,i}\Div( \,P_i\tr_{\,i}^\bot \I + \,P_i^\bot\tr_{\,i}^\top \I ) = \lambda\,, %\\
%
%\label{ELSasakiD1}
%&& -\,9 \Div (P_1^\perp\I_{\xi_1} \xi_1) %+4 \< \I_{\xi_2}\xi_1 + \I_{\xi_1}\xi_2, \xi_3\> -4 \< \I_{\xi_3}\xi_1 + \I_{\xi_1}\xi_3, \xi_2\>\nonumber \\
%%&&
% - \sum\nolimits_{\nu \ne 1} \Div (P_\nu \tr^\perp_\nu \I ) -10 -2n_4 \nonumber\\
%&& +\,2\,\overline{\rm S}_{\,\mD_1,\mD_2,\mD_3,\mD_4}
%-\sum\nolimits_{\,i}\Div( \,P_i\tr_{\,i}^\bot \I + \,P_i^\bot\tr_{\,i}^\top \I ) = \lambda\,, \\
%%\end{eqnarray}
%%and
%%\begin{eqnarray}
%\label{ELSasakiD2}
%&& -\,9 \Div P_2^\perp (\I_{\xi_2} \xi_2) %+4 \< \I_{\xi_3}\xi_2 + \I_{\xi_2}\xi_3, \xi_1\> -4 \< \I_{\xi_2}\xi_1 + \I_{\xi_1}\xi_2, \xi_3\>\nonumber \\
%%&&
% - \sum\nolimits_{\nu \ne 2} \Div (P_\nu \tr^\perp_\nu \I ) -10 -2n_4 \nonumber\\
%&& +\,2\,\overline{\rm S}_{\,\mD_1,\mD_2,\mD_3,\mD_4} -\sum\nolimits_{\,i}\Div( \,P_i\tr_{\,i}^\bot \I + \,P_i^\bot\tr_{\,i}^\top \I ) = \lambda\,,\\
%\label{ELSasakiD3}
%&& -\,9 \Div (P_3^\perp\I_{\xi_3} \xi_3) % +4 \< \I_{\xi_3}\xi_1 + \I_{\xi_1}\xi_3, \xi_2\> -4 \< \I_{\xi_3}\xi_2 + \I_{\xi_2}\xi_3, \xi_1\>\nonumber \\
%%&&
% - \sum\nolimits_{\nu \ne 3} \Div ( P_\nu \tr^\perp_\nu \I ) -10 -2n_4 \nonumber\\
%&& +\,2\,\overline{\rm S}_{\,\mD_1,\mD_2,\mD_3,\mD_4}
%-\sum\nolimits_{\,i}\Div( \,P_i\tr_{\,i}^\bot \I + \,P_i^\bot\tr_{\,i}^\top \I ) = \lambda\,.
\end{eqnarray}
\end{subequations}
\end{Theorem}

\begin{proof}
%From \eqref{E-prop-X} we get
%\[
%2\,\overline{\rm S}_{\,\mD_1,\mD_2,\mD_3,\mD_4} - \Div X = 2\,\overline{\rm S}_{\,\mD_1,\mD_2,\mD_3,\mD_4}
%-\sum\nolimits_{\,i}\Div( \,P_i\tr_{\,i}^\bot \I + \,P_i^\bot\tr_{\,i}^\top \I \,)\,.
%\]
%Thus,
The Euler-Lagrange equation \eqref{ELSasakiD4} follows from Lemma~\ref{lemmadeltaQbarterms}, \eqref{deltaQQbarSasaki} and Theorem \ref{T-main01}.
To prove \eqref{ELSasakiDa} for $a=1$, we obtain from \eqref{deltaQQbarSasaki} evaluated on $(\xi_1, \xi_1)$ and Lemma \ref{lemmadeltaQbarterms} the following Euler-Lagrange equation:
\begin{eqnarray}\label{ELD1a}
&& \sum\nolimits_{\nu \ne 1} \< \phi_\nu (\xi_1,\xi_1), \frac{1}{2} \tr_\nu^\top \I\> - 2(\Div\phi_\nu^\top)(\xi_1,\xi_1) - \Div (P_\nu \tr^\perp_\nu \I )
\nonumber \\
&& -\,10 -2n_4 +\<{\tilde\phi}_1 (\xi_1,\xi_1), \frac{1}{2}\,\tr_1^\perp \I\> - 2(\Div{\tilde \phi}_1^\perp )(\xi_1,\xi_1) - \Div (P_1^\perp \tr^\top_1 \I ) \nonumber\\
&& +\,2\,\overline{\rm S}_{\,\mD_1,\mD_2,\mD_3,\mD_4}
-\sum\nolimits_{\,i}\Div( \,P_i\tr_{\,i}^\bot \I + \,P_i^\bot\tr_{\,i}^\top \I ) = \lambda\,.
\end{eqnarray}
For all $\nu \ne 1$ we get $\phi_\nu (\xi_1,\xi_1) = 2\,\I_{\xi_1} \xi_1 = {\tilde \phi}_1 (\xi_1,\xi_1)$. Since $\I$ is a contorsion tensor of a metric connection and $\mD_1$ is one-dimensional, we obtain $\I_{\xi_1} \xi_1 = P_1^\perp (\I_{\xi_1} \xi_1 )$.
On the other hand, from \eqref{ELconnectionNewI7} we get $P_1^\perp(\tr^\perp_1 \I ) =0$; hence,
\[
\sum\nolimits_{\nu \ne 1}\<\phi_\nu, \frac{1}{2}\,\tr_\nu^\top \I\> +\<{\tilde\phi}_1, \frac{1}{2}\,\tr_1^\perp\I\> = 2\<\I_{\xi_1}\xi_1, \tr^\perp_1\I\>=0\,.
\]
Hence, \eqref{ELD1a} can be simplified to
\begin{eqnarray} \label{ELD1b}
&& -\sum\nolimits_{\nu \ne 1} \big( 2\,(\Div \phi_\nu^\top) (\xi_1,\xi_1) + \Div (P_\nu \tr^\perp_\nu \I ) \big)
%\nonumber \\ &&
 -10 -2n_4 - 2 (\Div {\tilde \phi}_1^\perp)(\xi_1,\xi_1)  \nonumber\\
&& -\,\Div (P_1^\perp \tr^\top_1 \I )
 +\,2\,\overline{\rm S}_{\,\mD_1,\mD_2,\mD_3,\mD_4} -\sum\nolimits_{\,i}\Div( \,P_i\tr_{\,i}^\bot \I + \,P_i^\bot\tr_{\,i}^\top \I ) = \lambda\,.
\end{eqnarray}
Using
\begin{eqnarray*}
\Div \phi_2^\top (\xi_1, \xi_1) &=& 2\Div (P_2^\top\I_{\xi_1}\xi_1 ) - 2 \< \phi_2 (\nabla_{\xi_2} \xi_1, \xi_1 ), \xi_2\>
- 2\sum\nolimits_{\,j} \< \phi_2^\top (\nabla_{ E_{j}}\xi_1, \xi_1 ), E_{j}\> \\
%&=& 2\Div(P_2\I_{\xi_1}\xi_1 ) + 2 \< \phi_2 (\xi_3, \xi_1 ), \xi_2\> \\
&=& 2\Div(P_2\I_{\xi_1}\xi_1 ) + 2 \< \I_{\xi_3}\xi_1 + \I_{\xi_1}\xi_3, \xi_2\>,\\
%%%%%%%%%
\Div\phi_3^\top (\xi_1, \xi_1)
%&=& 2\Div (P_3\I_{\xi_1}\xi_1 ) - 2 \< \phi_3 (\nabla_{\xi_3} \xi_1, \xi_1 ), \xi_3\>
%- 2\sum\nolimits_{\,j} \< \phi_3^\top (\nabla_{ E_{j}}\xi_1, \xi_1 ), E_{j}\> \\
%&=& 2\Div(P_3\I_{\xi_1}\xi_1 ) - 2 \< \phi_3 (\xi_2, \xi_1 ), \xi_3\> \\
&=& 2\Div(P_3\I_{\xi_1}\xi_1 ) - 2 \< \I_{\xi_2}\xi_1 + \I_{\xi_1}\xi_2, \xi_3\>,\\
%%%%%%%%%
\Div\phi_4^\top (\xi_1, \xi_1)
%&=& 2\Div (P_4\I_{\xi_1}\xi_1 ) - 2 \sum\nolimits_{\,j} \< \phi_4 (\nabla_{ E_{j}}\xi_1, \xi_1 ), E_{j}\> \\
&=& 2\Div( P_4\I_{\xi_1}\xi_1 ) + 2 \sum\nolimits_{\,j} \< \phi_4 (\tT_{\xi_1} E_{j}, \xi_1 ), E_{j}\> = 2 \Div( P_4\I_{\xi_1}\xi_1 )\,,
\end{eqnarray*}
(because $\tT_{\xi_1} E_{j} \in \mD_4$), \eqref{assumptionXYZ} and
%\begin{eqnarray*}
 $\Div {\tilde \phi}_1^\perp (\xi_1, \xi_1) = 2 \Div( P_1^\perp\I_{\xi_1} \xi_1 )$,
%\end{eqnarray*}
since $\nabla_Z\,\xi_1 \perp \xi_1$ for all $Z \in TM$,
we get \eqref{ELSasakiDa} for $a=1$ from \eqref{ELD1b}.
Similarly,
%using
%\begin{eqnarray*}
%\phi_1(\xi_2, \xi_2) &=& \phi_3(\xi_2, \xi_2) = \phi_4(\xi_2, \xi_2) = {\tilde\phi}_2 (\xi_2, \xi_2) = (\I + \I^\wedge)_{\xi_2} \xi_2 = 2 \I_{\xi_2} \xi_2,\\
%\Div \phi_1^\top (\xi_2, \xi_2) &=& 2\Div( P_1^\top\I_{\xi_2}\xi_2 ) - 2 \< \phi_1 (\nabla_{\xi_1} \xi_2, \xi_2 ), \xi_1\>
%- 2 \sum\nolimits_{\,j} \< \phi_2^\top (\nabla_{ E_{j}}\xi_2, \xi_2 ), E_{j}\> \\
%&=& 2 \Div( P_1\I_{\xi_2}\xi_2 ) - 2 \< \I_{\xi_3}\xi_2 + \I_{\xi_2}\xi_3, \xi_1\>, \\
%\Div \phi_3^\top (\xi_2, \xi_2) &=& 2\Div( P_3\I_{\xi_2}\xi_2 ) - 2 \< \phi_3 (\nabla_{\xi_3} \xi_2, \xi_2 ), \xi_3\>
%- 2 \sum\nolimits_{\,j} \< \phi_3^\top (\nabla_{ E_{j}}\xi_2, \xi_2 ), E_{j}\> \\
%&=& 2 \Div( P_3\I_{\xi_2}\xi_2 ) + 2 \< \I_{\xi_2}\xi_1 + \I_{\xi_1}\xi_2, \xi_3\>, \\
%\phi_1(\xi_3, \xi_3) &=& \phi_2(\xi_3, \xi_3) = \phi_4 (\xi_3, \xi_3) = {\tilde\phi}_3 (\xi_3, \xi_3) = (\I + \I^\wedge)_{\xi_3}\xi_3 = 2 \I_{\xi_3} \xi_3,\\
%\Div \phi_1^\top (\xi_3, \xi_3) &=& 2\Div( P_1^\top\I_{\xi_3}\xi_3 ) - 2 \< \phi_1 (\nabla_{\xi_1} \xi_3, \xi_3 ), \xi_1\>
%- 2 \sum\nolimits_{\,j} \< \phi_2^\top (\nabla_{ E_{j}}\xi_3, \xi_3 ), E_{j}\> \\
%&=& 2 \Div( P_1\I_{\xi_3}\xi_3 ) + 2 \< \I_{\xi_3}\xi_2 + \I_{\xi_2}\xi_3, \xi_1\>,\\
%\Div \phi_2^\top (\xi_3, \xi_3) &=& 2\Div( P_2\I_{\xi_3}\xi_3 ) - 2 \< \phi_2 (\nabla_{\xi_2} \xi_3, \xi_3 ), \xi_2\>
%- 2 \sum\nolimits_{\,j} \< \phi_2^\top (\nabla_{ E_{j}}\xi_3, \xi_3 ), E_{j}\> \\
%&=& 2 \Div( P_2\I_{\xi_3}\xi_3 ) - 2 \< \I_{\xi_3}\xi_1 + \I_{\xi_1}\xi_3, \xi_2\>,
%\end{eqnarray*}
we get \eqref{ELSasakiDa} for $a=2,3$. % and~\eqref{ELSasakiD3}.
% and \eqref{ELSasakiD3}.
\end{proof}

\begin{Corollary}\label{Cor-3}
Let $(M,g)$ be a 3-Sasaki manifold, and a pair $(g,\I)$ be critical for the action \eqref{Eq-Smix-g}
with respect to adapted variations of $g$ preserving the volume of $(M,g)$ and all variations of \,$\I$,
and let $\I$ be such that \eqref{assumptionXYZ} holds.
Then the following equations %\eqref{E-delta-g-J} <-NO, it's not the Euler-Lagrange equations, but their trace and average.
are valid for some $\lambda\in\RR$:
\begin{subequations}
\begin{eqnarray} \label{trELSasakiD4}
&& - \frac{4+n_4}{n_4} \Div \big( P^\perp_4 \tr^\top_4 \I + \sum\nolimits_{\nu \ne 4} P_\nu \tr^\perp_\nu \I \big) = \lambda-30 \nonumber\\
&&\qquad -\,2\,\overline{\rm S}_{\,\mD_1,\mD_2,\mD_3,\mD_4} +\sum\nolimits_{\,\nu}\Div(\,P_\nu\tr_{\,\nu}^\bot \I + \,P_\nu^\bot\tr_{\,\nu}^\top \I )\,, \\
\label{sumELSasakiD123}
&& -\,9\sum\nolimits_{a=1}^3 \Div(P_a^\perp \I_{\xi_a} \xi_a + \sum\nolimits_{\nu \ne a} P_\nu \tr^\perp_\nu \I ) = \lambda +10 + 2\,n_4 \nonumber\\
&&\qquad -\,2\,\overline{\rm S}_{\,\mD_1,\mD_2,\mD_3,\mD_4} + \sum\nolimits_{\,\nu}\Div( \,P_\nu\tr_{\,\nu}^\bot \I + \,P_\nu^\bot\tr_{\,\nu}^\top \I )\, .
\end{eqnarray}
\end{subequations}
%In particular, a manifold $M$ does not admit complete 3-Sasaki structures and metric connections satisfying \eqref{assumptionXYZ} critical for \eqref{Eq-Smix-g}
In particular, no manifold admits complete 3-Sasaki structures and metric connections satisfying \eqref{assumptionXYZ} critical for \eqref{Eq-Smix-g}
with respect to adapted variations of $g$ and all variations of \,$\I$.
\end{Corollary}

\begin{proof}
%Similarly as in \cite{rz-3},
Using orthonormal basis $\{ E_j \}$ of $\mD_4$, where $E_{j+n_4/2} = \tT_{\xi_a} E_{j}$ \cite{blair}, we find
\[
\sum\nolimits_j {\rm Sym} ( \< {\tilde \phi}_4 (\tT_{\xi_a} E_{j} , E_{j} ) , \xi_a\> ) = 0.
\]
Hence, taking trace of Euler-Lagrange equations \eqref{ELSasakiD4} we get
%\begin{eqnarray*}
%&& \sum\nolimits_{\nu \ne 4} \big(\< \tr^\perp_\nu \I, \tr_\nu^\top \I\> - 4 \Div ( P_\nu \tr^\perp_\nu \I) + 72\,n_4
%+ 7 {\rm Sym}(\< X, \I_{\xi_2} \xi_1 - \I_{\xi_1} \xi_2\> \< Y, \xi_3\> ) \\
%%&& + 7 {\rm Sym}(\< X, \I_{\xi_1} \xi_3 - \I_{\xi_3} \xi_1\> \< Y, \xi_2\> )
%%+ 7 {\rm Sym}(\< X, \I_{\xi_3} \xi_2 - \I_{\xi_2} \xi_3\> \< Y, \xi_1\> ) =0 \\
% - n_4 \Div (P_\nu \tr^\perp_\nu \I )\big) \\
%%%%%
%dual
%%%%%
%&& +\, \< \tr^\top_4 \I, \tr_4^\perp \I\> - 4 \Div( P^\perp_4 \tr^\top_4 \I ) - n_4 \Div (P^\perp_4 \tr^\top_4 \I ) \nonumber\\
%&& +\,\big(2\,\overline{\rm S}_{\,\mD_1,\mD_2,\mD_3,\mD_4}
% -\sum\nolimits_{\,\nu}\Div( \,P_\nu\tr_{\,\nu}^\bot \I + \,P_\nu^\bot\tr_{\,\nu}^\top \I ) \big) n_4 = \lambda\,n_4,
%\end{eqnarray*}
%which simplifies to the following:
\begin{eqnarray*}
&& \sum\nolimits_{\nu \ne 4} \big(\< \tr^\perp_\nu \I, \tr_\nu^\top \I\> - (4+n_4) \Div (P_\nu \tr^\perp_\nu \I)\big)
%\\
%%%%%
%dual
%%%%%
%&&
 {+}\< \tr^\top_4 \I, \tr_4^\perp \I\> {-} (4+n_4) \Div( P^\perp_4 \tr^\top_4 \I ) \nonumber\\
&& = \big(\lambda-30 -2\,\overline{\rm S}_{\,\mD_1,\mD_2,\mD_3,\mD_4}
+\sum\nolimits_{\,\nu}\Div( \,P_\nu\tr_{\,\nu}^\bot \I + \,P_\nu^\bot\tr_{\,\nu}^\top \I ) \big)\,n_4\,.
\end{eqnarray*}
Since $\I$ corresponds to a metric connection, we get $P_\mu\tr_\nu^\top \I
%^\top
=0
\ (\nu\le 3)$ (as each of these $\mD_\nu$ is one-dimensional and $\< \xi_\nu, \I_{\xi_\nu} \xi_\nu\> =0$), and by \eqref{ELconnectionNewI7} we get
$\< \tr^\perp_\nu \I, \tr_\nu^\top \I\> =0\ (\nu\le 3)$. Hence,
\begin{eqnarray*}
&& -\,(4+n_4) \sum\nolimits_{\nu \ne 4} \Div (P_\nu \tr^\perp_\nu \I)
%%%%%
%dual
%%%%%
+\< \tr^\top_4 \I, \tr_4^\perp \I\> - (4+n_4) \Div( P^\perp_4 \tr^\top_4 \I ) \nonumber\\
&& = \big(\lambda-30 - 2\,\overline{\rm S}_{\,\mD_1,\mD_2,\mD_3,\mD_4}
+ \sum\nolimits_{\,\nu}\Div( \,P_\nu\tr_{\,\nu}^\bot \I + \,P_\nu^\bot\tr_{\,\nu}^\top \I ) \big)\,n_4\,.
\end{eqnarray*}
From \eqref{ELconnectionNewI7}, we get $P_4^\perp (\tr_4^\perp \I )=0$, thus \eqref{ELconnectionNew2} yields
%\[
 $\< \tr^\top_4 \I,\, \tr_4^\perp \I\> = \< \tr^\top_4 \I,\, P_4^\top\tr_4^\perp \I\> = 0$,
%\]
and \eqref{trELSasakiD4} follows. We obtain \eqref{sumELSasakiD123} as the %sum
average of %(\ref{ELSasakiD1}-d).
\eqref{ELSasakiDa} for $a=1,2,3$.
The last claim follows from the result in \cite{bg}, where completeness of metric is proved to imply compactness of the manifold, %and
the divergence theorem for closed manifolds and comparing the constants on right-hand sides of \eqref{trELSasakiD4} and \eqref{sumELSasakiD123}, which cannot be both zero for any dimension $n_4$.
\end{proof}

\begin{Example}\rm
Let $\I$ be such that \eqref{assumptionXYZ} and all equations of Theorem~\ref{corI} hold, and
\[
 \I_X X=0,\quad
 \I_X Y + \I_Y X = P_4 (\I_{P_4 X} P_4 Y + \I_{P_4 Y} P_4 X),\quad X,Y \in\mathfrak{X}_M\,.
\]
By \eqref{E-Q1Q2-gen} we get $\overline{\rm S}_{\,\mD_1,\mD_2,\mD_3,\mD_4} = {\rm const}$.
Thus, all equations %(\ref{ELSasakiD1}-d) and also
\eqref{ELSasakiD4} and \eqref{ELSasakiDa} for $a=1,2,3$ are valid, but with different $\lambda$'s.
Hence, for each $\mu=1, \ldots ,4$ such pair $(g,\I)$ on a 3-Sasaki manifold $(M,g)$ is critical for $\mD_\mu$-variations preserving the volume of $\Omega$, but not for all adapted variations preserving the volume of $\Omega$.
\end{Example}

\begin{Remark}%\label{Remark}
\rm
The Euler-Lagrange equation on $\mD_4$ is incompatible with those for $\mD_a\ (a=1,2,3)$ on compact manifolds,
see Corollary~\ref{Cor-3}, also when considering \eqref{Eq-Smix-g0}, as the functional of an adapted metric $g$. %(with $\I=0$).
Indeed, by Lemma~\ref{lemmadeltaQbarterms} and \eqref{deltaQmetric} we get a contradiction in \cite[(3.7)]{r-EH-k}, which, adapted to our notation, reads:
\[
- ({\rm S}_{\,\mD_1, \mD_2,\mD_3,\mD_4} +\lambda) g_4 = \delta Q_4  = -12 g_4,\quad
-{\rm S}_{\,\mD_1, \mD_2,\mD_3,\mD_4}g_4 + \lambda=\delta Q_a = 2n_4 - 6 \quad
(a=1,2,3)\,.
%
%\lambda g_4 -\delta Q_4 = -12 g_4,\quad
%\lambda=\delta Q_a = 2n_4 - 6\quad
%(a=1,2,3)\,.
\]
However, we can consider the following weighted action:
\begin{equation}\label{Eq-Smix-gc}
J_{\mD,c} : g \mapsto\int_{M} \big(\sum\nolimits_{\,a=1}^3 {\rm S}_{\,\mD_a, \mD_a^\perp } + c\,{\rm S}_{\,\mD_4, \mD_4^\perp} \big)\,{\rm d}\vol_g \,.
\end{equation}
Then, a 3-Sasaki metric $g$ on $M^{3+n_4}$ is critical for \eqref{Eq-Smix-gc} %with $\I=0$
if and only if $c=- \frac{ n_4} {n_4 + 6}$.
%\section{Addition to 3-Sasaki}
%Now this result is in Remark \ref{remarkc}, here are just the computations (which we probably don't need to write), so this section can be removed.
%%
%Indeed, %for Riemannian case
%we get
%%\[
%$\delta Q_4 = -6g_4 -6g_4 = \lambda\, g_4$,
%%\]
%so
%%\[
%$\lambda = -12$.
%%\]
%%On the other hand,
%For $a=1,2,3$ we get
%%\[
%$\delta Q_a = -4 +n_4 -4 +2+n_4 = \lambda$,
%%\]
%so
%%\[
%$\lambda = 2n_4 - 6$,
%%\]
%these are compatible if and only if $-12=2n_4 - 6$, i.e., $n_4 = -3$, which is impossible.
%On the other hand,
%%
Indeed, for $\tilde{\rm S} = \sum_{a=1}^3 {\rm S}_{\,\mD_a,\mD^\bot_a} + c\,{\rm S}_{\,\mD_4,\mD^\bot_4}$ with $c\in\RR$,
we get
%\[
$\delta Q_4 = -6g_4 -6 c g_4 = \lambda\,g_4$,
%\]
so $\lambda = -6(1+c)$, and for $a=1,2,3$ we get
%\[
$\delta Q_a = -4 +c n_4 -4 +2+n_4 = \lambda$,
%\]
so
%\[
$\lambda = n_4 (c+1) - 6$.
%\]
%All together:
%%\[
%$-6-6c = -6 +c n_4 + n_4$,
%that is
%%\quad \Rightarrow\quad
%$c=- \frac{ n_4} {n_4 + 6}$.
%\]
%\[
%(-6-n_4)c =n_4,
%\]
%\[
% c=- \frac{ n_4} {n_4 + 6}.
%\]
%or
%\[
%n_4 = \frac{-6c }{c+1}.
%\]
\end{Remark}

\begin{Corollary}
Let $\mD_\mu\ (\mu=1,2,3)$ be distributions determined by orthonormal vector fields $\xi_1, \xi_2, \xi_3$
on a unit sphere $(S^3,g)$ with the metric $g$ induced from the Euclidean space $\mathbb{R}^4$, such that $\nabla_{\xi_a} \xi_b = \xi_c$ for even permutation of $a,b,c$. Then $\nabla$ is the only metric-compatible connection,
whose contorsion tensor is critical for the action \eqref{Eq-Smix-g} with fixed $g$.
%\eqref{actionI}.
\end{Corollary}

\begin{proof}
From \eqref{3Sasaki3different} and the fact that $\I$ corresponds to metric connection, we find the following:
\begin{eqnarray*}
\< \I_{\xi_3} \xi_1, \xi_2 \> = \< \I_{\xi_1} \xi_3, \xi_2 \> = - \< \I_{\xi_1} \xi_2, \xi_3 \>\,,\\
\< \I_{\xi_2} \xi_3, \xi_1 \> = \< \I_{\xi_3} \xi_2, \xi_1 \> = - \< \I_{\xi_3} \xi_1, \xi_2 \> = \< \I_{\xi_1} \xi_2, \xi_3 \>\,,\\
\< \I_{\xi_1} \xi_2, \xi_3 \> = \< \I_{\xi_2} \xi_1, \xi_3 \> = - \< \I_{\xi_2} \xi_3, \xi_1 \> = - \< \I_{\xi_1} \xi_2, \xi_3 \> = 0\,.
\end{eqnarray*}
Hence, $\< \I_{\xi_a} \xi_b, \xi_c \> = 0$ for $a \ne b \ne c \ne a$, $a,b,c \in \{1,2,3\}$.
Also, for a $3$-Sasaki structure we get $H_a =0$ for $a=1,2,3$; using $\I = -\I^*$ and \eqref{ELconnectionNewI7}, see \eqref{E-delta-I-J}, we obtain
\[
0 = P_3 \I_{\xi_2} \xi_2 + P_2 \I_{\xi_3} \xi_3 = P_3 \I_{\xi_1} \xi_1 + P_1 \I_{\xi_2} \xi_2
= P_2 \I_{\xi_1} \xi_1 + P_1 \I_{\xi_2} \xi_2\, .
%= P_1 \I_{\xi_2} \xi_2 + P_2 \I_{\xi_3} \xi_3
\]
It follows that $\I_{\xi_1 }\xi_1 = -P_1(\I_{\xi_2}\xi_2 + \I_{\xi_3}\xi_3) = 0$, because by $\I = -\I^*$ we have $P_1\I_{\xi_1}\xi_1 = 0$.
Similarly, $\I_{\xi_2}\xi_2 = \I_{\xi_3} \xi_3 =0$. From $\I = -\I^*$ we get $\< \I_{\xi_a} \xi_b, \xi_b\> = 0$ for $a, b\in\{1,2,3\}$,
thus all components of critical tensor $\I$ vanish.
\end{proof}

\subsection*{Acknowledgements}

The second author was supported by grant Miniatura [2021/05/X/ST1/00359] of  Narodowe Centrum Nauki, Poland.

%\LastPageEnding

\end{document}